\documentclass[a4paper,11pt,reqno]{amsart}

\usepackage{amssymb}
\usepackage[all]{xy}

\usepackage{hyperref}
\usepackage{xcolor}
\usepackage{enumitem}
\usepackage{subcaption}
\usepackage[final]{graphicx}
\usepackage{float}
\usepackage{epic}
\usepackage{setspace}

\usepackage[all]{xy}
\usepackage{color}

\usepackage{verbatim}

\usepackage{tikz}
\usetikzlibrary[topaths]

\newcount\mycount

\numberwithin{equation}{section}

\numberwithin{equation}{section}
\newtheorem{introtheorem}{Theorem}

\newtheorem{introcorollary}[introtheorem]{Corollary}

\newtheorem{theorem}{Theorem}[section]
\newtheorem{lemma}[theorem]{Lemma}
\newtheorem{proposition}[theorem]{Proposition}
\newtheorem{corollary}[theorem]{Corollary}
\newtheorem*{theorem*}{Theorem}
\newtheorem*{corollary*}{Corollary}

\theoremstyle{definition}
\newtheorem{definition}[theorem]{Definition}

\theoremstyle{remark}
\newtheorem{example}[theorem]{Example}

\theoremstyle{remark}

\newtheorem{remark}[theorem]{Remark}

\begin{document}

\title{Universal C$^{\ast}$-algebras from graph products: structure and applications}

\author{Mario Klisse}

\address{KU Leuven, Department of Mathematics,
Celestijnenlaan 200B, 3001 Leuven, Belgium}

\email{mario.klisse@kuleuven.be}

\date{\today. \emph{MSC2010:}  46L05, 46L09. The author is supported by the FWO postdoctoral grant 1203924N of the Research Foundation Flanders.}

\maketitle
\begin{abstract}
In this article, we introduce and investigate a class of C$^{\ast}$-algebras generated by reduced graph products of C$^{\ast}$-algebras, augmented with families of projections naturally associated with words in right-angled Coxeter groups. These ambient C$^{\ast}$-algebras possess a rich and tractable combinatorial structure, which enables the deduction of a variety of structural properties. Among other results, we establish universal properties, characterize nuclearity and exactness in terms of the vertex algebras, and analyze the ideal structure. In the second part of the article, we leverage this framework to derive new insights into the structure of graph product C$^{\ast}$-algebras -- many of which are novel even in the case of free products.
\end{abstract}

\vspace{3mm}

\section*{Introduction}

\vspace{3mm}

After their introduction as an ingredient of Voiculescu's groundbreaking non-commutative probability theory in \cite{Voiculescu85} (see also \cite{Avitzour82}), free products have become a fundamental tool in the study of operator algebras. Voiculescu's construction can be viewed as the operator-algebraic analogue of free products of groups, with both frameworks exhibiting a structural compatibility. In the group-theoretic setting, free products are naturally generalized through Green's graph products of groups in \cite{Green90}, a construction which starts from a simplicial graph with a discrete group assigned to each vertex. The resulting group amalgamates its vertex groups, enforcing commutation relations that mirror the graph's adjacency structure. This framework interpolates between free and Cartesian products, encompasses important classes such as right-angled Artin and right-angled Coxeter groups, and preserves many group-theoretic properties.

Partially motivated by these developments, analogous constructions in the setting of operator algebras have been introduced and studied by M\l otkowski \cite{Mlotkowski04}, Speicher and Wysocza\'{n}ski \cite{SpeicherWysoczanski16}, and Caspers and Fima \cite{CaspersFima17}. In recent years, graph product structures have attracted growing interest, particularly in relation to free probability (\cite{Mlotkowski04, SpeicherWysoczanski16, CharlesworthCollins21, CSHJEN24-2}), Popa's deformation/rigidity theory (\cite{ChifanSantiagoSucpikarnon18, Caspers20, BorstCaspersWasilewski24, BorstCaspersChen24, CaspersChen25, DrimbeVaes25}), and approximation properties (\cite{CaspersFima17, Atkinson20, Borst24, BorstCaspersChen24}); see also \cite{Caspers16, CaspersKlisseLarsen21, CharlesworthJekel25}. Furthermore, von Neumann algebras arising from group-theoretic graph products have been studied intensively in \cite{BorstCaspers24, ChifanKunnawalkam24, ChifanDavisDrimbe25-1, ChifanDavisDrimbe25-2}. Comparatively little is known about their C$^{\ast}$-algebraic counterparts beyond the free product case.

Let $\Gamma$ be a finite, undirected, simplicial graph with vertex set $V\Gamma$ and edge set $E\Gamma$, and let $\mathbf{A}=(A_{v})_{v\in V\Gamma}$ denote a family of unital C$^{\ast}$-algebras, each equipped with a GNS-faithful state $\omega_{v}$. Following \cite{CaspersFima17}, the \emph{reduced graph product C$^{\ast}$-algebra} $\mathbf{A}_{\Gamma}:=\star_{v,\Gamma}(A_{v},\omega_{v})$ is a unital C$^{\ast}$-algebra with a canonical GNS-faithful state $\omega_{\Gamma}$. This algebra contains isomorphic copies of the vertex algebras $A_{v}$ such that $\omega_{\Gamma}$ restricts to the respective state $\omega_{v}$ and such that $[A_{v},A_{v^{\prime}}]=0$ whenever $(v,v')\in E\Gamma$. Moreover, reduced words in the vertex algebras satisfy a freeness condition with respect to $\omega_{\Gamma}$. In the extreme cases, $\mathbf{A}_{\Gamma}$ recovers Voiculescu's reduced free product when $\Gamma$ has no edges, and the tensor product $\bigotimes_{v\in V\Gamma}A_{v}$ when $\Gamma$ is complete.

Building on the author's previous work on right-angled Hecke C$^{\ast}$-algebras in \cite{Klisse23-1, Klisse23-2}, this article introduces and studies ambient C$^{\ast}$-algebras $\mathfrak{A}(\mathbf{A},\Gamma)$ generated by the reduced graph product $\mathbf{A}_{\Gamma}$ together with a family of projections $(Q_{v})_{v\in V\Gamma}$ associated with words in the right-angled Coxeter group $W_{\Gamma}:=\star_{v,\Gamma}\mathbb{Z}_{2}$. Since every graph product can be expressed as an amalgamated free product of suitable substructures, our construction can be viewed as a refined and well-behaved analog of Hasegawa's construction in \cite{Hasegawa19}. Furthermore, for two-dimensional vertex algebras, $\mathfrak{A}(\mathbf{A},\Gamma)$ identifies with the reduced crossed product associated with the canonical action of $W_{\Gamma}$ on its combinatorial boundary (see \cite{Lecureux10, CapraceLecureux11, Klisse23-1}).

The motivation for the present work is two-fold. On the one hand, the ambient algebras $\mathfrak{A}(\mathbf{A},\Gamma)$ contain the graph product $\mathbf{A}_\Gamma$ and exhibit a rich and tractable combinatorial structure, leading to a variety of interesting and useful properties. Specifically, for any finite, undirected, simplicial graph $\Gamma$ and unital C$^{\ast}$-algebras $\mathbf{A}=(A_{v})_{v\in V\Gamma}$ with GNS-faithful states $\omega_{v}$, the ambient algebra $\mathfrak{A}(\mathbf{A},\Gamma)$ admits a natural \emph{gauge action} $\mathbb{T}^{V\Gamma}\curvearrowright\mathfrak{A}(\mathbf{A},\Gamma)$, and a canonical expected C$^{\ast}$-subalgebra $\mathcal{D}(\mathbf{A},\Gamma)$ of ``diagonal'' operators. Furthermore, in Proposition \ref{DensityStatement} we describe a dense $\ast$-subalgebra of $\mathfrak{A}(\mathbf{A},\Gamma)$ consisting of linear combinations of products of \emph{creation}, \emph{diagonal} and \emph{annihilation} operators. These features allow one to formulate and prove a range of properties.

\begin{introtheorem}[{Theorem \ref{UniversalProperty} and Theorem \ref{UniversalProperty2}}] \label{Introduction1} Let $\Gamma$ be a finite, undirected, simplicial graph and let $\mathbf{A}:=(A_{v})_{v\in V\Gamma}$ be a collection of unital C$^{\ast}$-algebras, equipped with GNS-faithful states $(\omega_{v})_{v\in V\Gamma}$. Then, $\mathfrak{A}(\mathbf{A},\Gamma)$ satisfies the following universal properties: 
\begin{enumerate}
\item For $v_{0}\in V\Gamma$ define 
\[
\mathbf{A}_{1}:=(A_{v})_{v\in V\text{\emph{Star}}(v_{0})},\,\mathbf{A}_{2}:=(A_{v})_{v\in V(\Gamma\setminus\{v_{0}\})},\,\mathbf{B}:=(A_{v})_{v\in V\text{\emph{Link}}(v_{0})}
\]
and consider $\mathfrak{A}_{1}:=\mathfrak{A}(\mathbf{A}_{1},V\text{\emph{Star}}(v_{0}))$, $\mathfrak{A}_{2}:=\mathfrak{A}(\mathbf{A}_{2},V(\Gamma\setminus\{v_{0}\}))$ and $B:=\mathfrak{A}(\mathbf{B},\text{\emph{Link}}(v_{0}))$. Then every unital C$^{\ast}$-algebra $\overline{\mathfrak{A}}$ generated by the images of unital $\ast$-homomorphisms $\kappa_{1}:\mathfrak{A}_{1}\rightarrow\overline{\mathfrak{A}}$, $\kappa_{2}:\mathfrak{A}_{2}\rightarrow\overline{\mathfrak{A}}$ satisfying $\kappa_{1}|_{B}=\kappa_{2}|_{B}$ and $\kappa_{1}(Q_{v_{0}})\kappa_{2}(Q_{v})=0$ for all $v\in V(\Gamma \setminus\text{\emph{Star}}(v_{0}))$ admits a surjective $\ast$-homomorphism $\phi:\mathfrak{A}(\mathbf{A},\Gamma)\twoheadrightarrow\overline{\mathfrak{A}}$ with $\phi|_{\mathfrak{A}_{1}}=\kappa_{1}$ and $\phi|_{\mathfrak{A}_{2}}=\kappa_{2}$. 
\item If the vertex C$^{\ast}$-algebras $(A_{v})_{v\in V\Gamma}$ are nuclear, every unital C$^{\ast}$-algebra $\overline{\mathfrak{A}}$ generated by the images of unital $\ast$-homomorphisms $\kappa_{v}:\mathfrak{A}(A_{v},\{v\})\rightarrow\overline{\mathfrak{A}}$, $v\in V\Gamma$ satisfying 
\[
[\kappa_{v}(x),\kappa_{v^{\prime}}(y)]=0\text{ for all }x\in\mathfrak{A}(A_{v},\{v\}),y\in\mathfrak{A}(A_{v^{\prime}},\{v^{\prime}\})
\]
with $(v,v^{\prime})\in E\Gamma$ and 
\[
\kappa_{v}(Q_{v})\kappa_{v^{\prime}}(Q_{v^{\prime}})=0\text{ for all }(v,v^{\prime})\in E\Gamma^{c}
\]
admits a surjective $\ast$-homomorphism $\phi:\mathfrak{A}(\mathbf{A},\Gamma)\twoheadrightarrow\overline{\mathfrak{A}}$ with $\phi(x)=\kappa_{v}(x)$ for $v\in V\Gamma$, $x\in\mathfrak{A}(A_{v},\{v\})$. 
\end{enumerate}
\end{introtheorem}

In the special case of free products, analogous results were previously obtained by Hasegawa via an identification with Cuntz--Pimsner algebras \cite[Corollary 4.2.3]{Hasegawa19}. In the more intricate graph product setting, that approach breaks down; however, Theorem \ref{Introduction1} can still be established through a novel approach combining a refinement of methods from Katsura's theory in the context of C$^\ast$-algebras associated with C$^{\ast}$-correspondences \cite{Katsura04, Katsura06} with an inductive argument along the graph structure.

More concretely, given an enumeration $(s_1,\dots,s_L)$ of the vertex set $V\Gamma$, in Subsection~\ref{subsec:Universality} we associate to each multi-index $(n_1,\ldots,n_r) \in \mathbb{N}^r$ with $r \leq L$ a closed linear subspace $B_{n_1,\ldots,n_r} \subseteq \mathfrak{A}(\mathbf{A},\Gamma)$. These subspaces are spanned by linear combinations of products of creation, diagonal, and annihilation operators corresponding to the vertices of $\Gamma$, with multiplicities prescribed by the tuple $(n_1,\ldots,n_r)$. As shown in Lemma~\ref{FixedpointIdentification}, suitable unions of sums of these subspaces identify with fixed-point algebras of appropriate restrictions of the gauge action $\mathbb{T}^{V\Gamma} \curvearrowright \mathfrak{A}(\mathbf{A},\Gamma)$. This observation, combined with orthogonality considerations and an induction over $r$ (starting from $r = L$), leads to the proof of the first part of Theorem~\ref{UniversalProperty}. The second part then follows by iterating the first.

We believe that this approach is of independent interest and may be adapted to a broader class of constructions beyond the present context. In addition, the techniques yield the following criterion for exactness and nuclearity.

\begin{introtheorem}[{Theorem \ref{NuclearityExactness}}] \label{Introduction3} Let $\Gamma$ be a finite, undirected, simplicial graph and let $\mathbf{A}:=(A_{v})_{v\in V\Gamma}$ be a collection of unital C$^{\ast}$-algebras, equipped with GNS-faithful states $(\omega_{v})_{v\in V\Gamma}$. Then the following two statements hold: 
\begin{enumerate}
\item The C$^{\ast}$-algebra $\mathfrak{A}(\mathbf{A},\Gamma)$ is nuclear if and only if $A_{v}$ is nuclear for every $v\in V\Gamma$. 
\item The C$^{\ast}$-algebra $\mathfrak{A}(\mathbf{A},\Gamma)$ is exact if and only if $A_{v}$ is exact for every $v\in V\Gamma$. 
\end{enumerate}
\end{introtheorem}

Beyond these universal and approximation properties, we construct a natural closed two-sided ideal $\mathfrak{I}(\mathbf{A},\Gamma)$ in $\mathfrak{A}(\mathbf{A},\Gamma)$ consisting of operators ``vanishing at infinity'' with respect to the grading of the underlying graph product Hilbert space (see Subsection \ref{subsec:Ideal}). While $\mathfrak{I}(\mathbf{A},\Gamma)$ coincides with the compact operators in the case of finite-dimensional vertex algebras (see Proposition \ref{FiniteDimensionalIdeal}), in general it is a genuinely new ingredient that constitutes a useful tool in the study of the ideal structure of the underlying graph product C$^\ast$-algebras (see Theorem \ref{SimplicityCriterion1}); its relevance is far from obvious on the graph product level alone.

Under suitable assumptions, the ideal turns out to be maximal. The proof of this is loosely inspired by Archbold–Spielberg’s work on simplicity of crossed products by discrete groups acting on Abelian C$^{\ast}$‑algebras \cite{ArchboldSpielberg94}, but significant technical adaptations are required to accommodate the richer combinatorial structure present here.

\begin{introtheorem}[{Theorem \ref{MaximalityTheorem}}] \label{Introduction4} Let $\Gamma$ be a finite, undirected, simplicial graph with $\#V\Gamma\geq3$ whose complement $\Gamma^{c}$ is connected and let $\mathbf{A}:=(A_{v})_{v\in V\Gamma}$ be a collection of unital C$^{\ast}$-algebras, equipped with GNS-faithful states $(\omega_{v})_{v\in V\Gamma}$. Then $\mathfrak{I}(\mathbf{A},\Gamma)\triangleleft\mathfrak{A}(\mathbf{A},\Gamma)$ is a maximal ideal. \end{introtheorem}

On the other hand, our construction serves as a framework for analyzing structural properties of graph product C$^{\ast}$-algebras. For example, Theorem \ref{Introduction3} provides new, conceptual proofs of approximation properties, where the second statement strengthens \cite[Theorem H]{BorstCaspersChen24}.

\begin{introcorollary}[{Corollary \ref{GraphProductExactness} and Corollary \ref{GraphProductNuclearity}}] \label{Introduction5} Let $\Gamma$ be a finite, undirected, simplicial graph and let $\mathbf{A}:=(A_{v})_{v\in V\Gamma}$ be a collection of unital C$^{\ast}$-algebras, each equipped with a GNS-faithful state $\omega_{v}$. Then the following statements hold: 
\begin{enumerate}
\item The graph product C$^{\ast}$-algebra $\mathbf{A}_{\Gamma}$ is exact if and only if the vertex algebras $(A_{v})_{v\in V\Gamma}$ are all exact. 
\item For $v\in V\Gamma$ denote the GNS-Hilbert space with respect to $\omega_{v}$ by $\mathcal{H}_v$ and assume that $A_{v}\subseteq\mathcal{B}(\mathcal{H}_v)$ contains the compact operators. Then the graph product C$^{\ast}$-algebra $\mathbf{A}_{\Gamma}$ is nuclear if and only if the vertex algebras $(A_{v})_{v\in V\Gamma}$ are all nuclear. 
\end{enumerate}
\end{introcorollary}

Moreover, Theorem \ref{Introduction4} enables us to formulate new criteria for the simplicity and trace uniqueness of graph products -- results that, to our knowledge, are novel even in the free product setting and represent a substantial strengthening of the author’s earlier work on the simplicity of right‑angled Hecke C$^\ast$‑algebras. In the context of von Neumann algebras, similar results have been obtained in \cite{CSHJEN24-1} (see also \cite{Garncarek16, RaumSkalski23}) by entirely different methods.

As discussed in Subsection \ref{GraphSimplicity}, the assumption that the complement $\Gamma^c$ is connected is largely cosmetic and can be relaxed under suitable technical conditions.

\begin{introtheorem}[{Theorem \ref{SimplicityCriterion1}, Corollary \ref{SimplicityCriterion2}, Corollary \ref{SimplicityCriterion3}, and Theorem \ref{TraceUniqueness}}] \label{Introduction6} Let $\Gamma$ be a finite, undirected, simplicial graph with $\# \Gamma \geq 3$, and let $\mathbf{A}:=(A_{v})_{v\in V\Gamma}$ be a collection of unital C$^{\ast}$-algebras, each equipped with a GNS-faithful state $\omega_{v}$. Assume that the complement $\Gamma^{c}$ is connected, that every vertex $v\in V\Gamma$ admits elements $a_{v}\in\ker(\omega_{v})$, $q_{v}>0$ with $a_{v} a_{v}^{\ast} \geq q_{v}\omega_{v}(a_{v}^{\ast}a_{v})1>0$, and that the multi-parameter $(q_{v})_{v\in V\Gamma}$ is not contained in the closure of the region of convergence of the multivariate growth series $\sum_{\mathbf{w}\in W_{\Gamma}}z_{\mathbf{w}}$. Then the following statements hold: 
\begin{enumerate}
\item The graph product C$^{\ast}$-algebra $\mathbf{A}_{\Gamma}$ is simple if and only if $\mathbf{A}_{\Gamma}\cap\mathfrak{I}(\mathbf{A},\Gamma)=0$. 
\item If the vertex C$^{\ast}$-algebras $(A_{v})_{v \in V\Gamma}$ are all finite-dimensional and the states $(\omega_v)_{v \in V\Gamma}$ are faithful, then $\mathbf{A}_{\Gamma}$ is simple. 
\item If for every $v\in V\Gamma$ the element $a_{v}$ is a unitary with $\omega_v(a_vx)=\omega_v(xa_v)$ for all $x \in A_v$, then $\mathbf{A}_{\Gamma}$ is simple. 
\item If for every $v \in V\Gamma$ the element $a_v$ is a unitary and $\omega_v$ is tracial, then $\omega_{\Gamma}$ is the unique tracial state on $\mathbf{A}_{\Gamma}$. 
\end{enumerate}
Furthermore, if $\mathbf{A}_{\Gamma}$ is simple, the canonical inclusion $\mathbf{A}_{\Gamma}\hookrightarrow\mathfrak{A}(\mathbf{A},\Gamma)/\mathfrak{I}(\mathbf{A},\Gamma)$ is \emph{C$^{\ast}$-irreducible} in the sense that every intermediate C$^{\ast}$-algebra is simple as well. \end{introtheorem}

C$^{\ast}$-irreducible inclusions have been introduced an studied by Rørdam in \cite{Rordam23}.\\

\vspace{3mm}

\noindent \emph{Structure}. The paper is organized as follows. Section 1 provides the necessary background, including fundamental concepts from graph theory, right-angled Coxeter groups, and graph products of C$^{\ast}$-algebras. In Section 2, we introduce our main construction and examine its key properties. This includes a detailed analysis of the natural gauge action, conditional expectations, universality, nuclearity, exactness, and the ideal structure. Finally, in Section 3, we apply the results established in the previous section to investigate the nuclearity, exactness, simplicity, and the unique trace property of graph product C$^{\ast}$-algebras.\\

\vspace{3mm}


\noindent \textbf{Acknowledgements.} I am grateful to Nadia Larsen for her valuable feedback on an earlier draft of this paper. I also thank Pierre Fima and Diego Martínez for bringing relevant references in connection with this work to my attention. Finally, I acknowledge the support of the Research Foundation Flanders (FWO) through postdoctoral grant 1203924N.

\vspace{5mm}


\section{Preliminaries and Notation\label{sec:Preliminaries-and-notation}}

\vspace{3mm}


\subsection{General Notation}

We denote by $\mathbb{N} := \{0,1,2,\ldots\}$ the set of non-negative integers, and by $\mathbb{N}_{\geq 1} := \{1,2,3,\ldots\}$ the set of positive integers. The neutral element of a group is written as $e$. 

Inner products of Hilbert spaces are assumed to be linear in the second argument. Given a Hilbert space $\mathcal{H}$, we denote by $\mathcal{B}(\mathcal{H})$ the C$^*$-algebra of bounded linear operators on $\mathcal{H}$, and by $\mathcal{K}(\mathcal{H})$ the ideal of compact operators.  If $P \in \mathcal{B}(\mathcal{H})$ is a projection, we write $P^{\perp} := 1 - P$ for its orthogonal complement.

\vspace{3mm}


\subsection{Graphs\label{subsec:Graphs}}

Given a graph $\Gamma$, we denote its \emph{vertex set} by $V\Gamma$ and its \emph{edge set} by $E\Gamma$. Throughout this article, all graphs are assumed to be finite, undirected, and \emph{simplicial}, meaning that $E\Gamma \subseteq (V\Gamma \times V\Gamma) \setminus \{(v,v) \mid v \in V\Gamma\}$. The tuples in 
\[
\mathcal{W}_{\Gamma}:=\bigsqcup_{i\in\mathbb{N}}\underset{i\text{ times}}{\underbrace{(V\Gamma\times\ldots\times V\Gamma)}}
\]
are called \emph{words} in $V\Gamma$. We typically denote such words by bold letters. Unlike in \cite{CaspersFima17}, we include the empty word $\emptyset$ in $\mathcal{W}_{\Gamma}$ by convention.

Following \cite{Green90} and \cite{CaspersFima17}, we equip $\mathcal{W}_{\Gamma}$ with the \emph{shuffle equivalence} relation $\sim$, generated by the rule
\begin{equation}
(v_1, \ldots, v_{i-1}, v_i, v_{i+1}, \ldots, v_n) \sim (v_1, \ldots, v_{i-1}, v_{i+1}, v_i, \ldots, v_n)
\quad \text{if } (v_i, v_{i+1}) \in E\Gamma.
\label{eq:ShuffleEquivalence}
\end{equation}
Two words belonging to the same equivalence class under this relation are said to be \emph{shuffle equivalent}. Similarly, if $\mathbf{v},\mathbf{w}\in\mathcal{W}_{\Gamma}$ are contained in the same equivalence class induced by shuffle equivalence and the additional relation that $(v_{1},\ldots,v_{i},v_{i+1},v_{i+2},\ldots,v_{n})$ is equivalent to $(v_{1},\ldots,v_{i},v_{i+2},\ldots,v_{n})$ if $v_{i}=v_{i+1}$, we write $\mathbf{v}\simeq\mathbf{w}$ and say that the elements are \emph{equivalent}.\\

For a vertex $v \in V\Gamma$, the \emph{link} $\mathrm{Link}(v)$ is defined as the subgraph of $\Gamma$ induced by the vertex set $\{v' \in V\Gamma \mid (v,v') \in E\Gamma\}$, and the \emph{star} $\mathrm{Star}(v)$ is the subgraph induced by $\{v\} \cup V(\mathrm{Link}(v))$.

A word $\mathbf{v} = (v_1, \ldots, v_n) \in \mathcal{W}_{\Gamma}$ is called \emph{reduced} if, for all indices $1 \leq i < j \leq n$ with $v_i = v_j$, there exists an index $i < k < j$ such that $v_k \notin \mathrm{Star}(v_i)$. We denote the set of all reduced words by $\mathcal{W}_{\mathrm{red}}$. Observe that two reduced words that are equivalent under $\simeq$ are necessarily shuffle equivalent.

The \emph{length} of a word $\mathbf{v} \in \mathcal{W}_{\Gamma}$, denoted by $|\mathbf{v}|$, is defined as the length of the shortest representative in its equivalence class. A word $\mathbf{v} = (v_1, \ldots, v_n)$ is reduced if and only if $|\mathbf{v}| = n$.

According to \cite[Lemma~1.3]{CaspersFima17}, any two equivalent reduced words $\mathbf{v} = (v_1, \ldots, v_n)$ and $\mathbf{w} = (w_1, \ldots, w_n)$ admit a unique permutation $\sigma$ of the set $\{1,\ldots,n\}$ such that
\[
\mathbf{w} = (v_{\sigma(1)}, \ldots, v_{\sigma(n)})
\quad \text{and} \quad
\sigma(i) > \sigma(j) \text{ whenever } i > j \text{ and } v_i = v_j.
\]

In light of the discussion above, we fix a set $\mathcal{W}_{\mathrm{min}} \subseteq \mathcal{W}_{\mathrm{red}}$ of representatives for the shuffle equivalence classes. The elements of $\mathcal{W}_{\mathrm{min}}$ are referred to as \emph{minimal} words. Every word in $\mathcal{W}_{\Gamma}$ is equivalent to a unique minimal word.

Given a finite, undirected, simplicial graph $\Gamma$, we denote its \emph{complement}, consisting of the vertex set $V\Gamma$ and the edge set $\left\{ (v,v^{\prime})\in V\Gamma\times V\Gamma\mid v\neq v^{\prime}, \; (v,v^{\prime})\notin E\Gamma\right\}$, by $\Gamma^{c}$. The complement is a finite, undirected and simplicial graph as well.

A \emph{clique} is a complete subgraph of $\Gamma$.

\vspace{3mm}


\subsection{Coxeter Groups\label{subsec:Right-angled-Coxeter-groups}}

A \emph{Coxeter group} is a group $W$ admitting a presentation of the form
\[
W = \left\langle S \,\middle|\, (st)^{m_{st}} = e \text{ for all } s, t \in S \right\rangle,
\]
where $S$ is a (possibly infinite) generating set, and the exponents $m_{st} \in \{1,2,\ldots,\infty\}$ satisfy $m_{ss} = 1$ and $m_{st} \geq 2$ for all distinct $s, t \in S$. A relation of the form $(st)^m = e$ is imposed only when $m_{st} < \infty$; the case $m_{st} = \infty$ indicates that no such relation is imposed. The pair $(W, S)$ is called a \emph{Coxeter system}. It is said to have \emph{finite rank} if the set $S$ is finite, and is called \emph{right-angled} if $m_{st} \in \{2, \infty\}$ for all $s \neq t$, meaning that the generators either commute or generate the infinite dihedral group. For right-angled Coxeter groups, if we have a cancellation of the form $s_1 \cdots s_n = s_1 \cdots \widehat{s_i} \cdots \widehat{s_j} \cdots s_n$ for $s_1,\ldots,s_n \in S$, then $s_i = s_j$ and $s_i$ commutes with $s_{i+1},\ldots,s_{j-1}$. For further background on Coxeter groups, we refer the reader to \cite{Davis08}.

Given a Coxeter system $(W, S)$, we denote the associated word length function by $|\cdot|$. For elements $\mathbf{v}, \mathbf{w} \in W$, we say that $\mathbf{w}$ \emph{starts} in $\mathbf{v}$ if $|\mathbf{v}^{-1}\mathbf{w}| = |\mathbf{w}| - |\mathbf{v}|$. In this case, we write $\mathbf{v} \leq_R \mathbf{w}$. Similarly, $\mathbf{w}$ is said to \emph{end} in $\mathbf{v}$ if $|\mathbf{w}\mathbf{v}^{-1}| = |\mathbf{w}| - |\mathbf{v}|$, and we write $\mathbf{v} \leq_L \mathbf{w}$. These relations define partial orders on $W$, known respectively as the \emph{right weak Bruhat order} and the \emph{left weak Bruhat order}. Both structures turn $W$ into a \emph{complete meet semilattice} (see \cite[Proposition~3.2.1]{BjoernerBrenti05}). For notational convenience, we will typically write $\leq$ in place of $\leq_R$.

Let $\Gamma$ be a finite, undirected, simplicial graph. We associate to $\Gamma$ a finite-rank, right-angled Coxeter system $(W_{\Gamma}, S_{\Gamma})$ by setting
\[
S_{\Gamma} := V\Gamma
\quad \text{and} \quad
W_{\Gamma} := \left\langle S_{\Gamma} \,\middle|\, s^2 = e \text{ for all } s \in S_{\Gamma}, \, st = ts \text{ if } (s,t) \in E\Gamma \right\rangle.
\]
This presentation defines a right-angled Coxeter group, where generators correspond to the vertices of $\Gamma$, and commuting relations correspond to edges.

The group $W_{\Gamma}$ may be identified with the set $\mathcal{W}_{\mathrm{min}}$ of minimal words (or with $\mathcal{W}_{\mathrm{red}}$ modulo shuffle equivalence), via the map
 $(v_1, \ldots, v_n) \mapsto v_1 \cdots v_n$.
This identification endows $\mathcal{W}_{\mathrm{min}}$ with a natural group structure. Moreover, the weak right (and left) Bruhat order on $W_{\Gamma}$ induces a partial order on $\mathcal{W}_{\mathrm{min}}$. For ease of notation, we will henceforth write $W_{\Gamma}$ in place of $\mathcal{W}_{\mathrm{min}}$ when no confusion can arise. Note that the length function $|\cdot|$ on $\mathcal{W}_{\mathrm{min}}$ defined in the previous subsection coincides with the word length in $W_{\Gamma}$ with respect to the generating set $S_{\Gamma}$.

\vspace{3mm}


\subsection{Graph Products of C$^{\ast}$-algebras\label{GraphProductCAlgebras}}

Graph products of operator algebras were introduced by Caspers and Fima in \cite{CaspersFima17} as operator-algebraic analogues of Green’s graph products of groups \cite{Green90}. Similar constructions also appear in \cite{Mlotkowski04} and \cite{SpeicherWysoczanski16}. These products generalize both Voiculescu’s free products (see \cite{Voiculescu85} and also \cite{Avitzour82}) and tensor products by associating to a finite, undirected, simplicial graph $\Gamma$ with a collection of unital C$^{\ast}$-algebras (or von Neumann algebras) $(A_v)_{v \in V\Gamma}$ -- each equipped with a GNS-faithful state $\omega_v$ -- a new C$^{\ast}$-algebra into which the vertex algebras embed canonically, with commutation relations governed by the structure of the graph.\\

Let $\Gamma$ be a finite, undirected, simplicial graph, and let $\mathbf{A} := (A_v)_{v \in V\Gamma}$ be a family of unital C$^{\ast}$-algebras, each equipped with a GNS-faithful state $\omega_v$. Each $A_v$ can be viewed as a subalgebra of $\mathcal{B}(\mathcal{H}_v)$, where $\mathcal{H}_v := L^2(A_v, \omega_v)$ is the corresponding GNS-Hilbert space; the associated cyclic vector will be denoted by $\xi_v \in \mathcal{H}_v$.

For $x \in \mathcal{B}(\mathcal{H}_v)$, define $x^{\circ} := x - \langle x \xi_v, \xi_v \rangle 1$, and let $A_v^{\circ} := \ker(\omega_v)$. Set $\mathcal{H}_v^{\circ} := \mathcal{H}_v \ominus \mathbb{C} \xi_v$, and for $\mathbf{v} = (v_1, \dots, v_n) \in \mathcal{W}_{\text{red}}$, define $\mathcal{H}_{\mathbf{v}}^{\circ} := \mathcal{H}_{v_1}^{\circ} \otimes \dots \otimes \mathcal{H}_{v_n}^{\circ}$.

As described in Subsection~\ref{subsec:Right-angled-Coxeter-groups}, any two equivalent reduced words $\mathbf{v} = (v_1, \dots, v_n)$ and $\mathbf{w} = (w_1, \dots, w_n)$ are related by a unique permutation $\sigma$ such that $\mathbf{w} = (v_{\sigma(1)}, \dots, v_{\sigma(n)})$ and $\sigma(i) > \sigma(j)$ whenever $i > j$ and $v_i = v_j$. This induces a unitary  $\mathcal{Q}_{\mathbf{v}, \mathbf{w}}: \mathcal{H}_{\mathbf{v}}^{\circ} \to \mathcal{H}_{\mathbf{w}}^{\circ}$ given by $\xi_{1}\otimes \ldots \otimes\xi_{n}\mapsto\xi_{\sigma(1)}\otimes \ldots \otimes\xi_{\sigma(n)}$. For simplicity, we suppress the unitary $\mathcal{Q}_{\mathbf{v},\mathbf{w}}$ in the notation and identify $\mathcal{W}_{\text{min}}$ with $W_{\Gamma}$ as in Subsection \ref{subsec:Right-angled-Coxeter-groups}.

The \emph{graph product Hilbert space} (or \emph{Fock space}) associated to the data is defined by
\[
\mathcal{H}_{\Gamma} := \mathbb{C} \Omega \oplus \bigoplus_{\mathbf{w} \in W_{\Gamma} \setminus \{e\}} \mathcal{H}_{\mathbf{w}}^{\circ},
\]
where $\Omega$ denotes the \emph{vacuum vector}. Occasionally, we denote $\mathcal{H}_e^\circ := \mathbb{C}\Omega$.

For each $v \in V\Gamma$, $x \in \mathcal{B}(\mathcal{H}_v)$, $\mathbf{w} \in W_{\Gamma}$, and elementary tensors $\xi_1 \otimes \dots \otimes \xi_n \in \mathcal{H}_{\mathbf{w}}^{\circ}$ with $n = |\mathbf{w}|$, define the operator $\lambda_v(x)$ by
\begin{eqnarray}
\nonumber
& & \lambda_v(x)(\xi_1 \otimes \dots \otimes \xi_n) \\
\nonumber
&:=&
\begin{cases}
x^{\circ} \xi_v \otimes \xi_1 \otimes \dots \otimes \xi_n + \langle x \xi_v, \xi_v \rangle \xi_1 \otimes \dots \otimes \xi_n, & \text{if } v \nleq \mathbf{w}, \\
(x \xi_1 - \langle x \xi_1, \xi_v \rangle \xi_v) \otimes \xi_2 \otimes \dots \otimes \xi_n + \langle x \xi_1, \xi_v \rangle \xi_2 \otimes \dots \otimes \xi_n, & \text{if } v \leq \mathbf{w} \text{ and } \xi_1 \in \mathcal{H}_v^{\circ}.
\end{cases}
\end{eqnarray}
This defines a faithful, unital $*$-homomorphism $\lambda_v: \mathcal{B}(\mathcal{H}_v) \to \mathcal{B}(\mathcal{H}_{\Gamma})$, and the images of $\lambda_v$ and $\lambda_{v'}$ commute whenever $(v, v') \in E\Gamma$, see \cite[Subsection 2.1]{CaspersFima17}. The \emph{algebraic graph product} is the $*$-algebra
\[
\star_{v,\Gamma}^{\text{alg}}(A_v,\omega_v) := \ast\text{-alg} \left( \{ \lambda_v(a) \mid v \in V\Gamma, a \in A_v \} \right) \subseteq \mathcal{B}(\mathcal{H}_{\Gamma}),
\]
and the \emph{graph product C$^{\ast}$-algebra} is its norm closure:
\[
\star_{v,\Gamma}(A_v,\omega_v) := \overline{\star_{v,\Gamma}^{\text{alg}}(A_v,\omega_v) }^{\|\cdot\|} \subseteq \mathcal{B}(\mathcal{H}_{\Gamma}).
\]
We also write $\mathbf{A}_{\Gamma}^{\text{alg}}:=\star_{v,\Gamma}^{\text{alg}}(A_v,\omega_v) $ and $\mathbf{A}_{\Gamma} := \star_{v,\Gamma}(A_v,\omega_v)$.

An operator $x \in \mathbf{A}_{\Gamma}$ of the form $x = \lambda_{v_1}(a_1) \dots \lambda_{v_n}(a_n)$ with $a_i \in A_{v_i}^{\circ}$ and $(v_1, \dots, v_n) \in \mathcal{W}_{\text{red}}$ is called \emph{reduced of type} $\mathbf{v} := v_1 \cdots v_n$, and the corresponding element $\mathbf{v}  \in W_{\Gamma}$ is the \emph{associated word}. We occasionally write $a_{\mathbf{v}} := \lambda_{v_1}(a_1) \dots \lambda_{v_n}(a_n) \in \mathbf{A}_{\Gamma}$.

The \emph{graph product state} $\omega_{\Gamma}$ on $\mathbf{A}_\Gamma$ is the restriction of the vector state associated with $\Omega$, and it is GNS-faithful. Moreover, $\omega_{\Gamma}(x) = 0$ for all reduced operators $x \in \mathbf{A}_{\Gamma}$ of type $\mathbf{v} \in W_{\Gamma} \setminus\{e\}$.

\begin{proposition}[{\cite[Proposition 2.12]{CaspersFima17}}]
Let $B$ be a unital C$^{\ast}$-algebra with a GNS-faithful state $\omega$. Suppose that for each $v \in V\Gamma$ there exists a unital faithful $*$-homomorphism $\pi_v: A_v \to B$ satisfying:
\begin{itemize}
    \item $B = C^{\ast}( \{ \pi_v(a) \mid v \in V\Gamma,\, a \in A_v \} )$.
    \item $\pi_v(A_v)$ and $\pi_{v'}(A_{v'})$ commute whenever $(v,v') \in E\Gamma$.
    \item For each reduced operator $x = \lambda_{v_1}(a_1) \dots \lambda_{v_n}(a_n)$ with $a_i \in A_{v_i}^{\circ}$, one has $\omega(x) = 0$.
\end{itemize}
Then there exists a unique $*$-isomorphism $\pi: \mathbf{A}_{\Gamma} \to B$ such that $\pi|_{A_v} = \pi_v$ for all $v \in V\Gamma$, and $\omega \circ \pi = \omega_{\Gamma}$.
\end{proposition}

\medskip

The construction also admits a right-handed version. For $v \in V\Gamma$, $x \in \mathcal{B}(\mathcal{H}_v)$, $\mathbf{w} \in W_{\Gamma}$, and elementary tensors $\xi_1 \otimes \dots \otimes \xi_n \in \mathcal{H}_{\mathbf{w}}^{\circ}$ with $n = |\mathbf{w}|$, define
\begin{eqnarray}
\nonumber
& & \rho_v(x)(\xi_1 \otimes \dots \otimes \xi_n) \\
\nonumber
&:=&
\begin{cases}
\xi_1 \otimes \dots \otimes \xi_n \otimes x^{\circ} \xi_v + \langle x \xi_v, \xi_v \rangle \xi_1 \otimes \dots \otimes \xi_n, & \text{if } v \nleq \mathbf{w}^{-1}, \\
\xi_1 \otimes \dots \otimes \xi_{n-1} \otimes (x \xi_n - \langle x \xi_n, \xi_v \rangle \xi_v) + \langle x \xi_n, \xi_v \rangle \xi_1 \otimes \dots \otimes \xi_{n-1}, & \text{if } v \leq \mathbf{w}^{-1} \text{ and } \xi_n \in \mathcal{H}_v^{\circ}.
\end{cases}
\end{eqnarray}
This yields a faithful, unital $*$-homomorphism $\rho_v: \mathcal{B}(\mathcal{H}_v) \to \mathcal{B}(\mathcal{H}_{\Gamma})$ such that $\rho_v$ and $\rho_{v'}$ commute whenever $(v,v') \in E\Gamma$. Furthermore, for $v, v' \in V\Gamma$ and $x \in \mathcal{B}(\mathcal{H}_v)$, $y \in \mathcal{B}(\mathcal{H}_{v'})$, we have $\lambda_v(x) \rho_{v'}(y) = \rho_{v'}(y) \lambda_v(x)$ whenever $v \neq v'$ or $v = v'$ and $xy = yx$ (see \cite[Proposition 2.3]{CaspersFima17}).

\vspace{3mm}


\section{Main Construction} \label{sec:Main-construction}

\vspace{3mm}

Let $\Gamma$ be a finite, undirected, simplicial graph, and let $\mathbf{A} := (A_{v})_{v \in V\Gamma}$ be a collection of unital C$^{\ast}$-algebras, each equipped with a GNS-faithful state $\omega_v$. As in Subsection~\ref{GraphProductCAlgebras}, consider the associated graph product Hilbert space $\mathcal{H}_{\Gamma} := \mathbb{C}\Omega \oplus \bigoplus_{\mathbf{w} \in W_{\Gamma} \setminus \{e\}} \mathcal{H}^{\circ}_{\mathbf{w}}$, and the corresponding graph product C$^{\ast}$-algebra $\mathbf{A}_{\Gamma} := \star_{v,\Gamma}(A_{v}, \omega_{v})$. For every element $\mathbf{w}$ in the right-angled Coxeter group $W_{\Gamma}$ (as defined in Subsection~\ref{subsec:Right-angled-Coxeter-groups}), let $Q_{\mathbf{w}} \in \mathcal{B}(\mathcal{H}_{\Gamma})$ denote the orthogonal projection onto the subspace $\bigoplus_{\mathbf{v} \in W_{\Gamma} \setminus \{e\} : \mathbf{w} \leq \mathbf{v}} \mathcal{H}^{\circ}_{\mathbf{v}} \subseteq \mathcal{H}_{\Gamma}$.

We define
\[
\mathfrak{A}(\mathbf{A}, \Gamma) := C^{\ast}\left( \{ Q_{v} \mid v \in V\Gamma \} \cup \mathbf{A}_{\Gamma} \right) \subseteq \mathcal{B}(\mathcal{H}_{\Gamma}),
\]
as the C$^{\ast}$-subalgebra of $\mathcal{B}(\mathcal{H}_{\Gamma})$ generated by all projections $Q_v$ for $v \in V\Gamma$ together with $\mathbf{A}_{\Gamma}$. Furthermore, we denote by $\mathcal{D}(\mathbf{A}, \Gamma)$ the C$^{\ast}$-subalgebra of $\mathfrak{A}(\mathbf{A}, \Gamma)$ consisting of all operators $x \in \mathfrak{A}(\mathbf{A}, \Gamma)$ that are \emph{diagonal}, in the sense that
\[
x(\mathbb{C} \Omega) \subseteq \mathbb{C} \Omega \quad \text{and} \quad x(\mathcal{H}^{\circ}_{\mathbf{w}}) \subseteq \mathcal{H}^{\circ}_{\mathbf{w}} \quad \text{for every } \mathbf{w} \in W_{\Gamma} \setminus \{e\}.
\]
By abuse of notation, we will also denote by $\omega_{\Gamma}$ the restriction of the vacuum vector state to $\mathfrak{A}(\mathbf{A}, \Gamma)$. For notational simplicity, we will suppress the $\ast$-embeddings $(\lambda_v)_{v \in V\Gamma}$ and simply regard each $a \in A_{v}^{\circ}$, for $v \in V\Gamma$, as an element of $\mathfrak{A}(\mathbf{A}, \Gamma)$.\\

The primary motivation for this construction stems from the theory of Hecke C$^{\ast}$-algebras, which have been studied in \cite{CaspersKlisseLarsen21, Klisse23-1, Klisse23-2}.

\begin{example} \label{MainExample}
\textbf{(1)} Let $(W, S)$ be a finite-rank right-angled Coxeter system, and let $q := (q_s)_{s \in S} \in \mathbb{R}_{>0}^{S}$ be a multi-parameter. The associated \emph{right-angled Hecke C$^{\ast}$-algebra} $C_{r,q}^{\ast}(W)$ is the C$^{\ast}$-subalgebra of $\mathcal{B}(\ell^2(W))$ generated by the family $(T_s^{(q)})_{s \in S}$, where
\[
T_{s}^{(q)} \delta_{\mathbf{w}} := 
\begin{cases}
\delta_{s\mathbf{w}}, & \text{if } s \nleq \mathbf{w} \\
\delta_{s\mathbf{w}} + p_s(q)\, \delta_{\mathbf{w}}, & \text{if } s \leq \mathbf{w}.
\end{cases}
\]
Here $(\delta_{\mathbf{w}})_{\mathbf{w} \in W}$ denotes the canonical orthonormal basis of $\ell^2(W)$ and $p_s(q) := q_s^{-1/2}(q - 1) \in \mathbb{R}$. Let $\tau_q$ denote the canonical faithful tracial state on $C_{r,q}^{\ast}(W)$ induced by the vector state associated with $\delta_e$. By \cite[Corollary~3.4]{Caspers20} (see also \cite[Section~1.10]{CaspersKlisseLarsen21} for the multi-parameter case), we have a graph product decomposition
\[
(C_{r,q}^{\ast}(W), \tau_q) \cong \star_{s,\Gamma} \left( C_{r,q}^{\ast}(W_s), \tau_{q_s} \right),
\]
where $W_s \cong \mathbb{Z}_2$ is the subgroup of $W$ generated by $s$, and where $\Gamma$ is defined via $V\Gamma := S$ and $E\Gamma := \{ (s, t) \in S \times S \mid st = ts \}$. Each vertex algebra $C_{r,q}^{\ast}(W_s)$ is two-dimensional, and the graph product Hilbert space coincides with $\ell^2(W)$. In \cite[Section~4]{Klisse23-1}, it is shown that $\mathfrak{A}(\mathbf{A}, \Gamma)$, for $\mathbf{A} := (C_{r,q}^{\ast}(W_s))_{s \in S}$, is isomorphic to the reduced crossed product associated with the canonical action of $W$ on its combinatorial compactification (cf.~\cite{Lecureux10, CapraceLecureux11, Klisse23-1}). In particular, $\mathfrak{A}(\mathbf{A}, \Gamma)$ is independent of the choice of the parameter $q$. Moreover, $\mathfrak{A}(\mathbf{A}, \Gamma)$ contains the algebra of compact operators $\mathcal{K}(\ell^2(W))$, and the corresponding quotient $\mathfrak{A}(\mathbf{A}, \Gamma)/\mathcal{K}(\ell^2(W))$ identifies with the reduced crossed product of $W$ acting on its combinatorial boundary $\partial(W, S)$.

\vspace{2mm}
\textbf{(2)} Similarly, let $\Gamma$ be a finite, undirected, simplicial graph and let $\mathbf{G}_{\Gamma} := \star_{v, \Gamma} G_v$ be a graph product of countable vertex groups $\mathbf{G} := (G_v)_{v \in V\Gamma}$. Denote by $\tau_v$ the canonical tracial state on $C_r^{\ast}(G_v)$ for each $v \in V\Gamma$. By \cite[Remark~2.13]{CaspersFima17}, the reduced group C$^{\ast}$-algebra satisfies
\[
(C_r^{\ast}(\mathbf{G}_{\Gamma}), \tau) \cong \star_{v, \Gamma}(C_r^{\ast}(G_v), \tau_v),
\]
where $\tau$ denotes the canonical tracial state on $C_r^{\ast}(\mathbf{G}_{\Gamma})$, and the corresponding graph product Hilbert space is $\ell^2(\mathbf{G}_{\Gamma})$. Define a countable, undirected, simplicial graph $K$ with vertex set $VK := \mathbf{G}_{\Gamma}$ and edge set
\[
EK := \left\{ \big(g_1 \cdots g_n, g_1 \cdots g_n h\big) \mid g_i \in G_{v_i} \setminus \{e\},\ h \in G_{v_{n+1}} \setminus \{e\},\ (v_1, \ldots, v_{n+1}) \in \mathcal{W}_{\text{red}} \right\}.
\]
Then $(K, o)$ forms a connected rooted graph with root $o := e \in \mathbf{G}_{\Gamma}$. By the same argument as in \cite[Section~4]{Klisse23-1}, the C$^{\ast}$-algebra $\mathfrak{A}(\mathbf{A}, \Gamma)$ identifies with the reduced crossed product arising from the canonical action of $\mathbf{G}_{\Gamma}$ on the compactification $\overline{(K, o)}$, as constructed in \cite[Section~2]{Klisse23-1}.
\end{example}

\vspace{3mm}


\subsection{The Projections $Q_{\mathbf{w}}$}

In \cite{Klisse23-1}, the author introduced and studied topological boundaries and compactifications of connected rooted graphs. As indicated in Example~\ref{MainExample}, this framework proves particularly fruitful when applied to (Cayley graphs of) Coxeter groups, especially in the context of Hecke operator algebras. Notably, it was employed in \cite{Klisse23-1} to study Ozawa’s Akemann–Ostrand property for Hecke-von Neumann algebras (see \cite{Ozawa04}), and later in \cite{Klisse23-2} to characterize the simplicity of right-angled Hecke C$^{\ast}$-algebras. 

Following \cite{Klisse23-1}, we briefly review the construction in the setting of Coxeter groups, for which it was introduced earlier by Lam–Thomas \cite{LamThomas15} and Caprace–Lécureux \cite{CapraceLecureux11}, albeit in different formalisms. For details and generalizations, we refer the reader to \cite{Klisse23-1}.

Let $(W,S)$ be a finite-rank Coxeter system, and let $K := \text{Cay}(W,S)$ denote the Cayley graph of $W$ with respect to the generating set $S$; that is, the graph with vertex set $W$ and edge set $\{ (\mathbf{v}, \mathbf{w}) \in W \times W \mid \mathbf{v}^{-1}\mathbf{w} \in S \}$. We equip $K$ with the metric $d(\mathbf{v}, \mathbf{w}) := |\mathbf{v}^{-1}\mathbf{w}|$. A \emph{geodesic path} in $K$ is a (finite or infinite) sequence $\alpha_0\alpha_1\cdots$ of vertices such that $d(\alpha_i, \alpha_j) = |i - j|$ for all $i,j$. Finite geodesic paths are frequently extended to infinite ones by repetition of their terminal vertex, without explicit mention.

For a geodesic path $\alpha$ and an element $\mathbf{w} \in W$, we write $\mathbf{w} \leq \alpha$ if $\mathbf{w} \leq \alpha_i$ for all sufficiently large $i$, and $\mathbf{w} \nleq \alpha$ otherwise. An equivalence relation $\sim$ is defined on the set of infinite geodesics by declaring $\alpha \sim \beta$ if and only if $\mathbf{w} \leq \alpha \Leftrightarrow \mathbf{w} \leq \beta$ for every $\mathbf{w} \in W$. Denote by $\partial(W,S)$ the set of equivalence classes, called the \emph{boundary} of $(W,S)$, and set $\overline{(W,S)} := W \cup \partial(W,S)$, the corresponding \emph{compactification}.

The weak right Bruhat order on $W$ extends to a partial order on $\overline{(W,S)}$ (see \cite[Lemma 2.2]{Klisse23-1}). We endow $\overline{(W,S)}$ with the topology generated by the subbasis
\[
\mathcal{U}_{\mathbf{w}} := \{ z \in \overline{(W,S)} \mid \mathbf{w} \leq z \}, \quad \mathcal{U}_{\mathbf{w}}^c := \{ z \in \overline{(W,S)} \mid \mathbf{w} \nleq z \},
\]
for $\mathbf{w} \in W$. With this topology, both $\partial(W,S)$ and $\overline{(W,S)}$ become compact, metrizable spaces, and $W$ embeds as a dense discrete subset of $\overline{(W,S)}$. The left action of $W$ on itself extends to a continuous action on $\overline{(W,S)}$ by homeomorphisms, preserving the boundary.

\medskip

The compactification can also be described operator-algebraically. We recall the definition of the operators $Q_{\mathbf{w}} \in \mathcal{B}(\mathcal{H}_\Gamma)$, where $Q_{\mathbf{w}}$ is the orthogonal projection onto the Hilbert subspace $\bigoplus_{\mathbf{v} \in W_\Gamma \setminus \{e\} : \mathbf{w} \leq \mathbf{v}} \mathcal{H}_{\mathbf{v}}^\circ \subseteq \mathcal{H}_\Gamma$. In \cite{Klisse23-2}, a similar family $(P_{\mathbf{w}})_{\mathbf{w} \in W}$ of projections on $\ell^2(W)$ was considered. Letting $(\delta_{\mathbf{v}})_{\mathbf{v} \in W}$ denote the standard orthonormal basis of $\ell^2(W)$, define $P_{\mathbf{w}}$ as the orthogonal projection onto the closed subspace spanned by $\{ \delta_{\mathbf{v}} \mid \mathbf{w} \leq \mathbf{v} \}$.

As discussed in Subsection \ref{subsec:Right-angled-Coxeter-groups}, the weak right Bruhat order turns $W$ into a complete meet semilattice, and the join $\mathbf{v} \vee \mathbf{w}$ exists if and only if $\mathbf{v}$ and $\mathbf{w}$ have a common upper bound. One checks that $P_{\mathbf{v}}P_{\mathbf{w}} = P_{\mathbf{v} \vee \mathbf{w}}$, where $P_{\mathbf{v} \vee \mathbf{w}} := 0$ if the join does not exist. As shown in \cite[Theorem 2.13]{Klisse23-1}, the C$^*$-subalgebra $\mathcal{D}(W,S) \subseteq \mathcal{B}(\ell^2(W))$ generated by the $P_{\mathbf{w}}$ is universal for projections satisfying these relations. Moreover, there are homeomorphisms
\[
\overline{(W,S)} \cong \mathrm{Spec}(\mathcal{D}(W,S)), \quad \partial(W,S) \cong \mathrm{Spec}(\pi(\mathcal{D}(W,S))),
\]
where $\pi : \mathcal{B}(\ell^2(W)) \twoheadrightarrow \mathcal{B}(\ell^2(W)) / \mathcal{K}(\ell^2(W))$ is the canonical quotient map onto the Calkin algebra. These homeomorphisms are induced by the identifications $\mathcal{D}(W,S)\cong C(\overline{(W,S)})$ and $\pi(\mathcal{D}(W,S))\cong C(\partial(W,S))$ defined by $P_{\mathbf{w}}\mapsto\chi_{\mathcal{U}_{\mathbf{w}}}$ and $\pi(P_{\mathbf{w}})\mapsto\chi_{\mathcal{U}_{\mathbf{w}}\cap\partial(W,S)}$ for $\mathbf{w}\in W$, where $\chi_{\mathcal{U}_{\mathbf{w}}}$ and $\chi_{\mathcal{U}_{\mathbf{w}}\cap\partial(W,S)}$ are the characteristic functions on $\mathcal{U}_{\mathbf{w}}$ and $\mathcal{U}_{\mathbf{w}}\cap\partial(W,S)$ respectively.

\medskip

The construction in \cite{Klisse23-1} is compatible with the graph product setting, as shown in the following result.

\begin{lemma} \label{IdentificationLemma}
Let $\Gamma$ be a finite, undirected, simplicial graph, and let $\mathbf{A} = (A_v)_{v \in V\Gamma}$ be a family of unital C$^*$-algebras with GNS-faithful states $(\omega_v)_{v \in V\Gamma}$. Then the assignment $P_{\mathbf{w}} \mapsto Q_{\mathbf{w}}$ defines a $*$-isomorphism $\mathcal{D}(W_\Gamma, S_\Gamma) \cong C^*(\{ Q_{\mathbf{w}} \mid \mathbf{w} \in W_\Gamma \})$. \end{lemma}

\begin{proof}
The universal property of $\mathcal{D}(W_\Gamma, S_\Gamma)$ implies that the assignment $P_{\mathbf{w}} \mapsto Q_{\mathbf{w}}$ extends to a surjective $*$-homomorphism $\rho$. For each $v \in V\Gamma$, choose a unit vector $\eta_v \in \mathcal{H}_v^\circ$ and define $\eta_{\mathbf{v}} := \eta_{v_1} \otimes \cdots \otimes \eta_{v_n} \in \mathcal{H}_{\mathbf{v}}^\circ$ for reduced words $\mathbf{v} = v_1 \cdots v_n$. Then, for all $\mathbf{v}, \mathbf{w} \in W_\Gamma$, $\langle P_{\mathbf{w}} \delta_{\mathbf{v}}, \delta_{\mathbf{v}} \rangle = \langle \rho(P_{\mathbf{w}}) \eta_{\mathbf{v}}, \eta_{\mathbf{v}} \rangle$. Thus, for $x \in \mathcal{D}(W_\Gamma, S_\Gamma) \subseteq \ell^\infty(W_\Gamma)$,
\[
\|x\| = \sup_{\mathbf{v} \in W_\Gamma} | \langle x \delta_{\mathbf{v}}, \delta_{\mathbf{v}} \rangle | = \sup_{\mathbf{v} \in W_\Gamma} | \langle \rho(x) \eta_{\mathbf{v}}, \eta_{\mathbf{v}} \rangle | \leq \|\rho(x)\|,
\]
so $\rho$ is isometric and hence a $*$-isomorphism.
\end{proof}

\medskip

As in Lemma \ref{IdentificationLemma}, let $\Gamma$ be a finite, undirected, simplicial graph and let $\mathbf{A}:=(A_{v})_{v\in V\Gamma}$ be a collection of unital C$^{\ast}$-algebras, equipped with GNS-faithful states $(\omega_{v})_{v\in V\Gamma}$. The action of $W_\Gamma$ on $\mathcal{D}(W_\Gamma, S_\Gamma) \subseteq \ell^\infty(W_\Gamma)$ given by $(\mathbf{w}.f)(\mathbf{v}) := f(\mathbf{w}^{-1}\mathbf{v})$ for $\mathbf{v}, \mathbf{w}\in W_\Gamma$ induces an action on $\mathrm{Spec}(\mathcal{D}(W_\Gamma, S_\Gamma)) \cong \overline{(W_\Gamma,S_\Gamma)}$, which corresponds to the left multiplication action of $W_\Gamma$ on itself. Using Lemma~\ref{IdentificationLemma}, this action transfers to $C^*(\{ Q_{\mathbf{w}} \mid \mathbf{w}\in W_{\Gamma} \})$, yielding the following analogue of \cite[Proposition 2.2]{Klisse23-2}.

\begin{lemma} \label{ActionLemma}
Let $\Gamma$ be a finite, undirected, simplicial graph and let $\mathbf{A}:=(A_{v})_{v\in V\Gamma}$ be a collection of unital C$^{\ast}$-algebras, equipped with GNS-faithful states $(\omega_{v})_{v\in V\Gamma}$. Then, for all $v \in V\Gamma$ and $\mathbf{w} \in W_\Gamma$, the following identities hold:
\begin{enumerate}
    \item If $\mathbf{w} \notin C_{W_\Gamma}(v)$, then $v.Q_{\mathbf{w}} = Q_{v\mathbf{w}}$.
    \item If $\mathbf{w} \in C_{W_\Gamma}(v)$ and $v \leq \mathbf{w}$, then $v.Q_{\mathbf{w}} = Q_{v\mathbf{w}} - Q_{\mathbf{w}}$.
    \item If $\mathbf{w} \in C_{W_\Gamma}(v)$ and $v \nleq \mathbf{w}$, then $v.Q_{\mathbf{w}} = Q_{\mathbf{w}}$.
\end{enumerate}
Here $C_{W_\Gamma}(v) := \{ \mathbf{u} \in W_\Gamma \mid \mathbf{u}v = v\mathbf{u} \}$ denotes the centralizer of $v$ in $W_\Gamma$.
\end{lemma}

The analogue of Lemma~\ref{ActionLemma} appears as Proposition~2.2 in~\cite{Klisse23-2}, where it plays a key role in the characterization of the simplicity of right-angled Hecke C$^*$-algebras. Another essential ingredient is the analysis of the canonical action of right-angled Coxeter groups on their combinatorial boundaries -- specifically, the characterization of minimality, strong proximality, and topological freeness as presented in~\cite[Theorem~3.19 and Proposition~3.25]{Klisse23-1}.

In Subsection~\ref{subsec:Ideal}, we will require the following refinement of~\cite[Proposition~3.25]{Klisse23-1}, which involves the concept of closed walks in graphs.

\begin{definition}
Let $\Gamma$ be a finite, undirected, simplicial graph. A \emph{walk} in $\Gamma$ is a sequence of vertices $(v_{1}, \dots, v_{n}) \in V\Gamma \times \cdots \times V\Gamma$ such that $(v_i, v_{i+1}) \in E\Gamma$ for all $i = 1, \dots, n-1$. The walk is said to be \emph{closed} if $(v_1, v_n) \in E\Gamma$, and it is said to \emph{cover the whole graph} if $\{v_1, \dots, v_n\} = V\Gamma$.
\end{definition}

A finite, undirected, simplicial graph $\Gamma$ admits a closed walk that covers the entire graph if and only if it is connected.

\begin{proposition} \label{TopologicalFrenessImplication}
Let $\Gamma$ be a finite, undirected, simplicial graph with $\#V\Gamma \geq 3$, and assume its complement $\Gamma^{c}$ is connected. Let $\mathbf{w} \in W_{\Gamma}$ be arbitrary, and let $\mathcal{S} \subseteq W_{\Gamma}$ be a finite subset. Then there exists $\mathbf{v} \in W_{\Gamma}$ and a closed walk $(t_1, \dots, t_n) \in V\Gamma \times \cdots \times V\Gamma$ in the complement $\Gamma^c$ that covers the entire graph such that $|\mathbf{w}\mathbf{v}(t_{1}\cdots t_{n})^{L}|=|\mathbf{w}|+|\mathbf{v}|+|(t_{1}\cdots t_{n})^{L}|$ and $\mathbf{w}\mathbf{v}(t_{1}\cdots t_{n})^{L}\nleq\mathbf{x}\mathbf{w}\mathbf{v}(t_{1}\cdots t_{n})^{L}$ for all $\mathbf{x}\in\mathcal{S}\setminus\{e\}$, $L\in\mathbb{N}_{\geq 1}$. 
\end{proposition}

\begin{proof}
By~\cite[Proposition~3.25]{Klisse23-1}, the canonical action $W_{\Gamma} \curvearrowright \partial(W_{\Gamma}, S_{\Gamma})$ is topologically free. Hence, the set
\[
\bigcup_{\mathbf{x} \in \mathcal{S} \setminus \{e\}} \{ z \in \partial(W_{\Gamma}, S_{\Gamma}) \mid \mathbf{x} . z = z \}
\]
has empty interior and therefore does not contain the open set $\mathcal{U}_{\mathbf{w}} = \{ z \in \partial(W_{\Gamma}, S_{\Gamma}) \mid \mathbf{w} \leq z \}$. Thus, we can find $z_0 \in \partial(W_{\Gamma}, S_{\Gamma})$ with $\mathbf{w} \leq z_0$ and $\mathbf{x} . z_0 \neq z_0$ for all $\mathbf{x} \in \mathcal{S} \setminus \{e\}$.

Let $(\beta_i)_{i \in \mathbb{N}} \subseteq W_{\Gamma}$ be a geodesic ray with $\beta_0 = e$ and $|\mathbf{w} \beta_i| = |\mathbf{w}| + |\beta_i|$ for all $i$, such that $\mathbf{w} \beta_i \to z_0$. For each $\mathbf{x} \in \mathcal{S} \setminus \{e\}$, there exists $\mathbf{u}_{\mathbf{x}} \in W_{\Gamma}$ such that either
\[
\mathbf{u}_{\mathbf{x}} \leq z_0 \; \text{and} \; \mathbf{u}_{\mathbf{x}} \nleq \mathbf{x} . z_0,
\quad \text{or} \quad
\mathbf{u}_{\mathbf{x}} \nleq z_0 \; \text{and} \; \mathbf{u}_{\mathbf{x}} \leq \mathbf{x} . z_0.
\]
This implies the existence of $i_0(\mathbf{x}) \in \mathbb{N}$ such that for all $i \geq i_0(\mathbf{x})$ either
\[
\mathbf{u}_{\mathbf{x}} \leq \mathbf{w} \beta_i \; \text{ and } \; \mathbf{u}_{\mathbf{x}} \nleq \mathbf{x} \mathbf{w} \beta_i,
\quad \text{or } \quad \mathbf{u}_{\mathbf{x}} \nleq \mathbf{w} \beta_i \; \text{ and } \; \mathbf{u}_{\mathbf{x}} \leq \mathbf{x} \mathbf{w} \beta_i.
\]
Let $i_0 > \max\{i_0(\mathbf{x}) \mid \mathbf{x} \in \mathcal{S} \setminus \{e\}\}$. By choosing $i_0$ large enough, we may further assume that for all $i \geq i_0$ and all $\mathbf{x} \in \mathcal{S}$, $|\mathbf{w} \beta_i| = |\mathbf{w} \beta_{i_0}| + |\beta_{i_0}^{-1} \beta_i|$ and $|\mathbf{x} \mathbf{w} \beta_i| = |\mathbf{x} \mathbf{w} \beta_{i_0}| + |\beta_{i_0}^{-1} \beta_i|$.

Let $(t_1, \dots, t_n)$ be a closed path in $\Gamma^c$ covering the entire graph such that $t_n \leq_L \mathbf{w} \beta_{i_0}$. Then the element $\mathbf{w}\beta_{i_{0}}(t_{1}\ldots t_{n})^{L}$ satisfies $|\mathbf{w}\beta_{i_{0}}(t_{1}\cdots t_{n})^{L}|=|\mathbf{w}|+|\beta_{i_{0}}|+|(t_{1}\cdots t_{n})^{L}|$ for every $L\in\mathbb{N}$ and by the choice of the path, $\mathbf{u}_{\mathbf{x}}\leq\mathbf{w}\beta_{i_{0}}(t_{1}\cdots t_{n})^{L}$, $\mathbf{u}_{\mathbf{x}}\nleq\mathbf{x}\mathbf{w}\beta_{i_{0}}(t_{1}\cdots t_{n})^{L}$ or $\mathbf{u}_{\mathbf{x}}\nleq\mathbf{w}\beta_{i_{0}}(t_{1}\cdots t_{n})^{L}$, $\mathbf{u}_{\mathbf{x}}\leq\mathbf{x}\mathbf{w}\beta_{i_{0}}(t_{1}\cdots t_{n})^{L}$ for every $L\in\mathbb{N}$, $\mathbf{x}\in\mathcal{S}\setminus\{e\}$. We distinguish two cases:\\

\begin{itemize}
\item \emph{Case 1:} If $\mathbf{u}_{\mathbf{x}} \leq \mathbf{w} \beta_{i_0} (t_1 \cdots t_n)^L$ and $\mathbf{u}_{\mathbf{x}} \nleq \mathbf{x} \mathbf{w} \beta_{i_0} (t_1 \cdots t_n)^L$, then by transitivity of the partial order, $\mathbf{w} \beta_{i_0} (t_1 \cdots t_n)^L \nleq \mathbf{x} \mathbf{w} \beta_{i_0} (t_1 \cdots t_n)^L$.\\
\item \emph{Case 2:} If $\mathbf{u}_{\mathbf{x}} \nleq \mathbf{w} \beta_{i_0} (t_1 \cdots t_n)^L$ and $\mathbf{u}_{\mathbf{x}} \leq \mathbf{x} \mathbf{w} \beta_{i_0} (t_1 \cdots t_n)^L$, assuming $\mathbf{w} \beta_{i_0} (t_1 \cdots t_n)^L \leq \mathbf{x} \mathbf{w} \beta_{i_0} (t_1 \cdots t_n)^L$ would imply that the latter word starts with both $\mathbf{w} \beta_{i_0} (t_1 \cdots t_n)^L$ and $\mathbf{u}_{\mathbf{x}}$. But by our construction this is only possible if already $\mathbf{u}_{x}\leq\mathbf{w}\beta_{i_{0}}(t_{1}\cdots t_{n})^{L}$, thus leading to a contradiction.
\end{itemize}

\vspace{3mm}

Combining the two cases implies the desired statement.
\end{proof}

\vspace{3mm}


\subsection{Creation, Annihilation and Diagonal Operators}

In this subsection, we introduce three classes of operators -- creation, annihilation, and diagonal -- which will be fundamental in our subsequent analysis.

\begin{definition} \label{CreationDefinition}
Let $\Gamma$ be a finite, undirected simplicial graph, and let $\mathbf{A} := (A_v)_{v \in V\Gamma}$ be a collection of unital C$^{\ast}$-algebras, each endowed with a GNS-faithful state $\omega_v$. For $v \in V\Gamma$ and $a \in A_v$, we define:
\begin{itemize}
    \item the \emph{creation operator} associated to $a$ by $a^{\dagger} := Q_v a Q_v^{\perp} \in \mathfrak{A}(\mathbf{A}, \Gamma)$,
    \item the \emph{diagonal operator} associated to $a$  by $\mathfrak{d}(a) := Q_v a Q_v  \in \mathfrak{A}(\mathbf{A}, \Gamma)$,
    \item the \emph{annihilation operator} associated to $a$  by $((a^*)^{\dagger})^* := Q_v^{\perp} a Q_v  \in \mathfrak{A}(\mathbf{A}, \Gamma)$.
\end{itemize}
\end{definition}

\begin{remark}
The terminology in Definition~\ref{CreationDefinition} is justified by the following identities: for $v \in V\Gamma$, $a \in A_v$, and $\xi \in \mathcal{H}_{\mathbf{w}}^\circ$ with $v \nleq \mathbf{w} \in W_\Gamma$,
\[
a^{\dagger} \xi = (a^\circ \xi_v) \otimes \xi, \quad \mathfrak{d}(a)\xi = 0, \quad ((a^*)^{\perp})^* \xi = 0,
\]
whereas for $b \in A_v^\circ$,
\[
a^{\dagger} ( b\xi_v \otimes \xi )= 0, \quad \mathfrak{d}(a) ( b\xi_v \otimes \xi ) = (ab)^\circ \xi_v \otimes \xi, \quad ((a^*)^{\dagger})^* ( b\xi_v \otimes \xi )= \omega_v(ab) \xi.
\]
In particular, we have $\mathfrak{d}(a) \in \mathcal{D}(\mathbf{A}, \Gamma)$ for all $a \in A_v$, $v \in V\Gamma$.
\end{remark}

The operators defined above satisfy the following identities, which we will utilize extensively throughout the remainder of this article.

\begin{lemma} \label{MainIdentities}
Let $\Gamma$ be a finite, undirected simplicial graph, and let $\mathbf{A} := (A_v)_{v \in V\Gamma}$ be a collection of unital C$^{\ast}$-algebras with GNS-faithful states $(\omega_v)_{v \in V\Gamma}$. Then:
\begin{enumerate}
    \item For $v,v' \in V\Gamma$ and $a \in A_v$, $b \in A_{v'}$, we have:
    \[
    a^{\dagger} b^{\dagger} = 0 \quad \text{if } v = v', \quad \text{and} \quad [a^{\dagger}, b^{\dagger}] = 0 \quad \text{if } (v,v') \in E\Gamma.
    \]
    
    \item For $v,v' \in V\Gamma$ and $a \in A_v$, $b \in A_{v'}$, the following holds:
    \[
    b^{\dagger} \mathfrak{d}(a) = 0 \quad \text{if } v = v', \quad \text{and} \quad \mathfrak{d}(a) b^{\dagger} = 
    \begin{cases}
        (ab - \omega_v(b) a)^{\dagger}, & \text{if } v = v', \\
        0, & \text{if } (v,v') \in E\Gamma^c, \\
        b^{\dagger} \mathfrak{d}(a), & \text{if } (v,v') \in E\Gamma.
    \end{cases}
    \]
    
    \item For $v,v' \in V\Gamma$ and $a \in A_v$, $b \in A_{v'}$, we have:
    \[
    a^{\dagger}(b^{\dagger})^* = \mathfrak{d}(ab^*) + \mathfrak{d}(a)\mathfrak{d}(b^*) \quad \text{if } v = v',
    \]
    and
    \[
    (a^{\dagger})^* b^{\dagger} = 
    \begin{cases}
        (\omega_v(a^*b) - \overline{\omega_v(a)} \omega_v(b)) Q_v^\perp, & \text{if } v = v', \\
        0, & \text{if } (v,v') \in E\Gamma^c, \\
        b^{\dagger} (a^{\dagger})^*, & \text{if } (v,v') \in E\Gamma.
    \end{cases}
    \]
    
    \item For $v,v' \in V\Gamma$ and $a \in A_v$, $b \in A_{v'}$, we have:
    \[
    \mathfrak{d}(a) \mathfrak{d}(b) = 0 \quad \text{if } (v,v') \in E\Gamma^c, \quad \text{and} \quad [\mathfrak{d}(a), \mathfrak{d}(b)] = 0 \quad \text{if } (v,v') \in E\Gamma.
    \]
    
    \item For $v \in V\Gamma$, $a \in A_v^\circ$, and $\mathbf{w} \in W_\Gamma$, the following relations hold:
    \[
    Q_{\mathbf{w}} a^{\dagger} = a^{\dagger}(v.Q_{\mathbf{w}}), \quad Q_{\mathbf{w}} (a^{\dagger})^* = (a^{\dagger})^*(v.Q_{\mathbf{w}}).
    \]
\end{enumerate}
\end{lemma}

\begin{proof} The proof proceeds by direct computations using the definitions of the operators and the combinatorial structure of the underlying graph. For this, let $v,v^{\prime}\in V\Gamma$, $\mathbf{w}\in W_{\Gamma}$, $a\in A_{v}$, $b\in A_{v^{\prime}}$, and $\mathbf{u}\in W_{\Gamma}$, $\xi\in\mathcal{H}_{\mathbf{u}}^{\circ}$.\\

\emph{About (1)}: It is clear that $a^{\dagger}b^{\dagger}=(Q_{v}aQ_{v}^{\perp})(Q_{v^{\prime}}aQ_{v^{\prime}}^{\perp})=0$ for $v=v^{\prime}$. If $(v,v^{\prime})\in E\Gamma$, we have 
\[
a^{\dagger}b^{\dagger}\xi=\begin{cases}
(a^{\circ}\xi_{v})\otimes(b^{\circ}\xi_{v^{\prime}})\otimes\xi), & \text{if }v^{\prime}\nleq\mathbf{u}\text{ and }v\nleq v^{\prime}\mathbf{u}\\
0, & \text{else}.
\end{cases}
\]
Note that by our assumption, $v^{\prime}\nleq\mathbf{u}$, $v\nleq v^{\prime}\mathbf{u}$ if and only if $v,v^{\prime}\nleq\mathbf{u}$, implying that $a^{\dagger}b^{\dagger}=b^{\dagger}a^{\dagger}$. This proves the second identity in (1).

\emph{About (2)}: We have $b^{\dagger}\mathfrak{d}(a)=(Q_{v}bQ_{v}^{\perp})(Q_{v}aQ_{v})=0$ for $v=v^{\prime}$. Furthermore, 
\begin{eqnarray*}
\mathfrak{d}(a)b^{\dagger}\xi & = & \mathfrak{d}(a)\left((b^{\circ}\xi_{v})\otimes\xi\right)\\
 & = & \left((ab^{\circ}-\omega_{v}(ab^{\circ})1)\xi_{v}\right)\otimes\xi+\omega_{v}(ab^{\circ})\xi\\
 & = & \left((ab-\omega_{v}(b)a-\omega_{v}(ab)1+\omega_{v}(a)\omega_{v}(b)1)\xi_{v}\right)\otimes\xi\\
 & = & \left((ab-\omega_{v}(b)a)^{\circ}\xi_{v}\right)\otimes\xi\\
 & = & (ab-\omega_{v}(b)a)^{\dagger}\xi
\end{eqnarray*}
if $v\nleq\mathbf{u}$, and 
\[
\mathfrak{d}(a)b^{\dagger}\xi=0=(ab-\omega_{v}(b)a)^{\dagger}\xi
\]
if $v\leq\mathbf{u}$. This implies $\mathfrak{d}(a)b^{\dagger}=(ab-\omega_{v}(b)a)^{\dagger}$ for $v=v^{\prime}$.

The third identity $\mathfrak{d}(a)b^{\dagger}\xi=0$ for $(v,v^{\prime})\in E\Gamma^{c}$ is obvious.

Lastly, for $(v,v^{\prime})\in E\Gamma$ one has $v^{\prime}\nleq\mathbf{u}$, $v\leq v^{\prime}\mathbf{u}$ if and only if $v\leq\mathbf{u}$, $v^{\prime}\nleq\mathbf{u}$, so that 
\[
\mathfrak{d}(a)b^{\dagger}\xi=\begin{cases}
\mathfrak{d}(a)((b^{\circ}\xi_{v^{\prime}})\otimes\xi), & \text{if }v^{\prime}\nleq\mathbf{u}\text{ and }v\leq\mathbf{u}\\
0, & \text{else}
\end{cases}=b^{\dagger}\mathfrak{d}(a)\xi.
\]
This implies $\mathfrak{d}(a)b^{\dagger}=b^{\dagger}\mathfrak{d}(a)$, as claimed.

\emph{About (3)}: For $v=v^{\prime}$ one has 
\[
a^{\dagger}(b^{\dagger})^{\ast}=Q_{v}aQ_{v}^{\perp}Q_{v}^{\perp}b^{\ast}Q_{v}=Q_{v}ab^{\ast}Q_{v}-(Q_{v}aQ_{v})(Q_{v}b^{\ast}Q_{v})=\mathfrak{d}(ab^{\ast})+\mathfrak{d}(a)\mathfrak{d}(b^{\ast}),
\]
as well as 
\begin{eqnarray*}
(a^{\dagger})^{\ast}b^{\dagger}\xi & = & \begin{cases}
(\omega_{v}(a^{\ast}b)-\overline{\omega_{v}(a)}\omega_{v}(b))\xi, & \text{if }v\nleq\mathbf{u}\\
0, & \text{else}
\end{cases}=(\omega_{v}(a^{\ast}b)-\overline{\omega_{v}(a)}\omega_{v}(b))Q_{v}^{\perp}\xi.
\end{eqnarray*}

For $(v,v^{\prime})\in E\Gamma^{c}$ the identity $Q_{v}Q_{v^{\prime}}=0$ holds so that $(a^{\dagger})^{\ast}b^{\dagger}=Q_{v}^{\dagger}a^{\ast}Q_{v}Q_{v^{\prime}}bQ_{v^{\prime}}^{\dagger}=0$.

Finally, for $(v,v^{\prime})\in E\Gamma$ one deduces $(a^{\dagger})^{\ast}b^{\dagger}=b^{\dagger}(a^{\dagger})^{\ast}$ as in (3).

\emph{About (4)}: The identites in (4) follow in the same way as the ones in (2) and (3).

\emph{About (5)}: We have that 
\[
Q_{\mathbf{w}}a^{\dagger}\xi=\begin{cases}
(a^{\circ}\xi_{v}\otimes\xi), & \text{if }v\leq\mathbf{u}\text{ and }\mathbf{w}\leq v\mathbf{u}\\
0, & \text{else}
\end{cases}=a^{\dagger}(v.Q_{\mathbf{w}})\xi,
\]
which implies $Q_{\mathbf{w}}a^{\dagger}=a^{\dagger}(v.Q_{\mathbf{w}})$. \end{proof}

\begin{proposition} \label{DensityStatement}
Let $\Gamma$ be a finite, undirected simplicial graph, and let $\mathbf{A} := (A_v)_{v \in V\Gamma}$ be a collection of unital C$^{\ast}$-algebras with GNS-faithful states $(\omega_v)_{v \in V\Gamma}$. Then the dense $\ast$-subalgebra $\mathcal{A} \subseteq \mathfrak{A}(\mathbf{A}, \Gamma)$ generated by $\mathbf{A}_\Gamma$ and the projections $(Q_v)_{v \in V\Gamma}$ coincides with the linear span
\begin{equation}
\label{eq:DensitySet}
\mathrm{Span} \left\{ (a_1^\dagger \cdots a_k^\dagger) \, d \, (b_1^\dagger \cdots b_l^\dagger)^* \ \middle| \
\begin{array}{l}
k, l \in \mathbb{N}, \ (u_1, \dots, u_k), (v_1, \dots, v_l) \in \mathcal{W}_{\mathrm{red}}, \\
a_i \in A_{u_i}^\circ, \ b_j \in A_{v_j}^\circ, \ d \in \mathcal{D}_0(\mathbf{A}, \Gamma)
\end{array}
\right\},
\end{equation}
where $\mathcal{D}_0(\mathbf{A}, \Gamma) \subseteq \mathcal{D}(\mathbf{A}, \Gamma)$ is the set consisting of $1$ and all finite products $\mathfrak{d}(c_1)\cdots \mathfrak{d}(c_n)$ with $c_i \in A_{w_i}$, where $\{w_1,\dots,w_n\} \subseteq V\Gamma$ forms a clique.
\end{proposition}

\begin{proof}
Let $X$ denote the set in ~\eqref{eq:DensitySet}. Given that for $v\in V\Gamma$ and $a\in A_{v}^{\circ}$,
\[
a=Q_{v}aQ_{v}+Q_{v}aQ_{v}^{\perp}+Q_{v}^{\perp}aQ_{v}+Q_{v}^{\perp}aQ_{v}^{\perp}=\mathfrak{d}(a)+a^{\dagger}+((a^{\ast})^{\dagger})^{\ast},
\]
it suffices to show that for every element of the form $x:=(a_{1}^{\dagger}\cdots a_{k}^{\dagger})d(b_{1}^{\dagger}\cdots b_{l}^{\dagger})^{\ast}$ with $k,l\in\mathbb{N}$, $(u_{1},\ldots,u_{k}),(v_{1},\ldots,v_{l})\in\mathcal{W}_{\text{red}}$, $a_{i}\in A_{v_{i}}^{\circ}$, $b_{j}\in A_{w_{j}}^{\circ}$ and $d\in\mathcal{D}_{0}(\mathbf{A},\Gamma)$ the products $yx$ and $xy$ with $y\in\{Q_{v},\mathfrak{d}(a),a^{\dagger},(a^{\dagger})^{\ast}\}$ are again contained in $X$. Without loss of generality we may restrict in our considerations to products of the form $yx$.

\begin{itemize}
\item \emph{Case 1}: Assume that $y=Q_{v}$ for some $v\in V\Gamma$. For $1\leq i \leq k$ we have by Lemma \ref{MainIdentities} that $Q_{v}a_{i}^{\dagger}=a_{i}^{\dagger}$ if $v=u_{i}$, $Q_{v}a_{i}^{\dagger}=0$ if $(v,u_{i})\in E\Gamma^{c}$, and $Q_{v}a_{i}^{\dagger}=a_{i}^{\dagger}Q_{v}$ if $(v,u_{i})\in E\Gamma$. By applying these identities repeatedly, we conclude $yx\in X$.
\item \emph{Case 2}: For $y=\mathfrak{d}(a)$ and  $1 \leq i \leq k$, Lemma \ref{MainIdentities} implies that $\mathfrak{d}(a)a_{i}^{\dagger}=((aa_{v})^{\circ})^{\dagger}$ if $v=v_{i}$, $\mathfrak{d}(a)a_{i}^{\dagger}=0$ if $(v,v_{i})\in E\Gamma^{c}$, and $\mathfrak{d}(a)a_{i}^{\dagger}=a_{i}^{\dagger}\mathfrak{d}(a)$ if $(v,v_{i})\in E\Gamma$. Again, a repeated application of these identities gives that $yx\in X$.
\item \emph{Case 3}: The case $y=a^{\dagger}$ is trivial.
\item \emph{Case 4}: Consider the case where $y=(a^{\dagger})^{\ast}$. For $1 \leq i \leq k $ we have by Lemma \ref{MainIdentities} that $(a^{\dagger})^{\ast}a_{i}^{\dagger}=\omega_{v}(a^{\ast}a_{i})Q_{v}^{\perp}$ if $v=v_{i}$, $(a^{\dagger})^{\ast}a_{i}^{\dagger}=0$ if $(v,v_{i})\in E\Gamma^{c}$, and $(a^{\dagger})^{\ast}a_{i}^{\dagger}=a_{i}^{\dagger}(a^{\dagger})^{\ast}$ if $(v,v_{i})\in E\Gamma^{c}$. Again, an inductive argument then implies in combination with Lemma \ref{MainIdentities} and the first and third case that $yx\in X$.
\end{itemize}
This finishes the proof. \end{proof}

In what follows, we refer to any non-zero operator of the form $(a_{1}^{\dagger}\cdots a_{k}^{\dagger})\, d\, (b_{1}^{\dagger}\cdots b_{l}^{\dagger})^{\ast}$, as in Proposition~\ref{DensityStatement}, with $k,l\in\mathbb{N}$, $(u_{1},\ldots,u_{k}), (v_{1},\ldots,v_{l}) \in \mathcal{W}_{\mathrm{red}}$, where $a_{i} \in A_{u_{i}}^{\circ}$ and $b_{j} \in A_{v_{j}}^{\circ}$ for $1 \leq i \leq k$, $1 \leq j \leq l$, and $d \in \mathcal{D}_{0}(\mathbf{A},\Gamma)$, as an \emph{elementary operator}. Denote by $\mathcal{E}(\mathbf{A},\Gamma)$ the collection of all such elementary operators. By Proposition~\ref{DensityStatement}, the span of $\mathcal{E}(\mathbf{A},\Gamma)$ is dense in $\mathfrak{A}(\mathbf{A},\Gamma)$.

A precise decomposition of reduced operators in $\mathbf{A}_{\Gamma}$ into linear combinations of elementary operators is provided in \cite[Proposition 2.6]{CaspersKlisseLarsen21}, where this structure plays a central role in the derivation of a Khintchine-type inequality for graph products.

We may canonically associate to each non-zero elementary operator a group element in the right-angled Coxeter group $W_{\Gamma}$, as the following lemma demonstrates.

\begin{lemma} \label{SignatureStatement}
Let $\Gamma$ be a finite, undirected, simplicial graph, and let $\mathbf{A} := (A_{v})_{v \in V\Gamma}$ be a collection of unital C$^*$-algebras equipped with GNS-faithful states $(\omega_{v})_{v \in V\Gamma}$. Let $x = (a_{1}^{\dagger} \cdots a_{k}^{\dagger})\, d\, (b_{1}^{\dagger} \cdots b_{l}^{\dagger})^{\ast} \in \mathcal{E}(\mathbf{A},\Gamma) \setminus \{0\}$ with $k,l \in \mathbb{N}$, $(u_{1}, \ldots, u_{k}), (v_{1}, \ldots, v_{l}) \in \mathcal{W}_{\mathrm{red}}$, $a_i \in A_{u_i}^{\circ}$, $b_j \in A_{v_j}^{\circ}$, and $d \in \mathcal{D}_{0}(\mathbf{A},\Gamma)$. Then the group element $(u_1 \cdots u_k)(v_1 \cdots v_l)^{-1} \in W_{\Gamma}$ depends only on the operator $x$ itself and is independent of the choice of the elements $a_i$, $b_j$, and $d$.
\end{lemma}

\begin{proof}
Let $0 \neq x \in \mathcal{E}(\mathbf{A},\Gamma)$ and suppose
\[
x = (a_1^{\dagger} \cdots a_k^{\dagger})\, d\, (b_1^{\dagger} \cdots b_l^{\dagger})^{\ast} = (c_1^{\dagger} \cdots c_m^{\dagger})\, e\, (f_1^{\dagger} \cdots f_n^{\dagger})^{\ast}
\]
be two such decompositions with $a_i \in A_{u_i}^{\circ}$, $b_j \in A_{v_j}^{\circ}$, $c_i \in A_{u_i'}^{\circ}$, $f_j \in A_{v_j'}^{\circ}$, $d,e \in \mathcal{D}_{0}(\mathbf{A},\Gamma)$, and corresponding reduced words $(u_1, \ldots, u_k), (v_1, \ldots, v_l), (u_1', \ldots, u_m'), (v_1', \ldots, v_n') \in \mathcal{W}_{\mathrm{red}}$.

Choose $\mathbf{w} \in W_{\Gamma}$ and $\xi \in \mathcal{H}_{\mathbf{w}}^{\circ}$ such that $x\xi \neq 0$. Then,
\[
x\xi \in \mathcal{H}_{(u_1 \cdots u_k)(v_1 \cdots v_l)^{-1} \mathbf{w}}^{\circ} \cap \mathcal{H}_{(u_1' \cdots u_m')(v_1' \cdots v_n')^{-1} \mathbf{w}}^{\circ}.
\]
By orthogonality of the subspaces $\mathcal{H}_{\mathbf{v}}^{\circ} \subseteq \mathcal{H}_{\Gamma}$ for distinct $\mathbf{v} \in W_{\Gamma}$, it follows that
\[
(u_1 \cdots u_k)(v_1 \cdots v_l)^{-1} = (u_1' \cdots u_m')(v_1' \cdots v_n')^{-1}.
\]
This finishes the proof.
\end{proof}

\begin{definition} \label{SignatureDefinition}
Let $\Gamma$ be a finite, undirected, simplicial graph and let $\mathbf{A}:=(A_{v})_{v\in V\Gamma}$ be a collection of unital C$^{\ast}$-algebras, equipped with GNS-faithful states $(\omega_{v})_{v\in V\Gamma}$. For an elementary operator $x \in \mathcal{E}(\mathbf{A},\Gamma)$, the group element
\[
\Sigma(x) := (u_1 \cdots u_k)(v_1 \cdots v_l)^{-1} \in W_{\Gamma}
\]
as defined in Lemma~\ref{SignatureStatement} is called the \emph{signature} of $x$. The map $\Sigma: \mathcal{E}(\mathbf{A},\Gamma) \to W_{\Gamma}$ is referred to as the \emph{signature map}.
\end{definition}

To the best of the author’s knowledge, the notion of a signature introduced in Definition~\ref{SignatureDefinition} is new. In combination with Proposition \ref{TopologicalFrenessImplication}, this concept will play a central role in the proof of Theorem~\ref{MaximalityTheorem}.

The next proposition characterizes when an elementary operator is diagonal in terms of its signature. The proof is a straightforward adaptation of Lemma~\ref{SignatureStatement}.

\begin{proposition} \label{DiagonalityCharacterization}
Let $\Gamma$ be a finite, undirected, simplicial graph and let $\mathbf{A}:=(A_{v})_{v\in V\Gamma}$ be a collection of unital C$^{\ast}$-algebras, equipped with GNS-faithful states $(\omega_{v})_{v\in V\Gamma}$. Then a non-zero elementary operator $x \in \mathcal{E}(\mathbf{A},\Gamma)$ is diagonal if and only if $\Sigma(x) = e$.
\end{proposition}

\begin{proof}
Suppose that $x:=(a_{1}^{\dagger}\cdots a_{k}^{\dagger})d(b_{1}^{\dagger}\cdots b_{l}^{\dagger})^{\ast} \in \mathcal{E}(\mathbf{A},\Gamma)$ with $k,l\in\mathbb{N}$, $(u_{1},\ldots,u_{k}),(v_{1},\ldots,v_{l})\in\mathcal{W}_{\text{red}}$, $a_{i}\in A_{v_{i}}^{\circ}$, $b_{j}\in A_{w_{j}}^{\circ}$ and $d\in\mathcal{D}_{0}(\mathbf{A},\Gamma)$.

For the ``only if'' direction assume that $x$ is diagonal. Then for any $\mathbf{w} \in W_{\Gamma}$ and $\xi \in \mathcal{H}_{\mathbf{w}}^{\circ}$ with $x\xi \neq 0$, we must have that $x\xi$ is contained in $\mathcal{H}_{\mathbf{w}}^{\circ} \cap \mathcal{H}_{(u_1 \cdots u_k)(v_1 \cdots v_l)^{-1} \mathbf{w}}^{\circ}$, implying $\Sigma(x) =(u_1 \cdots u_k)(v_1 \cdots v_l)^{-1} = e$.

Conversely, for the ``if'' direction, let $\Sigma(x) = e$. Then for all $\mathbf{w} \in W_{\Gamma}$ and $\xi \in \mathcal{H}_{\mathbf{w}}^{\circ}$, we have that $x\xi $ is contained in $\mathcal{H}_{(u_1 \cdots u_k)(v_1 \cdots v_l)^{-1} \mathbf{w}}^{\circ} = \mathcal{H}_{\mathbf{w}}^{\circ}$, showing that $x$ maps each subspace $\mathcal{H}_{\mathbf{w}}^{\circ}$ into itself, i.e., $x$ is diagonal.
\end{proof}

\vspace{3mm}


\subsection{Gauge Actions and Conditional Expectations}

Let $\Gamma$ be a finite, undirected, simplicial graph, and let $\mathbf{A} := (A_{v})_{v \in V\Gamma}$ be a collection of unital C$^*$-algebras, each equipped with a GNS-faithful state $\omega_v$. Given a reduced word $\mathbf{v} := (v_1, \dots, v_n) \in \mathcal{W}_{\text{red}}$ and a tuple $z := (z_v)_{v \in V\Gamma} \in \mathbb{C}^{V\Gamma}$, the product $z_{\mathbf{v}} := z_{v_1} \cdots z_{v_n} \in \mathbb{C}$ depends only on the shuffle equivalence class of $\mathbf{v}$; see, e.g., \cite[Chapter~17.1]{Davis08}. This allows us to define a strongly continuous family of unitaries $(U_z)_{z \in \mathbb{T}^{V\Gamma}} \subseteq \mathcal{B}(\mathcal{H}_\Gamma)$ via $U_z \Omega := \Omega$, and $U_z \xi := z_{\mathbf{w}} \xi$ for $\xi \in \mathcal{H}^{\circ}_{\mathbf{w}}$.

\begin{lemma} \label{GaugeActionLemma}
Let $\Gamma$ be a finite, undirected, simplicial graph and let $\mathbf{A}:=(A_{v})_{v\in V\Gamma}$ be a collection of unital C$^{\ast}$-algebras, equipped with GNS-faithful states $(\omega_{v})_{v\in V\Gamma}$. For every element $x := (a_1^{\dagger} \cdots a_k^{\dagger}) d (b_1^{\dagger} \cdots b_l^{\dagger})^* \in \mathcal{E}(\mathbf{A}, \Gamma)$ with $k, l \in \mathbb{N}$, $(u_1, \dots, u_k), (v_1, \dots, v_l) \in \mathcal{W}_{\text{\emph{red}}}$, $a_i \in A_{u_i}^\circ$, $b_j \in A_{v_j}^\circ$, $d \in \mathcal{D}_0(\mathbf{A}, \Gamma)$ and $z = (z_v)_{v \in V\Gamma} \in \mathbb{T}^{V\Gamma}$ we have
\[
U_z x U_z^* = \left(z_{u_1 \cdots u_k} z_{v_1 \cdots v_l}^{-1} \right) x.
\]
\end{lemma}

\begin{proof}
Let $\xi \in \mathcal{H}^{\circ}_{\mathbf{w}}$ for some $\mathbf{w} \in W_\Gamma$. Then, $U_z x U_z^* \xi = z_{\mathbf{w}}^{-1} U_z x \xi$. Since $x\xi \in \mathcal{H}_{(u_1 \cdots u_k)(v_1 \cdots v_l)^{-1} \mathbf{w}}^{\circ}$, we have
\[
U_z x U_z^* \xi = z_{\mathbf{w}}^{-1} z_{(u_1 \cdots u_k)(v_1 \cdots v_l)^{-1} \mathbf{w}} x\xi.
\]
If $x\xi\neq0$, we must further have $v_{1}\cdots v_{l}\leq\mathbf{w}$ and $u_{1}\cdots u_{k}\leq(u_{1}\cdots u_{k})(v_{1}\cdots v_{l})^{-1}\mathbf{w}$, so that 
\[
z_{\mathbf{w}}^{-1}z_{(u_{1}\cdots u_{k})(v_{1}\cdots v_{l})^{-1}\mathbf{w}}=z_{u_{1}\cdots u_{k}}z_{\mathbf{w}}^{-1}z_{(v_{1}\cdots v_{l})^{-1}\mathbf{w}}=z_{u_{1}\cdots u_{k}}z_{v_{1}\cdots v_{l}}^{-1}.
\]
Hence, $U_z x U_z^* \xi = (z_{u_1 \cdots u_k} z_{v_1 \cdots v_l}^{-1}) x \xi$, and the result follows.
\end{proof}

\begin{theorem} \label{GaugeActionTheorem}
Let $\Gamma$ be a finite, undirected, simplicial graph and let $\mathbf{A}:=(A_{v})_{v\in V\Gamma}$ be a collection of unital C$^{\ast}$-algebras, equipped with GNS-faithful states $(\omega_{v})_{v\in V\Gamma}$. Then:
\begin{enumerate}
\item Conjugation by the unitaries $(U_z)_{z \in \mathbb{T}^{V\Gamma}}$ induces a norm-continuous action $\alpha: \mathbb{T}^{V\Gamma} \curvearrowright \mathfrak{A}(\mathbf{A}, \Gamma)$.
\item There exists a faithful conditional expectation $\mathbb{E}: \mathfrak{A}(\mathbf{A}, \Gamma) \to \mathcal{D}(\mathbf{A}, \Gamma)$ satisfying $\omega_\Gamma \circ \mathbb{E} = \omega_\Gamma$.
\end{enumerate}
The action $\alpha$ is referred to as the \emph{gauge action}.
\end{theorem}

\begin{proof}
\emph{About (1)}: From Lemma~\ref{GaugeActionLemma}, it follows that $U_z$ conjugates elementary elements $x \in \mathcal{E}(\mathbf{A}, \Gamma)$ back into scalar multiples of themselves. Thus, $U_z \, (\text{Span}(\mathcal{E}(\mathbf{A}, \Gamma))) \, U_z^* \subseteq \text{Span}(\mathcal{E}(\mathbf{A}, \Gamma))$.
By Proposition~\ref{DensityStatement}, the action on $\mathcal{B}(\mathcal{H}_{\Gamma})$ then restricts to an action on $\mathfrak{A}(\mathbf{A}, \Gamma)$. Similarly, the identity in Lemma \ref{GaugeActionLemma} and the density of $\text{Span}(\mathcal{E}(\mathbf{A},\Gamma))$ in $\mathfrak{A}(\mathbf{A},\Gamma)$ imply that the induced action is norm continuous.

\emph{About (2)}: For each $\mathbf{w} \in W_\Gamma \setminus \{e\}$, let $p_{\mathbf{w}}$ be the orthogonal projection onto $\mathcal{H}_{\mathbf{w}}^{\circ}$ and $p_e$ the projection onto $\mathbb{C}\Omega$. Define $\mathbb{E}(x) := \sum_{\mathbf{w} \in W_\Gamma} p_{\mathbf{w}} x p_{\mathbf{w}}$, where the sum converges strongly. Then $\mathbb{E}$ is clearly linear, idempotent, and contractive:
\[
\|\mathbb{E}(x)\| = \sup_{\mathbf{w}} \|p_{\mathbf{w}} x p_{\mathbf{w}}\| \leq \|x\|.
\]
Therefore, by Tomiyama’s theorem (see, e.g., \cite[Theorem~1.5.10]{BrownOzawa08}), $\mathbb{E}$ is a conditional expectation onto its image. Moreover, for any element $x\in\mathcal{E}(\mathbf{A},\Gamma)$ of the form $x:=(a_{1}^{\dagger}\cdots a_{k}^{\dagger})d(b_{1}^{\dagger}\cdots b_{l}^{\dagger})^{\ast}$ with $k,l\in\mathbb{N}$, $(u_{1},\ldots,u_{k}),(v_{1},\ldots,v_{l})\in\mathcal{W}_{\text{red}}$, $a_{i}\in A_{v_{i}}^{\circ}$, $b_{j}\in A_{w_{j}}^{\circ}$, $d\in\mathcal{D}_{0}(\mathbf{A},\Gamma)$ we have that 
\[
\mathbb{E}(x)=\sum_{\mathbf{w}\in W_{\Gamma}}p_{\mathbf{w}}(a_{1}^{\dagger}\cdots a_{k}^{\dagger})d(b_{1}^{\dagger}\cdots b_{l}^{\dagger})^{\ast}p_{\mathbf{w}}=\left\{ \begin{array}{cc}
(a_{1}^{\dagger}\cdots a_{k}^{\dagger})d(b_{1}^{\dagger}\cdots b_{l}^{\dagger})^{\ast} & \text{, if }\Sigma(x)=e\\
0 & \text{, if }\Sigma(x)\neq e
\end{array}\right..
\]
By Proposition~\ref{DiagonalityCharacterization}, $\mathbb{E}(x) \in \mathcal{D}(\mathbf{A}, \Gamma)$. Proposition~\ref{DensityStatement} then implies that $\operatorname{im}(\mathbb{E}) = \mathcal{D}(\mathbf{A}, \Gamma)$.

To prove faithfulness, take any nonzero $x \in \mathfrak{A}(\mathbf{A}, \Gamma)$. Then there exists a unit vector $\xi \in \mathcal{H}^{\circ}_{\mathbf{w}}$, $\mathbf{w}\in W_{\Gamma}$ with $x\xi \neq 0$. Hence:
\[
\Vert\mathbb{E}(x^{\ast}x)\Vert\geq\Vert\mathbb{E}(x^{\ast}x)\xi\Vert \geq\langle\mathbb{E}(x^{\ast}x)\xi,\xi\rangle=\langle x^{\ast}x\xi,\xi\rangle=\Vert x\xi\Vert^{2}>0.
\]
Thus, $\mathbb{E}$ is faithful.
\end{proof}

\begin{corollary} \label{DiagonalCharacterization}
Let $\Gamma$ be a finite, undirected, simplicial graph, and let $\mathbf{A} := (A_{v})_{v \in V\Gamma}$ be a collection of unital C$^*$-algebras, each endowed with a GNS-faithful state $\omega_v$. Then the subspace
\begin{equation}
\mathrm{Span}\left\{ (a_{1}^{\dagger} \cdots a_{k}^{\dagger})\,d\,(b_{1}^{\dagger} \cdots b_{l}^{\dagger})^{*} \ \middle|\ k, l \in \mathbb{N},\ (v_1, \dots, v_l) \in \mathcal{W}_{\mathrm{red}},\ a_i, b_i \in A_{v_i}^{\circ},\ d \in \mathcal{D}_0(\mathbf{A}, \Gamma) \right\}, \label{eq:DiagonalDensitySet}
\end{equation}
is norm dense in $\mathcal{D}(\mathbf{A}, \Gamma)$.
\end{corollary}

\begin{proof}
By Proposition~\ref{DensityStatement} and Theorem~\ref{GaugeActionTheorem}, any element $x \in \mathcal{D}(\mathbf{A}, \Gamma)$ can be approximated in norm by finite sums of nonzero elements of the form $(a_{1}^{\dagger} \cdots a_{k}^{\dagger})\,d\,(b_{1}^{\dagger} \cdots b_{l}^{\dagger})^{*} \in \mathcal{E}(\mathbf{A}, \Gamma)$, where $k, l \in \mathbb{N}$, $(u_1, \dots, u_k), (v_1, \dots, v_l) \in \mathcal{W}_{\mathrm{red}}$, $a_i \in A_{v_i}^{\circ}$, $b_j \in A_{w_j}^{\circ}$, $d \in \mathcal{D}_0(\mathbf{A}, \Gamma)$, and such that the product $u_{1}\cdots u_{k}v_{l}\cdots v_{1}=(u_{1}\cdots u_{k})(v_{1}\cdots v_{l})^{-1}=e$ in $W_\Gamma$.

Since both $u_1 \cdots u_k$ and $v_1 \cdots v_l$ are reduced, there exist indices $1 \leq i \leq k$ and $1 \leq j \leq l$ with $u_i = v_j$ such that the cancellation of $u_i$ and $v_j$ preserves the words being reduced and $u_i$ commutes with all subsequent generators $u_{i+1}, \dots, u_k$ and $v_l, \dots, v_{j-1}$. By Lemma~\ref{MainIdentities}, we then have
\[
(a_{1}^{\dagger} \cdots a_{k}^{\dagger})\,d\,(b_{1}^{\dagger} \cdots b_{l}^{\dagger})^{*}
= (a_{1}^{\dagger} \cdots a_{i-1}^{\dagger} a_{i+1}^{\dagger} \cdots a_{k}^{\dagger}) \, \left(a_i^{\dagger} d (b_j^{\dagger})^*\right) \, (b_{1}^{\dagger} \cdots b_{j-1}^{\dagger} b_{j+1}^{\dagger} \cdots b_l^{\dagger})^*.
\]
Iterating this argument allows us to reduce to elements of the form~\eqref{eq:DiagonalDensitySet}, proving the desired density.
\end{proof}

\medskip

Inclusions of graphs naturally induce conditional expectations, as captured in the following theorem.

\begin{theorem} \label{ConditionalExpectation}
Let $\Gamma$ be a finite, undirected, simplicial graph, and let $\mathbf{A} := (A_v)_{v \in V\Gamma}$ be a collection of unital C$^*$-algebras equipped with GNS-faithful states $(\omega_v)_{v \in V\Gamma}$. Suppose $\Gamma_0 \subseteq \Gamma$ is an induced subgraph and define $\mathbf{A}|_{\Gamma_0} := (A_v)_{v \in V\Gamma_0}$. Then:
\begin{enumerate}
    \item There exists a $\ast$-embedding $\iota_{\Gamma,\Gamma_0} : \mathfrak{A}(\mathbf{A}|_{\Gamma_0}, \Gamma_0) \hookrightarrow \mathfrak{A}(\mathbf{A}, \Gamma)$ that canonically identifies the embedded copies of $\mathbf{A}_{\Gamma_{0}}$ and the projections $(Q_v)_{v \in V\Gamma_0}$ in both algebras.
    
    \item There exists a conditional expectation $\mathbb{E}_{\Gamma,\Gamma_0} : \mathfrak{A}(\mathbf{A}, \Gamma) \to \mathrm{im}(\iota_{\Gamma,\Gamma_0})$.
    
    \item For $x \in \mathcal{E}(\mathbf{A}, \Gamma)$, we have $\mathbb{E}_{\Gamma,\Gamma_0}(x) \neq 0$ if and only if $x \in \iota_{\Gamma,\Gamma_0}(\mathcal{E}(\mathbf{A}|_{\Gamma_0}, \Gamma_0))$.
\end{enumerate}
\end{theorem}

\begin{proof}
\emph{About (1)}: Let $(\mathbf{w}_{i})_{i\in I}\subseteq W_{\Gamma}$ be a family of group elements, including the identity, with $W_{\Gamma}=\bigsqcup_{i\in I}W_{\Gamma_{0}}\mathbf{w}_{i}$. By chosing the elements to have minimal length, we can make sure that $\mathcal{H}_{\mathbf{w}\mathbf{w}_{i}}^{\circ}\cong\mathcal{H}_{\mathbf{w}}^{\circ}\otimes\mathcal{H}_{\mathbf{w}_{i}}^{\circ}$ for all $\mathbf{w}\in W_{\Gamma_{0}}$, $i\in I$ (see, e.g., \cite[Lemma 4.3.1.]{Davis08}), so that we obtain an identification $\mathcal{H}_\Gamma \cong \bigoplus_{i \in I} \bigoplus_{\mathbf{w} \in W_{\Gamma_0}} \left(\mathcal{H}_{\mathbf{w}}^{\circ} \otimes \mathcal{H}_{\mathbf{w}_i}^{\circ}\right)$. In this picture, define $\iota_{\Gamma,\Gamma_{0}}:\mathfrak{A}(\mathbf{A}|_{\Gamma_{0}},\Gamma_{0})\hookrightarrow\mathfrak{A}(\mathbf{A},\Gamma)$ by $(\iota_{\Gamma,\Gamma_{0}}(x))(\xi\otimes\eta):=(x\xi)\otimes\eta$ for $x\in\mathfrak{A}(\mathbf{A}|_{\Gamma_{0}},\Gamma_{0})$, $i\in I$, $\mathbf{w}\in W_{\Gamma_{0}}$, $\xi\in\mathcal{H}_{\mathbf{w}}^{\circ}$, $\eta\in\mathcal{H}_{\mathbf{w}_{i}}$. It is straightforward to verify that this defines a faithful $*$-homomorphism preserving the canonical embeddings.

\emph{About (2)}: Denote by $S_{\Gamma,\Gamma_0}$ the canonical embedding of $\mathcal{H}_{\Gamma_0}$ into $\mathcal{H}_\Gamma$. For any elementary element $x:=(a_{1}^{\dagger}\cdots a_{k}^{\dagger})d(b_{1}^{\dagger}\cdots b_{l}^{\dagger})^{\ast}\in\mathcal{E}(\mathbf{A},\Gamma)$ with $k,l\in\mathbb{N}$, $(u_{1},\ldots,u_{k}),(v_{1},\ldots,v_{l})\in\mathcal{W}_{\text{red}}$, $a_{i}\in A_{v_{i}}^{\circ}$, $b_{j}\in A_{w_{j}}^{\circ}$, $d\in\mathcal{D}_{0}(\mathbf{A},\Gamma)$ with $S_{\Gamma,\Gamma_{0}}^{\ast}xS_{\Gamma,\Gamma_{0}}\neq0$ we observe that $u_{1},\ldots,u_{k},v_{1},\ldots,v_{l}\in V\Gamma_{0}$. The element $d\in\mathcal{D}_{0}(\mathbf{A},\Gamma)$ is either a multiple of $1$, or a finite product of the form $\mathfrak{d}(c_{1})\cdots\mathfrak{d}(c_{n})$ with $c_{i}\in A_{w_{i}}$, where $\{w_{1},\dots,w_{n}\}\subseteq V\Gamma$ forms a clique. By $S_{\Gamma,\Gamma_{0}}^{\ast}xS_{\Gamma,\Gamma_{0}}\neq0$ we must also have $w_{1},\ldots,w_{n}\in V\Gamma_{0}$, implying that $x$ can be viewed as an element in $\mathfrak{A}(\mathbf{A}|_{\Gamma_{0}},\Gamma_{0})$. Thus, the assignment $x \mapsto \iota_{\Gamma,\Gamma_0}(S_{\Gamma,\Gamma_0}^* x S_{\Gamma,\Gamma_0})$ extends linearly and continuously to a well-defined conditional expectation.

\emph{About (3)}: The equivalence follows directly from the considerations in (2).
\end{proof}

\medskip

The preceding theorem justifies viewing $\mathfrak{A}(\mathbf{A}|_{\Gamma_0}, \Gamma_0)$ canonically as an expected C$^*$-subalgebra of $\mathfrak{A}(\mathbf{A}, \Gamma)$ when $\Gamma_0$ is an induced subgraph. In light of this, we will henceforth suppress the embedding map $\iota_{\Gamma,\Gamma_0}$ in the notation.

The next proposition will be invoked repeatedly.  In particular, it enables us to decompose $\mathfrak{A}(\mathbf{A},\Gamma)$ as a tensor product of its elementary building blocks when $\Gamma$ happens to be a complete graph.

\begin{proposition} \label{TensorDecomposition}
Let $\Gamma$ be a finite, undirected, simplicial graph, and let $\mathbf{A} := (A_v)_{v \in V\Gamma}$ be a family of unital C$^*$-algebras with GNS-faithful states $(\omega_v)_{v \in V\Gamma}$. Suppose that $\Gamma$ is the disjoint union of two induced subgraphs $\Gamma_1$ and $\Gamma_2$ for which each vertex in $\Gamma_1$ is connected to each vertex in $\Gamma_2$, and define $\mathbf{A}_1 := (A_v)_{v \in V\Gamma_1}$ and $\mathbf{A}_2 := (A_v)_{v \in V\Gamma_2}$. Then
\[
\mathfrak{A}(\mathbf{A}, \Gamma) \cong \mathfrak{A}(\mathbf{A}_1, \Gamma_1) \otimes \mathfrak{A}(\mathbf{A}_2, \Gamma_2),
\]
with the identification $x \mapsto x \otimes 1$ for $x \in \mathfrak{A}(\mathbf{A}_1, \Gamma_1)$ and $y \mapsto 1 \otimes y$ for $y \in \mathfrak{A}(\mathbf{A}_2, \Gamma_2)$.
\end{proposition}

\begin{proof}
Define the unitary map $U : \mathcal{H}_{\Gamma_{1}}\otimes\mathcal{H}_{\Gamma_{2}}\rightarrow\mathbb{C}\Omega\oplus\bigoplus_{\mathbf{w}\in (W_{\Gamma_{1}} \times W_{\Gamma_{2}})\setminus \{e\}} \mathcal{H}_{\mathbf{w}}^{\circ}$ by setting
\[
\begin{aligned}
U(\Omega \otimes \Omega) &:= \Omega, \\
U(\Omega \otimes \xi) &:= \xi \quad \text{for } \xi \in \mathcal{H}_{\mathbf{w}}^{\circ},\ \mathbf{w} \in W_{\Gamma_1}\setminus\{e\} \\
U(\xi \otimes \Omega) &:= \xi \quad \text{for } \xi \in \mathcal{H}_{\mathbf{w}}^{\circ},\ \mathbf{w} \in W_{\Gamma_2}\setminus\{e\}, \\
U(\xi \otimes \eta) &:= \xi \otimes \eta \quad \text{for } \xi \in \mathcal{H}_{\mathbf{w}}^{\circ},\ \eta \in \mathcal{H}_{\mathbf{v}}^{\circ},\ \mathbf{w} \in W_{\Gamma_1}\setminus\{e\},\ \mathbf{v} \in W_{\Gamma_2}\setminus\{e\},
\end{aligned}
\]
and note that $\mathbb{C}\Omega\oplus\bigoplus_{\mathbf{w}\in (W_{\Gamma_{1}}\times W_{\Gamma_{2}})\setminus\{e\}}\mathcal{H}_{\mathbf{w}}^{\circ}$ identifies with $\mathcal{H}_{\Gamma}$, so that conjugation by $U$ implements an isomorphism $\mathcal{B}(\mathcal{H}_{\Gamma})\cong\mathcal{B}(\mathcal{H}_{\Gamma_{1}})\otimes\mathcal{B}(\mathcal{H}_{\Gamma_{2}})$. 
A routine verification shows that conjugation by $U$ yields:
\[
\begin{aligned}
U^* Q_v U &= Q_v \otimes 1 \quad \text{for all } v \in V\Gamma_1, \\
U^* Q_v U &= 1 \otimes Q_v \quad \text{for all } v \in V\Gamma_2, \\
U^* a U &= a \otimes 1 \quad \text{for } a \in A_v^{\circ},\ v \in V\Gamma_1, \\
U^* a U &= 1 \otimes a \quad \text{for } a \in A_v^{\circ},\ v \in V\Gamma_2.
\end{aligned}
\]
This establishes the claimed tensor decomposition.
\end{proof}

\vspace{3mm}


\subsection{Universality\label{subsec:Universality}}

The aim of this subsection is to demonstrate that the C$^{\ast}$-algebras constructed earlier satisfy the following useful universal property. In the special case of free products, analogous results were obtained by Hasegawa via an identification with Cuntz--Pimsner algebras \cite{Hasegawa19}; we are grateful to Pierre Fima for bringing this reference to our attention. In our setting, this approach is no longer applicable; however, the difficulty can be overcome by combining methods from \cite{Katsura04} with an inductive argument.

\begin{theorem} \label{UniversalProperty}
Let $\Gamma$ be a finite, undirected, simplicial graph, and let $\mathbf{A}:=(A_{v})_{v\in V\Gamma}$ be a collection of unital C$^{\ast}$-algebras equipped with GNS-faithful states $(\omega_{v})_{v\in V\Gamma}$. For $v_{0}\in V\Gamma$, define
\[
\mathbf{A}_{1}:=(A_{v})_{v\in V\text{\emph{Star}}(v_{0})},\quad
\mathbf{A}_{2}:=(A_{v})_{v\in V(\Gamma\setminus\{v_{0}\})},\quad
\mathbf{B}:=(A_{v})_{v\in V\text{\emph{Link}}(v_{0})}
\]
and consider $\mathfrak{A}_{1}:=\mathfrak{A}(\mathbf{A}_{1},\text{\emph{Star}}(v_{0}))$, $\mathfrak{A}_{2}:=\mathfrak{A}(\mathbf{A}_{2},\Gamma\setminus\{v_{0}\})$, and $B:=\mathfrak{A}(\mathbf{B},\text{\emph{Link}}(v_{0}))$. Then $\mathfrak{A}(\mathbf{A},\Gamma)$ satisfies the following universal property: every unital C$^{\ast}$-algebra $\overline{\mathfrak{A}}$ generated by the images of unital $\ast$-homomorphisms $\kappa_{1}:\mathfrak{A}_{1}\rightarrow\overline{\mathfrak{A}}$, $\kappa_{2}:\mathfrak{A}_{2}\rightarrow\overline{\mathfrak{A}}$ satisfying $\kappa_{1}|_{B}=\kappa_{2}|_{B}$ and
\[
\kappa_{1}(Q_{v_{0}})\kappa_{2}(Q_{v})=0 \quad \text{for all } v\in V(\Gamma\setminus\text{\emph{Star}}(v_{0}))
\]
admits a surjective $\ast$-homomorphism $\phi:\mathfrak{A}(\mathbf{A},\Gamma)\twoheadrightarrow\overline{\mathfrak{A}}$ with $\phi|_{\mathfrak{A}_{1}}=\kappa_{1}$ and $\phi|_{\mathfrak{A}_{2}}=\kappa_{2}$.
\end{theorem}

The proof of Theorem~\ref{UniversalProperty} requires some preparation.

As in the theorem, fix a finite, undirected, simplicial graph $\Gamma$, and let $\mathbf{A}:=(A_{v})_{v\in V\Gamma}$ be a collection of unital C$^{\ast}$-algebras equipped with GNS-faithful states $(\omega_{v})_{v\in V\Gamma}$. For $v_{0}\in V\Gamma$, define
\[
\mathbf{A}_1:=(A_{v})_{v\in V \text{Star}(v_0)},\quad
\mathbf{A}_2:=(A_{v})_{v\in V (\Gamma \setminus\{v_0\})},\quad
\mathbf{B}:=(A_{v})_{v\in V\text{Link}(v_{0})}
\]
and consider $\mathfrak{A}_{1}:=\mathfrak{A}(\mathbf{A}_{1},\text{{Star}}(v_{0}))$, $\mathfrak{A}_{2}:=\mathfrak{A}(\mathbf{A}_{2},\Gamma\setminus\{v_{0}\})$, and $B:=\mathfrak{A}(\mathbf{B},\text{{Link}}(v_{0}))$. Let $\overline{\mathfrak{A}}$ be the universal C$^{\ast}$-algebra generated by the images of unital $\ast$-homomorphisms $\kappa_{1}:\mathfrak{A}_{1}\rightarrow\overline{\mathfrak{A}}$, $\kappa_{2}:\mathfrak{A}_{2}\rightarrow\overline{\mathfrak{A}}$ satisfying $\kappa_{1}|_{B}=\kappa_{2}|_{B}$ and $\kappa_{1}(Q_{v_{0}})\kappa_{2}(Q_{v})=0$ for all $v\in V(\Gamma \setminus\text{Star}(v_{0}))$. It is clear that $\overline{\mathfrak{A}}$ exists and that both $\kappa_{1}$ and $\kappa_{2}$ are injective. For notational convenience, we write $\kappa_{v}(a):=\kappa_{1}(a)$ if $v\in V \text{Star}(v_0)$ and $a\in\mathfrak{A}_{1}$, and $\kappa_{v}(a):=\kappa_{2}(a)$ if $v\in V(\Gamma\setminus \{v_0\})$ and $a\in\mathfrak{A}_{2}$. Let $\phi:\overline{\mathfrak{A}}\twoheadrightarrow\mathfrak{A}(\mathbf{A},\Gamma)$ be the canonical surjective $\ast$-homomorphism provided by the universal property of $\mathfrak{A}$.

The strategy for proving Theorem~\ref{UniversalProperty} is as follows. First, in Lemma~\ref{IsomorphismTheorem2}, we show that $\phi$ restricts to an isomorphism between suitable C$^\ast$-subalgebras of $\overline{\mathfrak{A}}$ and $\mathfrak{A}(\mathbf{A}, \Gamma)$. In Lemma~\ref{FixedpointIdentification}, we identify these subalgebras with fixed point algebras for certain restrictions of the gauge actions. This identification, combined with orthogonality arguments (cf.\ Proposition~\ref{IsomorphismTheorem}) and an induction over the vertex set, yields the desired result.\\

Let us first complement the identities in Lemma \ref{MainIdentities} for products of operators in $\mathfrak{A}_{1}$ and $\mathfrak{A}_{2}$ with the following. The proof is straightforward and therefore omitted.

\begin{lemma} \label{UniversalMainIdentities}
Let $v \in V(\Gamma \setminus \emph{Star}(v_0))$, $a \in A_{v_0}$, and $b \in A_v$. Then the following identities hold:
\[
\kappa_1(\mathfrak{d}(a))\kappa_2(b^{\dagger}) = 
\kappa_1((a^{\dagger})^{*})\kappa_2(b^{\dagger}) = 
\kappa_1(\mathfrak{d}(a))\kappa_2(\mathfrak{d}(b)) = 0.
\]
\end{lemma}

Choose an enumeration $(s_1,\dots,s_L)$ of the vertices in $V\Gamma$, where $L := \#V\Gamma$. For each $1 \le i \le L$ and reduced word $(v_1,\dots,v_k) \in \mathcal{W}_{\mathrm{red}}$, define
\[
\#_i(v_1,\dots,v_k) := \#\{1 \le j \le k \mid v_j = s_i\},
\]
i.e., the number of times the letter $s_i$ appears in the word.

Now, for any $1 \le r \le L$ and multi-index $(n_1,\dots,n_r) \in \mathbb{N}^r$, define
\[
B_{n_1,\dots,n_r} = \overline{\mathrm{Span}}^{\Vert \cdot \Vert}\left\{
(a_1^{\dagger} \cdots a_k^{\dagger})\, d\, (b_1^{\dagger} \cdots b_l^{\dagger})^* \ \middle| \ 
\begin{array}{l}
k, l \in \mathbb{N},\ (u_1,\dots,u_k), (v_1,\dots,v_l) \in \mathcal{W}_{\mathrm{red}}, \\[3pt]
a_i \in A_{u_i}^{\circ},\ b_j \in A_{v_j}^{\circ},\ d \in \mathcal{D}_0(\mathbf{A},\Gamma), \\[3pt]
\#_i(u_1,\dots,u_k) = \#_i(v_1,\dots,v_l) = n_i \ \forall \ 1 \le i \le r
\end{array}
\right\}
\]
and define the corresponding subspace $\overline{B}_{n_1,\dots,n_r} \subseteq \overline{\mathfrak{A}}$ analogously. Moreover, for $n \in \mathbb{N}$, define $B_{\emptyset}^{\le n} := \sum_{i=0}^{n} B_i$ and $\overline{B}_{\emptyset}^{\le n} := \sum_{i=0}^{n} \overline{B}_i$, and for $1 \le r \le L-1$ define $B_{n_1,\dots,n_r}^{\le n} := \sum_{i=0}^{n} B_{n_1,\dots,n_r,i}$ and $\overline{B}_{n_1,\dots,n_r}^{\le n} := \sum_{i=0}^{n} \overline{B}_{n_1,\dots,n_r,i}$.

\begin{lemma} \label{MultiplicationLemma} Let $a_{1}\in A_{u_{1}}^{\circ},\ldots,a_{k}\in A_{u_{k}}^{\circ}$, $b_{1}\in A_{v_{1}}^{\circ},\ldots,b_{l}\in A_{v_{l}}^{\circ}$ with $(u_{1},\ldots,u_{k}),(v_{1},\ldots,v_{l})\in\mathcal{W}_{\text{\emph{red}}}$, and $1\leq i\le L$. Assume that $p\leq q$ with $p:=\#_{i}(v_{1},\ldots,v_{k})$, $q:=\#_{i}(w_{1},\ldots,w_{l})$. Then the following statements hold: \\

\begin{enumerate}
\item For every $d\in\text{\emph{Span}}(\mathcal{D}_{0}(\mathbf{A},\Gamma))$, the product $(a_{1}^{\dagger}\cdots a_{k}^{\dagger})^{\ast}d(b_{1}^{\dagger}\cdots b_{l}^{\dagger})$ can be expressed as a sum of products of the form $(\widetilde{a}_{i_{1}}^{\dagger}\cdots\widetilde{a}_{i_{m}}^{\dagger})^{\ast}e(\widetilde{b}_{j_{1}}^{\dagger}\cdots\widetilde{b}_{j_{n}}^{\dagger})$ for suitable distinct numbers 
\[
1\leq i_{1}<\ldots<i_{m}\leq k\;\text{ and }\;1\leq j_{1}<\ldots<j_{n}\leq l,
\]
operators 
\[
\widetilde{a}_{i_{1}}\in A_{u_{i_{1}}}^{\circ},\ldots,\widetilde{a}_{i_{m}}\in A_{u_{i_{m}}}^{\circ}\,,\;\widetilde{b}_{j_{1}}\in A_{v_{j_{1}}}^{\circ},\ldots,\widetilde{b}_{j_{n}}\in A_{v_{j_{n}}}^{\circ}\,,\;e\in\text{\emph{Span}}(\mathcal{D}_{0}(\mathbf{A},\Gamma))
\]
such that $m\leq k-p$, $n\leq l-p$, $(u_{i_{1}},\ldots,u_{i_{m}}),(v_{j_{1}},\ldots,v_{j_{n}})\in\mathcal{W}_{\text{\emph{red}}}$ with $\#_{i}(u_{i_{1}},\ldots,u_{i_{m}})=0$ and 
\[
\#_{j}(u_{1},\ldots,u_{k})-\#_{j}(u_{i_{1}},\ldots,u_{i_{m}})=\#_{j}(v_{1},\ldots,v_{l})-\#_{j}(v_{j_{1}},\ldots,v_{j_{n}})
\]
for all $1\leq j\leq L$.
\item For every $d$ in the span of the union of $\kappa_{1}(\mathcal{D}_{0}(\mathbf{A}_1,\text{\emph{Star}}(v_{0})))$ and $\kappa_{2}(\mathcal{D}_{0}(\mathbf{A}_2,\Gamma\setminus\{v_{0}\}))$, the product $(\kappa_{u_{1}}(a_{1}^{\dagger})\cdots\kappa_{u_{k}}(a_{k}^{\dagger}))^{\ast}d(\kappa_{v_{1}}(b_{1}^{\dagger})\cdots\kappa_{v_{l}}(b_{l}^{\dagger}))$ can be expressed as a sum of products of the form $(\kappa_{u_{i_{1}}}(\widetilde{a}_{i_{1}}^{\dagger})\cdots\kappa_{u_{i_{m}}}(\widetilde{a}_{i_{m}}^{\dagger}))e(\kappa_{v_{j_{1}}}(\widetilde{b}_{j_{1}}^{\dagger})\cdots\kappa_{v_{j_{n}}}(\widetilde{b}_{j_{n}}^{\dagger}))$ for suitable distinct numbers 
\[
1\leq i_{1}<\ldots<i_{m}\leq k\;\text{ and }\;1\leq j_{1}<\ldots<j_{n}\leq l,
\]
operators 
\[
\widetilde{a}_{i_{1}}\in A_{u_{i_{1}}}^{\circ},\ldots,\widetilde{a}_{i_{m}}\in A_{u_{i_{m}}}^{\circ}\,,\;\widetilde{b}_{j_{1}}\in A_{v_{j_{1}}}^{\circ},\ldots,\widetilde{b}_{j_{n}}\in A_{v_{j_{n}}}^{\circ}\,,
\]
and 
\[
e\in\text{\emph{Span}}(\kappa_{1}(\mathcal{D}_{0}(\mathbf{A}_1,\text{\emph{Star}}(v_{0})))\cup\kappa_{2}(\mathcal{D}_{0}(\mathbf{A}_2,\Gamma\setminus\{v_{0}\})))
\]
such that $m\leq k-p$, $n\leq l-p$, $(u_{i_{1}},\ldots,u_{i_{m}}),(v_{j_{1}},\ldots,v_{j_{n}})\in\mathcal{W}_{\text{\emph{red}}}$ with $\#_{i}(u_{i_{1}},\ldots,u_{i_{m}})=0$ and 
\[
\#_{j}(u_{1},\ldots,u_{k})-\#_{j}(u_{i_{1}},\ldots,u_{i_{m}})=\#_{j}(v_{1},\ldots,v_{l})-\#_{j}(v_{j_{1}},\ldots,v_{j_{n}})
\]
for all $1\leq j\leq L$. 
\end{enumerate}

In particular,
\[
(a_{1}^{\dagger}\cdots a_{k}^{\dagger})^{\ast}d(b_{1}^{\dagger}\cdots b_{l}^{\dagger})\in\text{\emph{Span}}(\mathcal{D}_{0}(\mathbf{A},\Gamma))
\]
and 
\[
(\kappa_{u_{1}}(a_{1}^{\dagger})\cdots\kappa_{u_{k}}(a_{k}^{\dagger}))^{\ast}d(\kappa_{v_{1}}(b_{1}^{\dagger})\cdots\kappa_{v_{l}}(b_{l}^{\dagger}))\in\text{\emph{Span}}(\kappa_{1}(\mathcal{D}_{0}(\mathbf{A}_1,\text{\emph{Star}}(v_{0})))\cup\kappa_{2}(\mathcal{D}_{0}(\mathbf{A}_2,\Gamma\setminus\{v_{0}\}))
\]
if $\#_{i}(u_{1},\ldots,u_{k})=\#_{i}(v_{1},\ldots,v_{k})$ for all $1\leq i\leq L$. \end{lemma}

\begin{proof} We prove only the first statement, as the second one follows analogously. We begin by establishing the following claim.\\

\emph{Claim}. Let $a_{1}\in A_{u_{1}}^{\circ},\ldots,a_{k}\in A_{u_{k}}^{\circ}$, $b_{1}\in A_{v_{1}}^{\circ},\ldots,b_{l}\in A_{v_{l}}^{\circ}$ and $d\in\text{Span}(\mathcal{D}_{0}(\mathbf{A},\Gamma))$ with $(u_{1},\ldots,u_{k}),(v_{1},\ldots,v_{l})\in\mathcal{W}_{\text{red}}$ and $1\leq i\le L$. Assume that $1\leq p\leq q$ with $p:=\#_{i}(u_{1},\ldots,u_{k})$, $q:=\#_{i}(v_{1},\ldots,v_{l})$. Then the product $(a_{1}^{\dagger}\cdots a_{k}^{\dagger})^{\ast}d(b_{1}^{\dagger}\cdots b_{l}^{\dagger})$ can be written as a sum of products of the form $(\widetilde{a}_{i_{1}}^{\dagger}\cdots\widetilde{a}_{i_{m}}^{\dagger})^{\ast}e(\widetilde{b}_{j_{1}}^{\dagger}\cdots\widetilde{b}_{j_{n}}^{\dagger})$ for suitable distinct numbers $1\leq i_{1}<\ldots<i_{m}\leq k$, $1\leq j_{1}<\ldots<j_{n}\leq l$, operators $\widetilde{a}_{i_{1}}\in A_{u_{i_{1}}}^{\circ},\ldots,\widetilde{a}_{i_{m}}\in A_{u_{i_{m}}}^{\circ}$, $\widetilde{b}_{j_{1}}\in A_{v_{j_{1}}}^{\circ},\ldots,\widetilde{b}_{j_{n}}\in A_{v_{j_{n}}}^{\circ}$, and $e\in\text{Span}(\mathcal{D}_{0}(\mathbf{A},\Gamma))$ such that $m\leq k-1$, $n\leq l-1$, $(u_{i_{1}},\ldots,u_{i_{m}}),(v_{j_{1}},\ldots,v_{j_{n}})\in\mathcal{W}_{\text{red}}$ with $\#_{i}(u_{i_{1}},\ldots,u_{i_{m}})=p-1$ and $\#_{j}(u_{1},\ldots,u_{k})-\#_{j}(u_{i_{1}},\ldots,u_{i_{m}})=\#_{j}(v_{1},\ldots,v_{l})-\#_{j}(v_{j_{1}},\ldots,v_{j_{n}})$ for all $1\leq j\leq L$.

\emph{Proof of the claim}. Without loss of generality, we may assume that $(a_{1}^{\dagger}\cdots a_{k}^{\dagger})^{\ast}d(b_{1}^{\dagger}\cdots b_{l}^{\dagger})\neq0$ and that $d=\mathfrak{d}(c_{1})\cdots\mathfrak{d}(c_{t})$ with $c_{1}\in A_{w_{1}}$, \ldots , $c_{t}\in A_{w_{t}}$, where $\{w_{1},\ldots,w_{t}\}\subseteq V\Gamma$ is forming a clique. Let $1\leq r\leq k$ and $1\leq r^{\prime}\leq l$ be the minimal integers with $u_{r}=s_{i}$ and $v_{r^{\prime}}=s_{i}$ so that 
\[
\left(a_{1}^{\dagger}\cdots a_{k}^{\dagger}\right)^{\ast}d\left(b_{1}^{\dagger}\cdots b_{l}^{\dagger}\right)=\left(a_{r+1}^{\dagger}\cdots a_{k}^{\dagger}\right)^{\ast}\left[\left(a_{1}^{\dagger}\cdots a_{r}^{\dagger}\right)^{\ast}\left(\mathfrak{d}(c_{1})\cdots\mathfrak{d}(c_{t})\right)\left(b_{1}^{\dagger}\cdots b_{r^{\prime}}^{\dagger}\right)\right]\left(b_{r^{\prime}+1}^{\dagger}\cdots b_{l}^{\dagger}\right).
\]
By Lemma \ref{MainIdentities} and the assumption $(a_{1}^{\dagger}\cdots a_{k}^{\dagger})^{\ast}d(b_{1}^{\dagger}\cdots b_{l}^{\dagger})\neq0$, the operator $(a_{1}^{\dagger})^{\ast}$ either commutes with $\mathfrak{d}(c_{1}),\ldots,\mathfrak{d}(c_{t})$, or there exists $1\leq j\leq t$ with $w_{j}=u_{1}$ so that 
\[
(a_{1}^{\dagger})^{\ast}\left(\mathfrak{d}(c_{1})\cdots\mathfrak{d}(c_{t})\right)=\left(\left(c_{j}^{\ast}a_{1}-\omega_{u_{1}}(c_{j}^{\ast}a_{1})1\right)^{\dagger}\right)^{\ast}\left(\mathfrak{d}(c_{1})\cdots\mathfrak{d}(c_{j-1})\mathfrak{d}(c_{j+1})\cdots\mathfrak{d}(c_{t})\right).
\]
It clearly suffices to consider only the first case, in which
\[
(a_{1}^{\dagger}\cdots a_{r}^{\dagger})^{\ast}(\mathfrak{d}(c_{1})\cdots\mathfrak{d}(c_{t}))=(a_{2}^{\dagger}\cdots a_{r}^{\dagger})^{\ast}(\mathfrak{d}(c_{1})\cdots\mathfrak{d}(c_{t}))(a_{1}^{\dagger})^{\ast}.
\]
Now, either $(a_{1}^{\dagger})^{\ast}$ and $b_{1}^{\dagger}$ commute with each other, or $u_{1}=v_{1}$. In the second case, Lemma \ref{MainIdentities} implies $(a_{1}^{\dagger})^{\ast}b_{1}^{\dagger}=\omega_{u_{1}}(a_{1}^{\ast}b_{1})Q_{u_{1}}^{\perp}$ so we may again reduce to the case where the operators commute.  Proceeding in this manner, we may assume that $(a_{1}^{\dagger})^{\ast}$ commutes with all operators $\mathfrak{d}(c_{1}),\ldots,\mathfrak{d}(c_{t})$ and $b_{1}^{\dagger},\ldots,b_{r^{\prime}-1}^{\dagger}$ so that $(a_{1}^{\dagger})^{\ast}d(b_{1}^{\dagger}\cdots b_{r^{\prime}-1}^{\dagger})=(\mathfrak{d}(c_{1})\cdots\mathfrak{d}(c_{t}))(b_{1}^{\dagger}\cdots b_{r^{\prime}-1}^{\dagger})(a_{1}^{\dagger})^{\ast}$. By repeating this argument, we conclude that it suffices to assume that all operators $(a_{1}^{\dagger})^{\ast},\ldots,(a_{r-1}^{\dagger})^{\ast}$ commute with $\mathfrak{d}(c_{1}),\ldots,\mathfrak{d}(c_{t})$ and $b_{1}^{\dagger},\ldots,b_{r^{\prime}-1}^{\dagger}$, implying that 
\begin{eqnarray*}
 &  & \left(a_{1}^{\dagger}\cdots a_{k}^{\dagger}\right)^{\ast}d\left(b_{1}^{\dagger}\cdots b_{k}^{\dagger}\right)\\
 & = & \left(a_{r+1}^{\dagger}\cdots a_{k}^{\dagger}\right)^{\ast}\times\left[(a_{r}^{\dagger})^{\ast}\left(\mathfrak{d}(c_{1})\cdots\mathfrak{d}(c_{t})\right)\left(b_{1}^{\dagger}\cdots b_{r^{\prime}-1}^{\dagger}\right)\left(a_{1}^{\dagger}\cdots a_{r-1}^{\dagger}\right)^{\ast}b_{r^{\prime}}\right]\times\left(b_{r^{\prime}+1}^{\dagger}\cdots b_{k}^{\dagger}\right).
\end{eqnarray*}
Note that by the choice of $r$ and $r'$, none of the operators $a_{1}^{\dagger}, \ldots, a_{r-1}^{\dagger}, b_{1}^{\dagger}, \ldots, b_{r'-1}^{\dagger}$ lies in $A_{s_{i}}^{\circ}$. Assuming further that $s_{i} \notin \{w_{1}, \ldots, w_{t}\}$, we obtain
\begin{eqnarray*}
 &  & \left(a_{1}^{\dagger}\cdots a_{k}^{\dagger}\right)^{\ast}d\left(b_{1}^{\dagger}\cdots b_{k}^{\dagger}\right)\\
 & = & \left(a_{r+1}^{\dagger}\cdots a_{k}^{\dagger}\right)^{\ast}\left[\left(\mathfrak{d}(c_{1})\cdots\mathfrak{d}(c_{t})\right)\left(b_{1}^{\dagger}\cdots b_{r^{\prime}-1}^{\dagger}\right)\right]\left((a_{r}^{\dagger})^{\ast}b_{r^{\prime}}\right)\left(a_{1}^{\dagger}\cdots a_{r-1}^{\dagger}\right)^{\ast}\left(b_{r^{\prime}+1}^{\dagger}\cdots b_{k}^{\dagger}\right)\\
 & = & \left(a_{r+1}^{\dagger}\cdots a_{k}^{\dagger}\right)^{\ast}\left[\left(\mathfrak{d}(c_{1})\cdots\mathfrak{d}(c_{t})\right)\left(b_{1}^{\dagger}\cdots b_{r^{\prime}-1}^{\dagger}\right)\right]\left(\omega_{s_{i}}(a_{r}^{\ast}b_{r^{\prime}})Q_{s_{i}}^{\perp}\right)\left(a_{1}^{\dagger}\cdots a_{r-1}^{\dagger}\right)^{\ast}\left(b_{r^{\prime}+1}^{\dagger}\cdots b_{k}^{\dagger}\right)\\
 & = & \omega_{s_{i}}(a_{r}^{\ast}b_{r^{\prime}})\left(a_{1}^{\dagger}\cdots a_{r-1}^{\dagger}a_{r+1}^{\dagger}\cdots a_{k}^{\dagger}\right)^{\ast}\left(\mathfrak{d}(c_{1})\cdots\mathfrak{d}(c_{t})Q_{s_{i}}^{\perp}\right)\left(b_{1}^{\dagger}\cdots b_{r^{\prime}-1}^{\dagger}b_{r^{\prime}+1}^{\dagger}\cdots b_{k}^{\dagger}\right).
\end{eqnarray*}
Since $\#_{i}(u_{1},\ldots,\widehat{u_{r}},\ldots,u_{k})=p-1$, $\#_{i}(v_{1},\ldots,\widehat{v_{r^{\prime}}},\ldots,v_{k})=q-1$ and $\mathfrak{d}(c_{1})\cdots\mathfrak{d}(c_{t})Q_{s_{i}}^{\perp}\in\text{Span}(\mathcal{D}_{0}(\mathbf{A},\Gamma))$, we conclude that $(a_{1}^{\dagger}\cdots a_{k}^{\dagger})^{\ast}d(b_{1}^{\dagger}\cdots b_{l}^{\dagger})$ can indeed be expressed in the desired form.\\

It is now clear that repeated application of the above claim yields the first statement of the lemma. The final statement follows analogously. \end{proof}

\begin{proposition} \label{IdealConstruction} Let $1\leq k\leq L-1$, $(n_{1},\ldots,n_{r})\in\mathbb{N}^{k}$, and $n\in\mathbb{N}$. Then:
\begin{enumerate}
\item $B_{\emptyset}^{\leq n}=\sum_{i=0}^{n}B_{i}$ is a C$^{\ast}$-algebra containing $B_{n}$ as a closed two-sided ideal.
\item $\overline{B}_{\emptyset}^{\leq n}=\sum_{i=0}^{n}\overline{B}_{i}$ is a C$^{\ast}$-algebra containing $\overline{B}_{n}$ as a closed two-sided ideal.
\item $B_{n_{1},\ldots,n_{r}}^{\leq n}=\sum_{i=0}^{n}B_{n_{1},\ldots,n_{r},i}$ is a C$^{\ast}$-algebra containing $B_{n_{1},\ldots,n_{r},n}$ as a closed two-sided ideal.
\item $\overline{B}_{n_{1},\ldots,n_{r}}^{\leq n}=\sum_{i=0}^{n}\overline{B}_{n_{1},\ldots,n_{r},i}$ is a C$^{\ast}$-algebra containing $\overline{B}_{n_{1},\ldots,n_{r},n}$ as a closed two-sided ideal.
\end{enumerate}
\end{proposition}
\begin{proof} We prove only the third statement, as the others follow analogously.

By induction, it suffices to show that $B_{n_{1},\ldots,n_{r},n}$ is a C$^{\ast}$-algebra with
\[
B_{n_{1},\ldots,n_{r}}^{\leq(n-1)}B_{n_{1},\ldots,n_{r},n}B_{n_{1},\ldots,n_{r}}^{\leq(n-1)}\subseteq B_{n_{1},\ldots,n_{r},n}
\]
for every $n\in\mathbb{N}$. That $B_{n_{1},\ldots,n_{r},n}$ is closed under multiplication follows from Lemma \ref{MultiplicationLemma} and the same reasoning as in the proof of Proposition \ref{DensityStatement}. Since $B_{n_{1},\ldots,n_{r},n}$ is closed and $\ast$-invariant, it must be a C$^{\ast}$-algebra. For the inclusion, it suffices to show that $B_{n_{1},\ldots,n_{r},m}B_{n_{1},\ldots,n_{r},n}\subseteq B_{n_{1},\ldots,n_{r},n}$ for all $m,n\in\mathbb{N}$ with $m\leq n$. Again, this follows from Lemma \ref{MultiplicationLemma} and the same argument used in the proof of Proposition \ref{DensityStatement}. \end{proof}

\begin{proposition} \label{IsomorphismTheorem}
For every tuple $(n_{1},\ldots,n_{L})\in\mathbb{N}^{L}$, the $\ast$-homomorphism $\phi:\overline{\mathfrak{A}}\twoheadrightarrow\mathfrak{A}(\mathbf{A},\Gamma)$ restricts to an isomorphism of the C$^{\ast}$-algebras $\overline{\text{\emph{Span}}}^{\Vert \cdot \Vert}(\kappa_{1}(\mathcal{D}_{0}(\mathbf{A}_1,\text{\emph{Star}}(v_{0})))\cup\kappa_{2}(\mathcal{D}_{0}(\mathbf{A}_2,\Gamma\setminus\{v_{0}\})))$ and $\overline{\text{\emph{Span}}}^{\Vert \cdot \Vert}(\mathcal{D}_{0}(\mathbf{A},\Gamma))$.
\end{proposition}

\begin{proof}
First, note that the linear space $\overline{\text{Span}}^{\Vert \cdot \Vert}(\kappa_{1}(\mathcal{D}_{0}(\mathbf{A}_1,\text{Star}(v_{0})))\cup\kappa_{2}(\mathcal{D}_{0}(\mathbf{A}_2,\Gamma\setminus\{v_{0}\})))$ is indeed a C$^{\ast}$-algebra. By the orthogonality of $\kappa_{1}(Q_{v_{0}})$ and all projections $\kappa_{2}(Q_{v})$ for $v\in V(\Gamma \setminus \text{Star}(v_0))$, we may express every element $x$ in this C$^{\ast}$-algebra as a sum of the form
\begin{eqnarray}
\nonumber x &=& \left(\kappa_{1}(Q_{v_{0}})+\kappa_{2}(P)+Q^{\perp}\right)x\left(\kappa_{1}(Q_{v_{0}})+\kappa_{2}(P)+Q^{\perp}\right) \\
\nonumber &=& \kappa_{1}(Q_{v_{0}})x+\kappa_{2}(P)x+Q^{\perp}x,
\end{eqnarray}
where $P := \bigvee_{v\in V(\Gamma \setminus \text{Star}(v_{0}))}Q_{v}\in\mathfrak{A}_{2}$ and $Q := \kappa_{1}(Q_{v_{0}})+\kappa_{2}(P)$. Note that $\kappa_{1}(Q_{v_{0}})x\in\text{im}(\kappa_{1})$, $\kappa_{2}(P)x\in\text{im}(\kappa_{2})$, and
\[
Q^{\perp}x \in Q^{\perp}\kappa_{1}\left(\overline{\text{Span}}^{\Vert \cdot \Vert}\left(\mathcal{D}_{0}(\mathbf{B},\text{Link}(v_{0}))\right)\right).
\]

By construction, all summands in the expression above have orthogonal support and ranges, so that
\[
\Vert x\Vert = \max\{\Vert\kappa_{1}(Q_{v_{0}})x\Vert,\Vert\kappa_{2}(P)x\Vert,\Vert Q^{\perp}x\Vert\}.
\]
Since $\phi$ is isometric on $\text{im}(\kappa_{1})$ and $\text{im}(\kappa_{2})$, we have $\Vert\kappa_{1}(Q_{v_{0}})x\Vert = \Vert Q_{v_{0}}\phi(x)\Vert$ and $\Vert\kappa_{2}(P)x\Vert = \Vert P\phi(x)\Vert$. Furthermore, since $Q^{\perp}x = Q^{\perp}\kappa_{1}(y)$ for some $y \in \overline{\text{Span}}^{\Vert \cdot \Vert}(\mathcal{D}_{0}(\mathbf{B},\text{Link}(v_{0})))$, the inequality $\Vert Q^{\perp}x\Vert \leq \Vert y\Vert = \Vert(Q_{v_{0}}+P)^{\perp}y\Vert$ holds. It follows that
\begin{eqnarray}
\nonumber \left\Vert x\right\Vert &\leq& \max\left\{ \Vert Q_{v_{0}}\phi(x)\Vert,\Vert P\phi(x)\Vert,\Vert(Q_{v_{0}}+P)^{\perp}y\Vert\right\} \\
\nonumber &=& \left\Vert Q_{v_{0}}\phi(x) + P\phi(x) + (Q_{v_{0}}+P)^{\perp}y \right\Vert \\
\nonumber &=& \Vert\phi(x)\Vert,
\end{eqnarray}
which shows that $\phi$ indeed restricts to an isomorphism, as claimed.
\end{proof}

Given a tuple $(n_{1},\ldots,n_{L}) \in \mathbb{N}^{L}$, consider the linear subspace
\begin{eqnarray} \label{HilbertModule}
\mathfrak{X}_{n_{1}, \dots, n_{L}} := \overline{\mathrm{Span}}^{\Vert \cdot \Vert} \left\{ \left. (a_{1}^{\dagger} \cdots a_{k}^{\dagger})\, d \,\right|\, 
\begin{array}{l}
k \in \mathbb{N},\quad (u_{1}, \dots, u_{k}) \in \mathcal{W}_{\mathrm{red}}, \\[4pt]
a_{i} \in A^{\circ}_{u_{i}},\quad d \in \mathcal{D}_{0}(\mathbf{A}, \Gamma), \\[4pt]
\#_{i}(u_{1}, \dots, u_{k}) = n_{i} \quad \forall\, 1 \leq i \leq L
\end{array}
\right\}.
\end{eqnarray}
By Lemma~\ref{MultiplicationLemma}, we have that $\xi^{\ast}\eta \in \mathcal{D}$ for all $\xi, \eta \in \mathfrak{X}_{n_{1}, \ldots, n_{L}}$, where $\mathcal{D} := \overline{\mathrm{Span}}(\mathcal{D}_{0}(\mathbf{A}, \Gamma))$. Furthermore, by Lemma~\ref{MainIdentities}, we have $d\xi, \xi d \in \mathfrak{X}_{n_{1}, \ldots, n_{L}}$ for all $\xi \in \mathfrak{X}_{n_{1}, \ldots, n_{L}}$ and $d \in \mathcal{D}$.

It follows that the inner product $\langle \cdot, \cdot \rangle$ on $\mathfrak{X}_{n_{1}, \ldots, n_{L}}$, given by $\langle \xi, \eta \rangle := \xi^{\ast} \eta$  for $\xi, \eta \in \mathfrak{X}_{n_{1}, \ldots, n_{L}} \subseteq \mathfrak{A}(\mathbf{A}, \Gamma)$ turns $\mathfrak{X}_{n_{1}, \ldots, n_{L}}$ into a Hilbert $\mathcal{D}$-module, where $\mathcal{D}$ acts via right multiplication. For further background on Hilbert modules, we refer the reader to~\cite{Lance95}.

For $\xi, \eta \in \mathfrak{X}_{n_{1}, \ldots, n_{L}}$, let $\theta_{\xi, \eta} \in \mathcal{B}(\mathfrak{X}_{n_{1}, \ldots, n_{L}})$ be the corresponding rank-one operator, defined by $\theta_{\xi, \eta}(\zeta) := \xi \langle \eta, \zeta \rangle$ for $\zeta \in \mathfrak{X}_{n_{1}, \ldots, n_{L}}$.
We denote the compact operators on $\mathfrak{X}_{n_{1}, \ldots, n_{L}}$ by
\[
\mathcal{K}(\mathfrak{X}_{n_{1}, \ldots, n_{L}}) := \overline{\mathrm{Span}}^{\Vert \cdot \Vert} \left\{ \theta_{\xi, \eta} \,\middle|\, \xi, \eta \in \mathfrak{X}_{n_{1}, \ldots, n_{L}} \right\}. 
\]
Then, by~\cite[Proposition 4.6.3]{BrownOzawa08}, we obtain a $\ast$-isomorphism $\mathcal{K}(\mathfrak{X}_{n_{1}, \ldots, n_{L}}) \cong B_{n_{1}, \ldots, n_{L}} \quad \text{via } \theta_{\xi, \eta} \mapsto \xi \eta^{\ast}$. By the same argument and invoking Proposition~\ref{IsomorphismTheorem}, we also obtain $\mathcal{K}(\mathfrak{X}_{n_{1}, \ldots, n_{L}}) \cong \overline{B}_{n_{1}, \ldots, n_{L}}$. We may therefore conclude the following implication.

\begin{lemma} \label{IsomorphismTheorem2} For every tuple $(n_{1},\ldots,n_{L})\in\mathbb{N}^{L}$ the $\ast$-homomorphism $\phi:\overline{\mathfrak{A}}\twoheadrightarrow\mathfrak{A}(\mathbf{A},\Gamma)$ restricts to an isomorphism $\overline{B}_{n_{1},\ldots,n_{L}}\cong B_{n_{1},\ldots,n_{L}}$. \end{lemma}

Recall that in the setting of Theorem~\ref{UniversalProperty}, the gauge action $\alpha : \mathbb{T}^{V\Gamma} \curvearrowright \mathfrak{A}(\mathbf{A}, \Gamma)$ satisfies $\alpha_{z}(x) = (z_{u_{1} \ldots u_{k}} z_{v_{1} \ldots v_{l}}^{-1}) x$ for every $x := (a_{1}^{\dagger} \cdots a_{k}^{\dagger})\, d\, (b_{1}^{\dagger} \cdots b_{l}^{\dagger})^{*} \in \mathcal{E}(\mathbf{A}, \Gamma)$ with $k, l \in \mathbb{N}$, $(u_{1}, \dots, u_{k}), (v_{1}, \dots, v_{l}) \in \mathcal{W}_{\mathrm{red}}$, $a_{i} \in A_{u_{i}}^{\circ}$, $b_{j} \in A_{v_{j}}^{\circ}$, and $d \in \mathcal{D}_{0}(\mathbf{A}, \Gamma)$.

Denote the restrictions of $\alpha$ to $\mathfrak{A}_{1}$ and $\mathfrak{A}_{2}$ by $\alpha_{1}$ and $\alpha_{2}$, respectively. Since $(\alpha_{1})_{z}|_{B} = (\alpha_{2})_{z}|_{B}$ for $z \in \mathbb{T}^{V\Gamma}$, and $(\alpha_{1})_{z}(Q_{v_{0}})(\alpha_{2})_{z}(Q_{v}) = 0$ for all $v \in V(\Gamma \setminus \mathrm{Star}(v_{0}))$, the universal property of $\overline{\mathfrak{A}}$ provides a surjective $\ast$-homomorphism $\beta_{z} : \overline{\mathfrak{A}} \twoheadrightarrow \overline{\mathfrak{A}}$ extending both $(\alpha_{1})_{z}$ and $(\alpha_{2})_{z}$. Note that $\beta_{\overline{z}} \circ \beta_{z} = \mathrm{id}_{\mathfrak{A}}$ for every $z \in \mathbb{T}^{V\Gamma}$, so that $\beta_{z} \in \mathrm{Aut}(\overline{\mathfrak{A}})$. We therefore obtain an action $\beta : \mathbb{T}^{V\Gamma} \curvearrowright \overline{\mathfrak{A}}$ such that $\phi \circ \beta_{z} = \alpha_{z} \circ \phi$ for every $z \in \mathbb{T}^{V\Gamma}$.

For every $1 \leq m \leq L = \# V\Gamma$, we canonically embed $\mathbb{T}^{m}$ into $\mathbb{T}^{V\Gamma}$ via $(z_{1}, \ldots, z_{m}) \mapsto (z_{1}, \ldots, z_{m}, 1, \ldots, 1)$, and denote the corresponding restricted actions $\mathbb{T}^{m} \curvearrowright \mathfrak{A}(\mathbf{A}, \Gamma)$ and $\mathbb{T}^{m} \curvearrowright \overline{\mathfrak{A}}$ by $\alpha_{m}$ and $\beta_{m}$, respectively. Note that
\[
(\alpha_{m})_{z}(B_{n_{1}, \ldots, n_{r}}) \subseteq B_{n_{1}, \ldots, n_{r}} \quad \text{and} \quad (\beta_{m})_{z}(\overline{B}_{n_{1}, \ldots, n_{r}}) \subseteq \overline{B}_{n_{1}, \ldots, n_{r}}
\]
for all $1 \leq r, m \leq L$, $(n_{1}, \ldots, n_{r}) \in \mathbb{N}^{r}$, and $z \in \mathbb{T}^{m}$.

\begin{lemma} \label{FixedpointIdentification}
For every $2 \leq r \leq L$ and $(n_{1}, \ldots, n_{r-1}) \in \mathbb{N}^{r-1}$, the fixed-point algebra
\[
B_{n_{1}, \ldots, n_{r-1}}^{\alpha_{r}} := \{x \in B_{n_{1}, \ldots, n_{r-1}} \mid (\alpha_{r})_{z}(x) = x \text{ for all } z \in \mathbb{T}^{r} \}
\]
coincides with the norm closure of $\bigcup_{n=0}^{\infty} B_{n_{1}, \ldots, n_{r-1}}^{\leq n}$, and the fixed-point algebra
\[
\overline{B}_{n_{1}, \ldots, n_{r-1}}^{\beta_{r}} := \{x \in \overline{B}_{n_{1}, \ldots, n_{r-1}} \mid (\beta_{r})_{z}(x) = x \text{ for all } z \in \mathbb{T}^{r} \}
\]
coincides with the norm closure of $\bigcup_{n=0}^{\infty} \overline{B}_{n_{1}, \ldots, n_{r-1}}^{\leq n}$.

Similarly,
\[
\mathfrak{A}(\mathbf{A}, \Gamma)^{\alpha_{1}} = \overline{\bigcup_{n=0}^{\infty} B_{\emptyset}^{\leq n}}^{\| \cdot \|}, \quad \text{and} \quad \overline{\mathfrak{A}}^{\beta_{1}} = \overline{\bigcup_{n=0}^{\infty} \overline{B}_{\emptyset}^{\leq n}}^{\| \cdot \|}.
\]
\end{lemma}

\begin{proof}
Integration over the restriction of $\alpha_{r}$ to $B_{n_{1}, \ldots, n_{r-1}}$ induces a conditional expectation $\mathbb{E}_{r}$ onto the fixed-point algebra $B_{n_{1}, \ldots, n_{r-1}}^{\alpha_{r}}$. For every element $x := (a_{1}^{\dagger} \cdots a_{k}^{\dagger})\, d\, (b_{1}^{\dagger} \cdots b_{l}^{\dagger})^{*} \in \mathcal{E}(\mathbf{A}, \Gamma)$ with $k, l \in \mathbb{N}$, $(u_{1}, \dots, u_{k}), (v_{1}, \dots, v_{l}) \in \mathcal{W}_{\mathrm{red}}$, $a_{i} \in A_{u_{i}}^{\circ}$, $b_{j} \in A_{v_{j}}^{\circ}$, $d \in \mathcal{D}_{0}(\mathbf{A}, \Gamma)$, and $\#_{i}(u_{1}, \ldots, u_{k}) = \#_{i}(v_{1}, \ldots, v_{l}) = n_{i}$ for $1 \leq i \leq r-1$, we have:
\begin{align*}
\int_{\mathbb{T}^{r}} (\alpha_{r})_{z}(x)\, dz 
&= \int_{\mathbb{T}^{r}} \left( \prod_{i=1}^{r} z_{i}^{\#_{i}(u_{1}, \ldots, u_{k})} z_{i}^{-\#_{i}(v_{1}, \ldots, v_{l})} \right) x \, d(z_{1}, \ldots, z_{r}) \\
&= \int_{\mathbb{T}} z^{\#_{r}(u_{1}, \ldots, u_{k}) - \#_{r}(v_{1}, \ldots, v_{l})}\;x \, dz  \\
&= 
\begin{cases}
x, & \text{if } \#_{r}(u_{1}, \ldots, u_{k}) = \#_{r}(v_{1}, \ldots, v_{l}) \\
0, & \text{otherwise}.
\end{cases}
\end{align*}
From the definition of $B_{n_{1}, \ldots, n_{r-1}}$, we thus obtain
\[
B_{n_{1}, \ldots, n_{r-1}}^{\alpha_{r}} = \overline{\bigcup_{n=0}^{\infty} B_{n_{1}, \ldots, n_{r-1}}^{\leq n}}^{\Vert \cdot \Vert}.
\]

The statements about $\overline{B}_{n_{1}, \ldots, n_{r-1}}^{\beta_{r}}$, $\mathfrak{A}(\mathbf{A}, \Gamma)^{\alpha_{1}}$, and $\overline{\mathfrak{A}}^{\beta_{1}}$ follow similarly.
\end{proof}

\begin{lemma} \label{InductionStatement}
Let $1\leq r\leq L-1$ be an integer, $(n_{1},\ldots,n_{r})\in\mathbb{N}^{r}$ a tuple, and assume that the $\ast$-homomorphism $\phi:\overline{\mathfrak{A}}\twoheadrightarrow\mathfrak{A}(\mathbf{A},\Gamma)$ restricts to an isomorphism $\overline{B}_{n_{1},\ldots,n_{r},n}\cong B_{n_{1},\ldots,n_{r},n}$ for every $n\in\mathbb{N}$. Then, $\phi$ also restricts to an isomorphism $\overline{B}_{n_{1},\ldots,n_{r}}^{\leq n}\cong B_{n_{1},\ldots,n_{r}}^{\leq n}$ for every $n\in\mathbb{N}$.

Similarly, if $\phi$ restricts to an isomorphism $\overline{B}_{n}\cong B_{n}$ for every $n\in\mathbb{N}$, then $\phi$ also restricts to an isomorphism $\overline{B}_{\emptyset}^{\leq n}\cong B_{\emptyset}^{\leq n}$ for every $n\in\mathbb{N}$.
\end{lemma}

\begin{proof}
We prove the statement by induction on $n\in\mathbb{N}$.

For $n=0$, we have
\[
\overline{B}_{n_{1},\ldots,n_{r}}^{\leq0} = \overline{B}_{n_{1},\ldots,n_{r},0} \cong B_{n_{1},\ldots,n_{r},0} = B_{n_{1},\ldots,n_{r}}^{\leq0}
\]
by our assumption.

For the induction step, let $n\in\mathbb{N}$ and assume that $\phi$ restricts to an isomorphism $\overline{B}_{n_{1},\ldots,n_{r}}^{\leq(n-1)} \cong B_{n_{1},\ldots,n_{r}}^{\leq(n-1)}$. Let $x\in\overline{B}_{n_{1},\ldots,n_{r}}^{\leq n}$ be an element with $\phi(x)=0$. By Proposition \ref{IdealConstruction}, the C$^\ast$-subalgebras $\overline{B}_{n_{1},\ldots,n_{r},n} \triangleleft \overline{B}_{n_{1},\ldots,n_{r}}^{\leq n}$ and $B_{n_{1},\ldots,n_{r},n} \triangleleft B_{n_{1},\ldots,n_{r}}^{\leq n}$ are closed two-sided ideals. Consider the corresponding quotient maps $q: B_{n_{1},\ldots,n_{r}}^{\leq n} \twoheadrightarrow B_{n_{1},\ldots,n_{r}}^{\leq n} / B_{n_{1},\ldots,n_{r},n}$ and $\overline{q}: \overline{B}_{n_{1},\ldots,n_{r}}^{\leq n} \twoheadrightarrow \overline{B}_{n_{1},\ldots,n_{r}}^{\leq n} / \overline{B}_{n_{1},\ldots,n_{r},n}$, as well as the induced map $\psi: \overline{B}_{n_{1},\ldots,n_{r}}^{\leq n} / \overline{B}_{n_{1},\ldots,n_{r},n} \twoheadrightarrow B_{n_{1},\ldots,n_{r}}^{\leq n} / B_{n_{1},\ldots,n_{r},n}$, $\psi(y + \overline{B}_{n_{1},\ldots,n_{r},n}) := \phi(y) + B_{n_{1},\ldots,n_{r},n}$. It is clear that $\psi \circ \overline{q} = q \circ \phi$, and thus $\psi \circ \overline{q}(x) = 0$. By construction, there exists $y \in \overline{B}_{n_{1},\ldots,n_{r}}^{\leq(n-1)}$ such that $\overline{q}(x) = \overline{q}(y)$. Therefore,
\[
\phi(y) + B_{n_{1},\ldots,n_{r},n}
= \psi(y + \overline{B}_{n_{1},\ldots,n_{r},n})
= \psi \circ \overline{q}(y)
= \psi \circ \overline{q}(x)
= 0,
\]
and hence
\[
\phi(y) \in B_{n_{1},\ldots,n_{r},n} \cap B_{n_{1},\ldots,n_{r}}^{\leq(n-1)}.
\]

By the induction hypothesis and the definition of $\phi$, it follows that
\[
y \in \overline{B}_{n_{1},\ldots,n_{r},n} \cap \overline{B}_{n_{1},\ldots,n_{r}}^{\leq(n-1)} \subseteq \overline{B}_{n_{1},\ldots,n_{r},n},
\]
and therefore $x \in \overline{B}_{n_{1},\ldots,n_{r},n}$. Since $\phi$ restricts to an injection on $\overline{B}_{n_{1},\ldots,n_{r},n}$, we conclude that $x = 0$. This proves the first claim.

The second statement is proved analogously.
\end{proof}

We have now collected all the ingredients required to prove Theorem \ref{UniversalProperty}.

\begin{proof}[{Proof of Theorem \ref{UniversalProperty}}]
Let us first prove the following claim.\\

\emph{Claim}. For all $1\leq r\leq L$ and every tuple $(n_{1},\ldots,n_{r})\in\mathbb{N}^{r}$, the $\ast$-homomorphism $\phi:\overline{\mathfrak{A}}\twoheadrightarrow\mathfrak{A}(\mathbf{A},\Gamma)$ restricts to an isomorphism $\overline{B}_{n_{1},\ldots,n_{r}}\cong B_{n_{1},\ldots,n_{r}}$.

\emph{Proof of the claim}. We prove the statement by induction on $r$, starting at $r = L$. By Lemma \ref{IsomorphismTheorem2}, we know that for every tuple $(n_{1},\ldots,n_{L})\in\mathbb{N}^{L}$, the map $\phi$ restricts to an isomorphism $\overline{B}_{n_{1},\ldots,n_{L}}\cong B_{n_{1},\ldots,n_{L}}$, which completes the base case $r = L$.

For the induction step, assume that $\overline{B}_{n_{1},\ldots,n_{r}}\cong B_{n_{1},\ldots,n_{r}}$ for some fixed $2 \leq r \leq L$ and all $(n_{1},\ldots,n_{r})\in\mathbb{N}^{r}$. Lemma \ref{InductionStatement} then implies that $\overline{B}_{n_{1},\ldots,n_{r-1}}^{\leq n} \cong B_{n_{1},\ldots,n_{r-1}}^{\leq n}$ for all  $(n_{1},\ldots,n_{r-1})\in\mathbb{N}^{r-1},\; n\in\mathbb{N}$, so that we obtain an isomorphism
\[
\overline{\bigcup_{n=0}^{\infty} \overline{B}_{n_{1},\ldots,n_{r-1}}^{\leq n}}^{\Vert \cdot \Vert} \cong \overline{\bigcup_{n=0}^{\infty} B_{n_{1},\ldots,n_{r-1}}^{\leq n}}^{\Vert \cdot \Vert}.
\]
By Lemma \ref{FixedpointIdentification}, these C$^{\ast}$-algebras coincide with the fixed point algebras $\overline{B}_{n_{1},\ldots,n_{r-1}}^{\beta_{r}}$ and $B_{n_{1},\ldots,n_{r-1}}^{\alpha_{r}}$, implying that $\overline{B}_{n_{1},\ldots,n_{r-1}} \cong B_{n_{1},\ldots,n_{r-1}}$ for all $(n_{1},\ldots,n_{r-1}) \in \mathbb{N}^{r-1}$; see, e.g., \cite[Proposition 4.5.1]{BrownOzawa08}. This proves the claim.\\

In particular, the claim implies -- together with Lemma \ref{InductionStatement} -- that the $\ast$-homomorphism $\phi:\overline{\mathfrak{A}} \twoheadrightarrow \mathfrak{A}(\mathbf{A},\Gamma)$ restricts to an isomorphism $\overline{B}_{\emptyset}^{\leq n} \cong B_{\emptyset}^{\leq n}$ for every $n \in \mathbb{N}$. By Lemma \ref{FixedpointIdentification}, we have $\overline{\mathfrak{A}}^{\beta_{1}} = \overline{\bigcup_{n=0}^{\infty} \overline{B}_{\emptyset}^{\leq n}}^{\Vert \cdot \Vert}$ and 
$\mathfrak{A}(\mathbf{A},\Gamma)^{\alpha_{1}} = \overline{\bigcup_{n=0}^{\infty} B_{\emptyset}^{\leq n}}^{\Vert \cdot \Vert}$, so that $\phi$ is indeed an isomorphism.
\end{proof}

By applying the universal property in Theorem \ref{UniversalProperty} repeatedly, we find that in the case of nuclear vertex algebras, it can be lifted to a property independent of the choice of $v_{0} \in V\Gamma$.

\begin{theorem} \label{UniversalProperty2}
Let $\Gamma$ be a finite, undirected, simplicial graph and let $\mathbf{A} := (A_{v})_{v \in V\Gamma}$ be a collection of unital nuclear C$^{\ast}$-algebras, equipped with GNS-faithful states $(\omega_{v})_{v \in V\Gamma}$. Then $\mathfrak{A}(\mathbf{A},\Gamma)$ satisfies the following universal property: for every unital C$^{\ast}$-algebra $\overline{\mathfrak{A}}$ generated by the images of unital $\ast$-homomorphisms $\kappa_{v} : \mathfrak{A}(A_{v},\{v\}) \rightarrow \overline{\mathfrak{A}}$, $v \in V\Gamma$, satisfying  $[\kappa_{v}(x),\kappa_{v^\prime}(y)] = 0$ for all $x \in \mathfrak{A}(A_{v},\{v\})$, $y \in \mathfrak{A}(A_{v^\prime},\{v^\prime\})$ with $(v,v^\prime) \in E\Gamma$, and $\kappa_{v}(Q_{v})\kappa_{v^\prime}(Q_{v^\prime}) = 0$ for all $(v,v^\prime) \in E\Gamma^{c}$, there exists a surjective $\ast$-homomorphism $\phi : \mathfrak{A}(\mathbf{A},\Gamma) \twoheadrightarrow \overline{\mathfrak{A}}$ with $\phi(x) = \kappa_{v}(x)$ for all $v \in V\Gamma$, $x \in \mathfrak{A}(A_{v},\{v\})$.
\end{theorem}

\begin{proof}
Let $\overline{\mathfrak{A}}$ be the universal C$^{\ast}$-algebra generated by the images of unital $\ast$-homomorphisms $\kappa_{v} : \mathfrak{A}(A_{v},\{v\}) \rightarrow \overline{\mathfrak{A}}$, $v \in V\Gamma$, satisfying the commutation and orthogonality conditions as in the theorem. Then there exists a surjective $\ast$-homomorphism $\phi : \overline{\mathfrak{A}} \twoheadrightarrow \mathfrak{A}(\mathbf{A},\Gamma)$ such that $\phi(\kappa_{v}(x)) = x$ for all $v \in V\Gamma$, $x \in \mathfrak{A}(A_{v},\{v\})$. It is clear that the maps $\kappa_{v}$ must be faithful.

We prove that $\phi$ is faithful by induction on the cardinality of the vertex set $V\Gamma$. The case $\#V\Gamma = 1$ is trivial. For the induction step, assume that the statement holds for all finite, undirected, simplicial graphs $\Gamma'$ with $\#V\Gamma' < \#V\Gamma$.

We distinguish two cases: \\

\begin{itemize}
\item \emph{Case 1}: If the graph $\Gamma$ is not complete, there exists a vertex $v_{0}\in V\Gamma$ with $\#V\text{Star}(v_{0})<\#V\Gamma$. Define $\mathbf{A}_1:=(A_{v})_{v\in V\text{Star}(v_0)}$, $\mathbf{A}_2:=(A_{v})_{v\in V(\Gamma \setminus \{v_0\})}$, $\mathbf{B}:=(A_{v})_{v\in V\text{Link}(v_{0})}$. Note that by the induction assumption the C$^{\ast}$-subalgebra $\overline{\mathfrak{A}}_{1}$ of $\overline{\mathfrak{A}}$ generated by the union $\bigcup_{v\in V \text{Star}(v_0)}\text{im}(\kappa_{v})$ is a quotient of $\mathfrak{A}_{1}$ via $x\mapsto\kappa_{v}(x)$ for $v\in V\Gamma_{1}$, $x\in\mathfrak{A}(A_{v},\{v\})$. The composition of this map with $\phi$ is the identity, so that $\overline{\mathfrak{A}}_{1}\cong\mathfrak{A}_{1}$. Similarly, the induction assumption implies $\overline{\mathfrak{A}}_{2}\cong\mathfrak{A}_{2}$, where $\overline{\mathfrak{A}}_{2}$ is generated by $\bigcup_{v\in V(\Gamma \setminus \{v_0\})}\text{im}(\kappa_{v})$ in $\overline{\mathfrak{A}}$. But this means that $\overline{\mathfrak{A}}$ is generated by the images of unital $\ast$-homomorphisms $\kappa_{1}:\mathfrak{A}_{1}\rightarrow\overline{\mathfrak{A}}$, $\kappa_{2}:\mathfrak{A}_{2}\rightarrow\overline{\mathfrak{A}}$ satisfying $\kappa_{1}|_{B}=\kappa_{2}|_{B}$ and $\kappa_{1}(Q_{v_{0}})\kappa_{2}(Q_{v})=0$ for all $v\in V(\Gamma \setminus \text{Star}(v_{0}))$. From Theorem \ref{UniversalProperty} we obtain a surjective $\ast$-homomorphism $\psi:\mathfrak{A}(\mathbf{A},\Gamma)\twoheadrightarrow\overline{\mathfrak{A}}$ with $\psi|_{\mathfrak{A}_{1}}=\kappa_{1}$ and $\psi|_{\mathfrak{A}_{2}}=\kappa_{2}$. But $\phi\circ\psi=\text{id}$, so that $\mathfrak{A}(\mathbf{A},\Gamma)\cong\overline{\mathfrak{A}}$ via $x\mapsto\kappa_{v}(x)$ for $v\in V\Gamma$, $x\in\mathfrak{A}(A_{v},\{v\})$, as claimed.
\item \emph{Case 2}: Suppose $\Gamma$ is complete. Then repeated application of Proposition \ref{TensorDecomposition} gives $\mathfrak{A}(\mathbf{A},\Gamma) \cong \bigotimes_{v \in V\Gamma} \mathfrak{A}(A_v,\{v\})$. It suffices to show that each $\mathfrak{A}(A_v,\{v\})$ is nuclear. Since this algebra is generated by $A_v$ and $Q_v \in \mathcal{B}(\mathcal{H}_v)$, and $Q_v ^\perp$ is the orthogonal projection onto $\mathbb{C}\xi_v$, it contains the compact operators $\mathcal{K}(\mathcal{H}_v)$. Thus, there is a short exact sequence
\[
\mathcal{K}(\mathcal{H}_v) \hookrightarrow \mathfrak{A}(A_v,\{v\}) = A_v + \mathcal{K}(\mathcal{H}_v) \twoheadrightarrow A_v / (A_v \cap \mathcal{K}(\mathcal{H}_v)).
\]
If $A_v$ is nuclear, so is the quotient above. Since $\mathcal{K}(\mathcal{H}_v)$ is nuclear, it follows that $\mathfrak{A}(A_v,\{v\})$ is nuclear as well.

\end{itemize}
This completes the proof.
\end{proof}

\begin{remark}
In general, the nuclearity assumption in Theorem \ref{UniversalProperty2} cannot be omitted. Indeed, let $\Gamma$ be the complete graph on two vertices $v_1$ and $v_2$, and let $A_{v_1} = A_{v_2} = C_r^{\ast}(G)$, the reduced group C$^{\ast}$-algebra of a discrete group $G$, equipped with the canonical tracial state. Set $\mathbf{A} := (A_{v_1}, A_{v_2})$. For $i = 1, 2$, the C$^\ast$-algebra $\mathfrak{A}(A_{v_i},\{v_i\})$ is the C$^\ast$-subalgebra of $\mathcal{B}(\ell^2(G))$ generated by $C_r^{\ast}(G)$ and $\mathcal{K}(\ell^2(G))$. Define a unitary $J \in \mathcal{B}(\ell^2(G))$ by $J\delta_g := \delta_{g^{-1}}$ for $g \in G$, where $(\delta_g)_{g \in G}\subseteq \ell^2(G)$ is the canonical basis, as well as $\ast$-homomorphisms $\kappa_{1}:\mathfrak{A}(A_{v_{1}},\{v_{1}\})\rightarrow\mathcal{B}(\ell^{2}(G))/\mathcal{K}(\ell^{2}(G))$ and $\kappa_{2}:\mathfrak{A}(A_{v_{2}},\{v_{2}\})\rightarrow\mathcal{B}(\ell^{2}(G))/\mathcal{K}(\ell^{2}(G))$ via $\kappa_{1}(x):=x+\mathcal{K}(\ell^{2}(G))$ and $\kappa_{2}(x):=JxJ+\mathcal{K}(\ell^{2}(G))$, whose images commute with each other. Assuming that $\mathfrak{A}(\mathbf{A},\Gamma)$ satisfies the universal property in Theorem \ref{UniversalProperty2}, we obtain a $\ast$-homomorphism
\[
\phi : \mathfrak{A}(A_{v_1},\{v_1\}) \otimes \mathfrak{A}(A_{v_2},\{v_2\}) \cong \mathfrak{A}(\mathbf{A},\Gamma) \rightarrow \mathcal{B}(\ell^2(G)) / \mathcal{K}(\ell^2(G))
\]
with $\phi(x \otimes 1) = \kappa_1(x)$ and $\phi(1 \otimes y) = \kappa_2(y)$ for all $x,y \in \mathfrak{A}(A_{v_i},\{v_i\})$. This restricts to a $\ast$-homomorphism
\[
C_r^{\ast}(G) \otimes C_r^{\ast}(G) \rightarrow \mathcal{B}(\ell^2(G)) / \mathcal{K}(\ell^2(G)), \quad x \otimes y \mapsto xJ y J + \mathcal{K}(\ell^2(G)),
\]
which implies that $G$ has Ozawa's Akemann–Ostrand property $(\mathcal{AO})$ (see \cite{Ozawa04} and \cite[Definition 6.1]{Anantharaman-Delaroche09}). However, there are known examples of groups that do not satisfy this property.
\end{remark}

\vspace{3mm}


\subsection{Nuclearity and Exactness}

The goal of this subsection is to prove the following theorem, which characterizes the nuclearity and exactness of our main construction.

\begin{theorem} \label{NuclearityExactness}
Let $\Gamma$ be a finite, undirected, simplicial graph, and let $\mathbf{A} := (A_{v})_{v \in V\Gamma}$ be a collection of unital C$^{\ast}$-algebras equipped with GNS-faithful states $(\omega_{v})_{v \in V\Gamma}$. Then the following two statements hold:
\begin{enumerate}
    \item The C$^{\ast}$-algebra $\mathfrak{A}(\mathbf{A},\Gamma)$ is nuclear if and only if $A_{v}$ is nuclear for every $v \in V\Gamma$.
    \item The C$^{\ast}$-algebra $\mathfrak{A}(\mathbf{A},\Gamma)$ is exact if and only if $A_{v}$ is exact for every $v \in V\Gamma$.
\end{enumerate}
\end{theorem}

\begin{remark}
Combined with \cite[Theorem 4.4.3]{BrownOzawa08}, Theorem \ref{NuclearityExactness} implies, in the context of Example \ref{MainExample} (1), that the action of the right-angled Coxeter group $W$ on its combinatorial boundary is amenable, thereby partially recovering the main result of \cite{Lecureux10} (see also \cite[Corollary 3.8]{Klisse23-1}). Similarly, if the vertex groups $(G_v)_{v \in V\Gamma}$ are all amenable, then the action of the graph product $\mathbf{G}_\Gamma = \star_{v, \Gamma} G_{v}$ on $\overline{(K,o)}$ in Example \ref{MainExample} (2) is amenable.
\end{remark}

The proof of Theorem \ref{NuclearityExactness} follows a similar structure to that of Theorem \ref{UniversalProperty} in Subsection \ref{subsec:Universality}. For this, we adopt the same notation as in Subsection \ref{subsec:Universality}: fix a finite, undirected, simplicial graph $\Gamma$, and let $\mathbf{A} := (A_{v})_{v \in V\Gamma}$ be a collection of unital C$^{\ast}$-algebras equipped with GNS-faithful states $(\omega_{v})_{v \in V\Gamma}$. Choose an enumeration $(s_{1}, \dots, s_{L})$ of the elements of $V\Gamma$, with $L := \#V\Gamma$, and define for every number $1 \leq r \leq L$, every tuple $(n_{1}, \dots, n_{r}) \in \mathbb{N}^{r}$, and every $n \in \mathbb{N}$ the C$^{\ast}$-algebras $B_{n_{1}, \dots, n_{r}}$, $B_{\emptyset}^{\leq n}$, and $B_{n_{1}, \dots, n_{r}}^{\leq n}$ as before. Furthermore, set $\mathcal{D} := \overline{\text{Span}}^{\Vert \cdot \Vert}(\mathcal{D}_{0}(\mathbf{A},\Gamma)) \subseteq \mathfrak{A}(\mathbf{A},\Gamma)$.

\begin{proposition} \label{DNuclearity}
If the C$^{\ast}$-algebras $(A_{v})_{v \in V\Gamma}$ are all nuclear (resp. exact), then $\mathcal{D}$ is nuclear (resp. exact) as well.
\end{proposition}

\begin{proof}
Assume that the C$^{\ast}$-algebras $(A_{v})_{v \in V\Gamma}$ are nuclear. We prove the nuclearity statement by induction on the size of the vertex set $V\Gamma$.

In the case where $V\Gamma = \{v\}$, we have $\mathcal{D} = \mathcal{D}(\mathbf{A},\Gamma)$. By Theorem \ref{GaugeActionTheorem} (3), the C$^{\ast}$-subalgebra $\mathcal{D}$ of $\mathfrak{A}(\mathbf{A},\Gamma)$ is expected. Since $A_{v}$ is nuclear, it follows -- by the same argument as in the proof of Theorem \ref{UniversalProperty2} -- that $\mathfrak{A}(\mathbf{A},\Gamma) = \mathfrak{A}(A_v, \{v\})$ is nuclear, and therefore $\mathcal{D}$ is nuclear as well.

For the induction step, consider a finite, undirected, simplicial graph $\Gamma$, and assume that the statement holds for all  graphs $\Gamma^{\prime}$ with $\#V\Gamma^{\prime} < \#V\Gamma$. We distinguish two cases:\\

\begin{itemize}
\item \emph{Case 1}: If the graph $\Gamma$ is not complete, then there exists $v_{0} \in V\Gamma$ such that $\#V\text{Star}(v_{0}) < \#V\Gamma$. Set $\mathbf{A}_1 := (A_{v})_{v \in V(\text{Star}(v_0))}$, $\mathbf{A}_2 := (A_{v})_{v \in V(\Gamma \setminus \{ v_0 \})}$, and $\mathbf{B} := (A_{v})_{v \in V(\text{Link}(v_0))}$. Further define the C$^{\ast}$-algebras
\begin{eqnarray}
\nonumber \mathcal{D}_{1} &:=& \overline{\text{Span}}^{\Vert \cdot \Vert} (\mathcal{D}_{0}(\mathbf{A}_1, \text{Star}(v_0))), \\
\nonumber \mathcal{D}_{2} &:=& \overline{\text{Span}}^{\Vert \cdot \Vert} (\mathcal{D}_{0}(\mathbf{A}_2, \Gamma \setminus \{v_0\})), \\
\nonumber \mathcal{D}_{B} &:=& \overline{\text{Span}}^{\Vert \cdot \Vert} (\mathcal{D}_{0}(\mathbf{B}, \text{Link}(v_0))).
\end{eqnarray}
Note that the conditional expectation $\mathbb{E}_{\Gamma, \text{Link}(v_0)}$ provided by Theorem \ref{ConditionalExpectation} restricts to a conditional expectation $\mathbb{E}_{B} : \mathcal{D} \rightarrow \mathcal{D}_{B}$. By the induction assumption, the C$^{\ast}$-algebras $\mathcal{D}_{1}$, $\mathcal{D}_{2}$, and $\mathcal{D}_{B}$ are nuclear.

Now let $\varepsilon > 0$ and let $\mathcal{F} \subseteq \mathcal{D}$ be a finite subset. As in the proof of Proposition \ref{IsomorphismTheorem}, we can express every $x \in \mathcal{F}$ as
\[
x = Q_{v_0}x + Px + Q^{\perp}x = Q_{v_0}x + Px + Q^{\perp}\mathbb{E}_{B}(x),
\]
where $P := \bigvee_{v \in V(\Gamma \setminus \text{Star}(v_0))} Q_{v} \in \mathfrak{A}(\mathbf{A}_2, \Gamma \setminus \{v_0\})$ and $Q := Q_{v_0} + P$. Note that $Q_{v_0}x \in \mathcal{D}_{1}$ and $Px \in \mathcal{D}_{2}$.

By assumption, there exists $n \in \mathbb{N}$ and contractive completely positive maps
\[
\varphi_{1} : \mathcal{D}_{1} \rightarrow M_{n}(\mathbb{C}), \quad 
\varphi_{2} : \mathcal{D}_{2} \rightarrow M_{n}(\mathbb{C}), \quad 
\varphi_{B} : \mathcal{D}_{B} \rightarrow M_{n}(\mathbb{C}),
\]
\[
\psi_{1} : M_{n}(\mathbb{C}) \rightarrow \mathcal{D}_{1}, \quad 
\psi_{2} : M_{n}(\mathbb{C}) \rightarrow \mathcal{D}_{2}, \quad 
\psi_{B} : M_{n}(\mathbb{C}) \rightarrow \mathcal{D}_B,
\]
such that for all $x \in \mathcal{F}$,
\begin{eqnarray}
\nonumber && \Vert Q_{v_0}x - (\psi_{1} \circ \varphi_{1})(Q_{v_0}x) \Vert < \frac{\varepsilon}{3}, \\
\nonumber && \Vert Px - (\psi_{2} \circ \varphi_{2})(Px) \Vert < \frac{\varepsilon}{3}, \\
\nonumber && \Vert \mathbb{E}_{B}(x) - (\psi_{B} \circ \varphi_{B})(\mathbb{E}_{B}(x)) \Vert < \frac{\varepsilon}{3}.
\end{eqnarray}

Define $\varphi : \mathcal{D} \rightarrow M_{n}(\mathbb{C}) \oplus M_{n}(\mathbb{C}) \oplus M_{n}(\mathbb{C})$ by
\[
\varphi(x) := (\varphi_{1}(Q_{v_0}x), \varphi_{2}(Px), \varphi_{B}(\mathbb{E}_{B}(x))),
\]
and $\psi : M_{n}(\mathbb{C}) \oplus M_{n}(\mathbb{C}) \oplus M_{n}(\mathbb{C}) \rightarrow \mathcal{D}$ by
\[
\psi(z_{1}, z_{2}, z_{3}) := Q_{v_0}\psi_{1}(z_{1}) + P\psi_{2}(z_{2}) + Q^{\perp}\psi_{B}(z_{3}).
\]
Then both $\varphi$ and $\psi$ are contractive completely positive maps, and for all $x \in \mathcal{F}$,
\begin{eqnarray}
\nonumber \Vert x - \psi \circ \varphi(x) \Vert 
&=& \Vert (Q_{v_0}x - Q_{v_0}(\psi_{1} \circ \varphi_{1})(Q_{v_0}x)) 
+ (Px - P(\psi_{2} \circ \varphi_{2})(Px)) \\
\nonumber && \quad + (Q^{\perp}\mathbb{E}_{B}(x) - Q^{\perp}(\psi_{B} \circ \varphi_{B})(\mathbb{E}_{B}(x))) \Vert 
< \varepsilon.
\end{eqnarray}
Hence, $\mathcal{D}$ is nuclear.

\item \emph{Case 2}: If $\Gamma$ is complete, then repeated application of Proposition \ref{TensorDecomposition} implies $ \mathcal{D} \cong \bigotimes_{v \in V\Gamma} \overline{\text{Span}}^{\Vert \cdot \Vert}(\mathcal{D}_{0}(A_{v}, \{v\}))$. By the induction assumption, each factor in the tensor product is nuclear, and so is $\mathcal{D}$.
\end{itemize}

This completes the proof of the nuclearity of $\mathcal{D}$. The statement about exactness follows similarly.
\end{proof}

\begin{corollary} \label{Nuclearity2}
If the C$^{\ast}$-algebras $(A_{v})_{v \in V\Gamma}$ are nuclear (resp. exact), then $B_{n_{1}, \dots, n_{L}}$ is nuclear (resp. exact) for every tuple $(n_{1}, \dots, n_{L}) \in \mathbb{N}^{L}$.
\end{corollary}

\begin{proof}
As shown in Subsection \ref{subsec:Universality}, we have $B_{n_{1}, \dots, n_{L}} \cong \mathcal{K}(\mathfrak{X}_{n_{1}, \dots, n_{L}})$, where $\mathfrak{X}_{n_{1}, \dots, n_{L}}$ is the Hilbert $\mathcal{D}$-module defined in \eqref{HilbertModule}. The C$^{\ast}$-algebra $\mathcal{K}(\mathfrak{X}_{n_{1}, \dots, n_{L}})$ is nuclear (resp. exact) if and only if $\mathcal{D}$ is; see, e.g., \cite[Exercise 4.6.3]{BrownOzawa08}. The desired conclusion follows from Proposition \ref{DNuclearity}.
\end{proof}

\begin{lemma} \label{NuclearityInductionStatement} 
Let $1\leq r\leq L-1$ be an integer, $(n_{1},...,n_{r})\in\mathbb{N}^{r}$ a tuple, and assume that $B_{n_{1},...,n_{r},n}$ is nuclear (resp. exact) for every $n\in\mathbb{N}$. Then $B_{n_{1},...,n_{r}}^{\leq n}$ is also nuclear (resp. exact) for every $n\in\mathbb{N}$. Similarly, if $B_{n}$ is nuclear (resp. exact) for every $n\in\mathbb{N}$, then $B_{\emptyset}^{\leq n}$ is also nuclear (resp. exact) for every $n\in\mathbb{N}$.
\end{lemma} 

\begin{proof}
\emph{Nuclearity:} For the first statement, assume that $B_{n_{1},...,n_{r},n}$ is nuclear for every $n\in\mathbb{N}$. We proceed by induction on $n\in\mathbb{N}$.

For $n=0$, we have $B_{n_{1},...,n_{r}}^{\leq 0} = B_{n_{1},...,n_{r},0}$, which is nuclear by assumption.

For the induction step, fix $n\in\mathbb{N}$ and assume that $B_{n_{1},...,n_{r}}^{\leq (n-1)}$ is nuclear. By Proposition \ref{IdealConstruction}, $B_{n_{1},...,n_{r},n} \triangleleft B_{n_{1},...,n_{r}}^{\leq n}$ is a closed two-sided ideal. We obtain a short exact sequence:
\begin{eqnarray} \label{eq:ShortExact}
B_{n_{1},...,n_{r},n} \hookrightarrow B_{n_{1},...,n_{r}}^{\leq n} \twoheadrightarrow B_{n_{1},...,n_{r}}^{\leq n}/B_{n_{1},...,n_{r},n} \cong B_{n_{1},...,n_{r}}^{\leq(n-1)}/(B_{n_{1},...,n_{r}}^{\leq(n-1)} \cap B_{n_{1},...,n_{r},n}).
\end{eqnarray}
By our assumptions, both $B_{n_{1},...,n_{r},n}$ and $B_{n_{1},...,n_{r}}^{\leq(n-1)}/(B_{n_{1},...,n_{r}}^{\leq(n-1)} \cap B_{n_{1},...,n_{r},n})$ are nuclear C$^{\ast}$-algebras. It follows that $B_{n_{1},...,n_{r}}^{\leq n}$ is nuclear as well, thus completing the induction.\\

\emph{Exactness:} The statement about exactness follows in a similar manner. Assume that $B_{n_{1},...,n_{r},n}$ is exact for every $n\in\mathbb{N}$ and proceed inductively as before.

The base case holds by assumption. 

For the induction step, assume that for some fixed $n\in\mathbb{N}$ the C$^{\ast}$-algebra $B_{n_{1},...,n_{r}}^{\leq(n-1)}$ is exact. Consider the short exact sequence
\begin{equation} \label{eq:ShortExactSequence2}
B_{n_{1},...,n_{r}}^{\leq(n-1)} \cap B_{n_{1},...,n_{r},n} \hookrightarrow B_{n_{1},...,n_{r}}^{\leq(n-1)} \twoheadrightarrow B_{n_{1},...,n_{r}}^{\leq(n-1)}/(B_{n_{1},...,n_{r}}^{\leq(n-1)} \cap B_{n_{1},...,n_{r},n}).
\end{equation}
Since $B_{n_{1},...,n_{r}}^{\leq(n-1)}$ is exact, it is also locally reflexive. By \cite[Proposition 9.1.4]{BrownOzawa08}, the sequence \eqref{eq:ShortExactSequence2} is locally split. That is, for every finite-dimensional operator system $E \subseteq B_{n_{1},...,n_{r}}^{\leq(n-1)}/(B_{n_{1},...,n_{r}}^{\leq(n-1)} \cap B_{n_{1},...,n_{r},n})$, there exists a unital completely positive map $\sigma: E \rightarrow B_{n_{1},...,n_{r}}^{\leq(n-1)}$ such that $q \circ \sigma = \mathrm{id}_E$, where $q$ is the canonical quotient map. Consequently, the sequence in \eqref{eq:ShortExact} is also locally split. Since exactness is preserved under extensions that locally split (see, e.g., \cite[Exercise 3.9.8]{BrownOzawa08}), and since both $B_{n_{1},...,n_{r},n}$ and the quotient are exact, it follows that $B_{n_{1},...,n_{r}}^{\leq n}$ is exact as well. This completes the induction.\\

The statement for the C$^{\ast}$-algebras $B_{\emptyset}^{\leq n}$, $n\in\mathbb{N}$, follows similarly.
\end{proof}

We can now prove Theorem \ref{NuclearityExactness}.

\begin{proof}[{Proof of Theorem \ref{NuclearityExactness}}]
We prove only the statement about nuclearity, as the case of exactness can be treated similarly.  For the “if” direction, assume that $A_v$ is nuclear for every $v \in V\Gamma$. We proceed by proving the following claim. \\

\emph{Claim}. For all $1 \leq r \leq L$ and every tuple $(n_1, \ldots, n_r) \in \mathbb{N}^r$, the C$^*$-algebra $B_{n_1, \ldots, n_r}$ is nuclear.

\emph{Proof of the claim}. As in the proof of Theorem \ref{UniversalProperty}, we proceed by induction on $r$, starting at $r = L$.

By Proposition \ref{DNuclearity}, $B_{n_1, \ldots, n_L}$ is nuclear for every tuple $(n_1, \ldots, n_L) \in \mathbb{N}^L$, which establishes the base case $r = L$.

For the induction step, assume that for some fixed $2 \leq r \leq L$, the C$^*$-algebra $B_{n_1, \ldots, n_r}$ is nuclear for all $(n_1, \ldots, n_r) \in \mathbb{N}^r$. Then Lemma \ref{NuclearityInductionStatement} implies that $B_{n_1, \ldots, n_{r-1}}^{\leq n}$ is nuclear for every $(n_1, \ldots, n_{r-1}) \in \mathbb{N}^{r-1}$ and every $n \in \mathbb{N}$. It follows that $\overline{\bigcup_{n=0}^{\infty} B_{n_1, \ldots, n_{r-1}}^{\leq n}}^{\Vert \cdot \Vert}$ is nuclear. By Lemma \ref{FixedpointIdentification}, this C$^*$-algebra coincides with the fixed-point algebra $B_{n_1, \ldots, n_{r-1}}^{\alpha_r}$, where $\alpha_r$ is defined as in Subsection \ref{subsec:Universality}. Using \cite[Theorem 4.5.2]{BrownOzawa08}, we conclude that $B_{n_1, \ldots, n_{r-1}}$ is nuclear for all $(n_1, \ldots, n_{r-1}) \in \mathbb{N}^{r-1}$. This completes the proof of the claim.\\

From the claim, we deduce -- using Lemma \ref{NuclearityInductionStatement} -- that $B_{\emptyset}^{\leq n}$ is nuclear for every $n \in \mathbb{N}$. Therefore, the C$^*$-algebra $\mathfrak{A}(\mathbf{A}, \Gamma)^{\alpha_1} = \overline{\bigcup_{n=0}^{\infty} B_{\emptyset}^{\leq n}}^{\Vert \cdot \Vert}$ is also nuclear. Together with \cite[Theorem 4.5.2]{BrownOzawa08}, this implies that $\mathfrak{A}(\mathbf{A}, \Gamma)$ is nuclear.\\

For the “only if” direction, assume that $\mathfrak{A}(\mathbf{A}, \Gamma)$ is a nuclear C$^*$-algebra. By Theorem \ref{ConditionalExpectation}, there exists a conditional expectation $\mathbb{E}_{\Gamma, \{v\}}: \mathfrak{A}(\mathbf{A}, \Gamma) \rightarrow \mathfrak{A}(A_v, \{v\})$ for every $v \in V\Gamma$, implying that each $\mathfrak{A}(A_v, \{v\})$ is nuclear. Since $A_v / (\mathcal{K}(\mathcal{H}_v) \cap A_v)$ is a quotient of $\mathfrak{A}(A_v, \{v\})$, it is also nuclear. Moreover, $\mathcal{K}(\mathcal{H}_v)$ is a type I C$^*$-algebra, so the C$^\ast$-subalgebra $\mathcal{K}(\mathcal{H}_v) \cap A_v$ is nuclear as well. Combining these two facts, we conclude that $A_v$ is nuclear. This completes the proof.
\end{proof}

\vspace{3mm}


\subsection{The Ideal $\mathfrak{I}(\mathbf{A},\Gamma)$}

\label{subsec:Ideal}

Let $\Gamma$ be a finite, undirected, simplicial graph, and let $\mathbf{A}:=(A_{v})_{v\in V\Gamma}$ be a collection of unital C$^{\ast}$-algebras, each equipped with a GNS-faithful state $\omega_{v}$. In this subsection, we identify a canonical maximal ideal $\mathfrak{I}(\mathbf{A},\Gamma)$ inside $\mathfrak{A}(\mathbf{A},\Gamma)$. As in the proof of Theorem \ref{GaugeActionTheorem}, consider for every element $\mathbf{w}\in W_{\Gamma}\setminus\{e\}$ the orthogonal projection $p_{\mathbf{w}}\in\mathcal{B}(\mathcal{H}_{\Gamma})$ onto the component $\mathcal{H}_{\mathbf{w}}^{\circ}$ in the decomposition $\mathcal{H}_{\Gamma}=\mathbb{C}\Omega\oplus\bigoplus_{\mathbf{w}\in W_{\Gamma}\setminus\{e\}}\mathcal{H}_{\mathbf{w}}^{\circ}$, and let $p_{e}$ be the projection onto $\mathbb{C}\Omega$. For $N\in\mathbb{N}$, define $P_{N}:=\sum_{\mathbf{w}\in W_{\Gamma}:|\mathbf{w}|\leq N}p_{\mathbf{w}}$, and set 
\begin{equation}
\mathfrak{I}(\mathbf{A},\Gamma):=\left\{ x\in\mathfrak{A}(\mathbf{A},\Gamma) \ \middle| \  \lim_{N\rightarrow\infty}\Vert\mathbb{E}(x^{\ast}x)P_{N}^{\perp}\Vert=0\right\} \subseteq\mathfrak{A}(\mathbf{A},\Gamma).\label{eq:IdealDefinition}
\end{equation}
Note that for every $x\in\mathfrak{A}(\mathbf{A},\Gamma)$,
\begin{equation}
\Vert\mathbb{E}(x^{\ast}x)P_{N}^{\perp}\Vert\leq\sup\left\{ \frac{\Vert x\Vert\Vert x\xi\Vert}{\Vert\xi\Vert} \ \middle| \  0\neq\xi\in\mathcal{H}_{\mathbf{w}}^{\circ}\text{ with }|\mathbf{w}|\geq N\right\} \leq\Vert x\Vert\Vert\mathbb{E}(x^{\ast}x)P_{N}^{\perp}\Vert^{\frac{1}{2}}.\label{eq:IdealInequality}
\end{equation}

\begin{proposition} \label{IdealExistence} The space $\mathfrak{I}(\mathbf{A},\Gamma)$ defined in (\ref{eq:IdealDefinition}) is a closed two-sided ideal in $\mathfrak{A}(\mathbf{A},\Gamma)$ that contains the compact operators $\mathcal{K}(\mathcal{H}_{\Gamma})\subseteq\mathcal{B}(\mathcal{H}_{\Gamma})$. \end{proposition}

\begin{proof} By \eqref{eq:IdealInequality} is clear that $\mathfrak{I}(\mathbf{A},\Gamma)$ is a linear subspace of $\mathfrak{A}(\mathbf{A},\Gamma)$. To see that it is also norm closed, let $(x_{i})_{i\in\mathbb{N}}\subseteq\mathfrak{I}(\mathbf{A},\Gamma)$ be a sequence converging to some $x\in\mathfrak{A}(\mathbf{A},\Gamma)$. For any $\varepsilon>0$, we can find $i_{0}\in\mathbb{N}$ such that $\Vert x^{\ast}x-x_{i_{0}}^{\ast}x_{i_{0}}\Vert<\frac{\varepsilon}{2}$, and $N_{0}\in\mathbb{N}$ such that $\Vert\mathbb{E}(x_{i_{0}}^{\ast}x_{i_{0}})P_{N_{0}}^{\perp}\Vert<\frac{\varepsilon}{2}$. Then, 
\[
\Vert\mathbb{E}(x^{\ast}x)P_{N}^{\perp}\Vert\leq\Vert\mathbb{E}(x^{\ast}x)P_{N_{0}}^{\perp}\Vert\leq\Vert x^{\ast}x-x_{i_{0}}^{\ast}x_{i_{0}}\Vert+\Vert\mathbb{E}(x_{i_{0}}^{\ast}x_{i_{0}})P_{N_{0}}^{\perp}\Vert<\varepsilon
\]
for all $N\geq N_{0}$, showing that $x\in\mathfrak{I}(\mathbf{A},\Gamma)$. Hence, $\mathfrak{I}(\mathbf{A},\Gamma)$ is norm closed.

It is clear that $\mathfrak{I}(\mathbf{A},\Gamma)$ is invariant under left multiplication with elements in $\mathfrak{A}(\mathbf{A},\Gamma)$. For right invariance, let $v\in V\Gamma$, $a\in A_{v}^{\circ}$, $d\in\mathcal{D}_{0}(\mathbf{A},\Gamma)$, and $x\in\mathfrak{I}(\mathbf{A},\Gamma)$. Using 
\[
a^{\dagger}p_{\mathbf{w}}=p_{v\mathbf{w}}a^{\dagger}p_{\mathbf{w}},\quad dp_{\mathbf{w}}=p_{\mathbf{w}}dp_{\mathbf{w}},\quad(a^{\dagger})^{\ast}p_{\mathbf{w}}=p_{v\mathbf{w}}(a^{\dagger})^{\ast}p_{\mathbf{w}}
\]
for $\mathbf{w}\in W_{\Gamma}$, we deduce that 
\begin{eqnarray}
\nonumber
\Vert\mathbb{E}((xa^{\dagger})^{\ast}(xa^{\dagger}))P_{N}^{\perp}\Vert &=& \sup_{\mathbf{w}\in W_{\Gamma}:|\mathbf{w}|>N}\Vert p_{\mathbf{w}}((xa^{\dagger})^{\ast}(xa^{\dagger}))p_{\mathbf{w}}\Vert \\
\nonumber
&\leq&  \sup_{\mathbf{w}\in W_{\Gamma}:|\mathbf{w}|>N} \Vert a\Vert^{2} \Vert p_{v\mathbf{w}}(x^{\ast}x)p_{v\mathbf{w}}\Vert \\
\nonumber
&\leq& \Vert a\Vert^{2}\Vert\mathbb{E}(x^{\ast}x)P_{N-1}^{\perp}\Vert,
\end{eqnarray}
and similarly,
\[
\Vert\mathbb{E}((xd)^{\ast}(xd))P_{N}^{\perp}\Vert\leq\Vert d\Vert^{2}\Vert\mathbb{E}(x^{\ast}x)P_{N}^{\perp}\Vert,\quad\Vert\mathbb{E}((x(a^{\dagger})^{\ast})^{\ast}(x(a^{\dagger})^{\ast}))P_{N}^{\perp}\Vert\leq\Vert a\Vert^{2}\Vert\mathbb{E}(x^{\ast}x)P_{N-1}^{\perp}\Vert.
\]
By Proposition \ref{DensityStatement}, this implies that $\mathfrak{I}(\mathbf{A},\Gamma)$ is also invariant under right multiplication. Thus, $\mathfrak{I}(\mathbf{A},\Gamma)$ is a closed two-sided ideal in $\mathfrak{A}(\mathbf{A},\Gamma)$.

To see that $\mathcal{K}(\mathcal{H}_{\Gamma})\subseteq\mathfrak{I}(\mathbf{A},\Gamma)$, note that every rank-one operator $\theta_{\xi,\eta}\in\mathcal{K}(\mathcal{H}_{\Gamma})$ with $\xi,\eta\in\mathcal{H}_{\Gamma}$, given by $\zeta\mapsto\langle\eta,\zeta\rangle\xi$, can be approximated by finite linear combinations of expressions of the form 
\[
(a_{1}^{\dagger}\cdots a_{k}^{\dagger})p_{e}(b_{1}^{\dagger}\cdots b_{l}^{\dagger})^{\ast}=\theta_{a_{1}\cdots a_{k}\Omega,b_{1}\cdots b_{l}\Omega},
\]
where $k,l\in\mathbb{N}$, $(u_{1},\dots,u_{k}),(v_{1},\dots,v_{l})\in\mathcal{W}_{\text{red}}$, $a_{i}\in A_{u_{i}}^{\circ}$, $b_{j}\in A_{v_{j}}^{\circ}$, and where $p_{e}=\prod_{v\in V\Gamma}Q_{v}^{\perp}\in\mathfrak{A}(\mathbf{A},\Gamma)$ is the orthogonal projection onto $\mathbb{C}\Omega$. \end{proof}

As shown in Theorem \ref{GaugeActionTheorem}, the C$^{\ast}$-algebra $\mathfrak{A}(\mathbf{A},\Gamma)$ admits a canonical gauge action $\alpha:\mathbb{T}^{V\Gamma}\curvearrowright\mathfrak{A}(\mathbf{A},\Gamma)$, induced by conjugation with the unitaries $U_{z}\in\mathcal{B}(\mathcal{H}_{\Gamma})$, $z\in\mathbb{T}^{V\Gamma}$, defined by $U_{z}\Omega:=\Omega$ and $U_{z}\xi:=z_{\mathbf{w}}\xi$ for $\xi\in\mathcal{H}_{\mathbf{w}}^{\circ}$. This action is compatible with the ideal $\mathfrak{I}(\mathbf{A},\Gamma)$.

\begin{lemma} \label{GaugeActionTheorem2} Let $\Gamma$ be a finite, undirected, simplicial graph, and let $\mathbf{A}:=(A_{v})_{v\in V\Gamma}$ be a collection of unital C$^{\ast}$-algebras, each equipped with a GNS-faithful state $\omega_{v}$. Then the gauge action $\alpha:\mathbb{T}^{V\Gamma}\curvearrowright\mathfrak{A}(\mathbf{A},\Gamma)$ from Theorem \ref{GaugeActionTheorem} satisfies $\alpha_{z}(\mathfrak{I}(\mathbf{A},\Gamma))\subseteq\mathfrak{I}(\mathbf{A},\Gamma)$ for every $z\in\mathbb{T}^{V\Gamma}$, and we also have $\mathbb{E}(\mathfrak{I}(\mathbf{A},\Gamma))\subseteq\mathfrak{I}(\mathbf{A},\Gamma)$, where $\mathbb{E}$ is the canonical faithful conditional expectation onto $\mathcal{D}(\mathbf{A},\Gamma)$. \end{lemma}

\begin{proof} The inclusion $\alpha_{z}(\mathfrak{I}(\mathbf{A},\Gamma))\subseteq\mathfrak{I}(\mathbf{A},\Gamma)$ for all $z\in\mathbb{T}^{V\Gamma}$ is immediate, since both $\mathbb{E}$ and $U_{z}^{\ast}$ commute with each $P_{N}^{\perp}$. It remains to show that $\mathbb{E}(\mathfrak{I}(\mathbf{A},\Gamma))\subseteq\mathfrak{I}(\mathbf{A},\Gamma)$. For $x\in\mathfrak{I}(\mathbf{A},\Gamma)$, we compute 
\begin{eqnarray}
\nonumber
\Vert\mathbb{E}(\mathbb{E}(x)^{\ast}\mathbb{E}(x))P_{N}^{\perp}\Vert &=& \sup_{\mathbf{w}\in W_{\Gamma}:|\mathbf{w}|>N}\Vert p_{\mathbf{w}}(\mathbb{E}(x)^{\ast}\mathbb{E}(x))p_{\mathbf{w}}\Vert \\
\nonumber
&\leq& \sup_{\mathbf{w}\in W_{\Gamma}:|\mathbf{w}|>N}\Vert p_{\mathbf{w}}(x^{\ast}x)p_{\mathbf{w}}\Vert \\
\nonumber
&=& \Vert\mathbb{E}(x^{\ast}x)P_{N}^{\perp}\Vert \\
\nonumber
&\rightarrow& 0,
\end{eqnarray}
as $N\to\infty$, hence $\mathbb{E}(x)\in\mathfrak{I}(\mathbf{A},\Gamma)$. \end{proof}

The lemma above implies, in particular, that the gauge action $\alpha:\mathbb{T}^{V\Gamma}\curvearrowright\mathfrak{A}(\mathbf{A},\Gamma)$ descends to a well-defined action on the quotient $\mathfrak{A}(\mathbf{A},\Gamma)/\mathfrak{I}(\mathbf{A},\Gamma)$, and similarly, $\mathbb{E}$ induces a well-defined conditional expectation $\mathfrak{A}(\mathbf{A},\Gamma)/\mathfrak{I}(\mathbf{A},\Gamma)\rightarrow\pi(\mathcal{D}(\mathbf{A},\Gamma))$, where $\pi:\mathfrak{A}(\mathbf{A},\Gamma)\twoheadrightarrow\mathfrak{A}(\mathbf{A},\Gamma)/\mathfrak{I}(\mathbf{A},\Gamma)$ denotes the quotient map. The induced conditional expectation is faithful. Indeed, for positive $x\in\mathfrak{A}(\mathbf{A},\Gamma)$ with $\mathbb{E}(x)\in\mathfrak{I}(\mathbf{A},\Gamma)$, also $\mathbb{E}(x)^{\frac{1}{2}}\in\mathfrak{I}(\mathbf{A},\Gamma)$. From 
\[
\Vert\mathbb{E}(x)P_{N}^{\perp}\Vert=\Vert\mathbb{E}(\mathbb{E}(x)^{\frac{1}{2}}\mathbb{E}(x)^{\frac{1}{2}})P_{N}^{\perp}\Vert\rightarrow0,
\]
it therefore follows that $x^{\frac{1}{2}}\in\mathfrak{I}(\mathbf{A},\Gamma)$, hence $x\in\mathfrak{I}(\mathbf{A},\Gamma)$.

We thank Diego Martínez for pointing out the relevance of \cite[Proposition 11.43]{Roe03} in the context of the proof of the following proposition.

\begin{proposition} \label{FiniteDimensionalIdeal}
Let $\Gamma$ be a finite, undirected simplicial graph, and let $\mathbf{A} := (A_{v})_{v \in V\Gamma}$ be a collection of finite-dimensional C$^{\ast}$-algebras, each equipped with a faithful state $\omega_{v}$. Then the ideal $\mathfrak{I}(\mathbf{A}, \Gamma)$ coincides with the compact operators $\mathcal{K}(\mathcal{H}_{\Gamma})$ on $\mathcal{H}_{\Gamma}$.
\end{proposition}

\begin{proof}
The inclusion $\mathcal{K}(\mathcal{H}_{\Gamma}) \subseteq \mathfrak{I}(\mathbf{A}, \Gamma)$ follows directly from Proposition~\ref{IdealExistence}.

To show the reverse inclusion, let $x \in \mathfrak{I}(\mathbf{A}, \Gamma)$. By Proposition~\ref{DensityStatement}, for any $\varepsilon > 0$ there exists a finite set $\mathcal{S} \subseteq W_{\Gamma}$ and a sum $y = \sum_{\mathbf{w} \in \mathcal{S}} y_{\mathbf{w}}$, where each $y_{\mathbf{w}}$ is a finite sum of elementary operators with signature $\mathbf{w}$, such that $\|x - y\| < \frac{\varepsilon}{3}$.

It is known that Coxeter groups are exact (cf.~\cite{DranishnikovJanuszkiewicz99} and \cite[Theorem~5.1.7]{BrownOzawa08}). Therefore, by \cite[Theorem~5.1.6]{BrownOzawa08}, there exists a finite subset $\mathcal{F} \subseteq W_{\Gamma}$ and a family of unit vectors $(\eta_{\mathbf{w}})_{\mathbf{w} \in W_{\Gamma}} \subseteq \ell^2(W_{\Gamma})$ with $\operatorname{supp}(\eta_{\mathbf{w}}) \subseteq \mathcal{F} \mathbf{w}$, such that
\[
\sup \left\{ \|\eta_{\mathbf{v}} - \eta_{\mathbf{w}}\| \;\middle|\; \mathbf{v}, \mathbf{w} \in W_{\Gamma},\ \mathbf{w}\mathbf{v}^{-1} \in \mathcal{S} \right\} < \frac{\varepsilon}{3 (\# \mathcal{S}) \|y\|}.
\]

Define an isometry $V : \mathcal{H}_{\Gamma} \to \mathcal{H}_{\Gamma} \otimes \ell^2(W_{\Gamma})$ by $V\eta := \sum_{\mathbf{w} \in W_{\Gamma}} (p_{\mathbf{w}} \eta) \otimes \eta_{\mathbf{w}}$, and let $m : \mathcal{B}(\mathcal{H}_{\Gamma}) \to \mathcal{B}(\mathcal{H}_{\Gamma})$ be the unital completely positive map given by $m(T) := V^*(T \otimes 1)V$. Then for all $T \in \mathcal{B}(\mathcal{H}_{\Gamma})$ and $\mathbf{v}, \mathbf{w} \in W_{\Gamma}$, we have
\[
p_{\mathbf{v}} m(T) p_{\mathbf{w}} = \langle \eta_{\mathbf{w}}, \eta_{\mathbf{v}} \rangle (p_{\mathbf{v}} T p_{\mathbf{w}}),
\]
implying $p_{\mathbf{v}} m(T) p_{\mathbf{w}} = 0$ whenever $\mathcal{F} \mathbf{v} \cap \mathcal{F} \mathbf{w} = \emptyset$.

We estimate the approximation error:
\begin{align*}
\|y - m(y)\| 
&\leq \sum_{\mathbf{w} \in \mathcal{S}} \|y_{\mathbf{w}} - m(y_{\mathbf{w}})\| \\
&= \sum_{\mathbf{w} \in \mathcal{S}} \sup_{\mathbf{v} \in W_{\Gamma}} \left| 1 - \langle \eta_{\mathbf{v}}, \eta_{\mathbf{w} \mathbf{v}} \rangle \right| \cdot \|p_{\mathbf{w}\mathbf{v}} y_{\mathbf{w}} p_{\mathbf{v}}\| \\
&\leq \sum_{\mathbf{w} \in \mathcal{S}} \sup_{\mathbf{v} \in W_{\Gamma}} \|\eta_{\mathbf{v}} - \eta_{\mathbf{w} \mathbf{v}}\| \cdot \|p_{\mathbf{w}\mathbf{v}} y p_{\mathbf{v}}\| \\
&< \frac{\varepsilon}{3}.
\end{align*}
Hence,
\[
\|x - m(x)\| \leq \|x - y\| + \|y - m(y)\| + \|m(x - y)\| < \varepsilon.
\]

It now suffices to prove the following claim.\\

\emph{Claim.} The operator $m(x)$ is compact.

\emph{Proof of the claim.}
Choose $L \in \mathbb{N}$ such that $\mathcal{F} \mathbf{v} \cap \mathcal{F} \mathbf{w} = \emptyset$ whenever $|\mathbf{v} \mathbf{w}^{-1}| > L$, and define
\[
B_L := \{\mathbf{w} \in W_{\Gamma} \mid |\mathbf{w}| \leq L\}.
\]
Then,
\[
m(x) = \sum_{\substack{\mathbf{v}, \mathbf{w} \in W_{\Gamma}: \\ |\mathbf{v} \mathbf{w}^{-1}| \leq L}} p_{\mathbf{v}} m(x) p_{\mathbf{w}} = \sum_{\mathbf{u} \in B_L} T_{\mathbf{u}},
\]
where
\[
T_{\mathbf{u}} := \sum_{\mathbf{v} \in W_{\Gamma}} \langle \eta_{\mathbf{v}}, \eta_{\mathbf{u} \mathbf{v}} \rangle (p_{\mathbf{u} \mathbf{v}} x p_{\mathbf{v}}),
\]
with convergence in the strong operator topology.

Each partial sum $\sum_{\mathbf{v} \in W_{\Gamma} : |\mathbf{v}| \leq N} \langle \eta_{\mathbf{v}}, \eta_{\mathbf{u} \mathbf{v}} \rangle (p_{\mathbf{u} \mathbf{v}} x p_{\mathbf{v}})$ has finite rank, and
\begin{align*}
\| \sum_{\substack{\mathbf{v} \in W_{\Gamma}: \\ |\mathbf{v}| > N}} \langle \eta_{\mathbf{v}}, \eta_{\mathbf{u} \mathbf{v}} \rangle (p_{\mathbf{u} \mathbf{v}} x p_{\mathbf{v}}) \| 
&\leq \sup_{\mathbf{v} \in W_{\Gamma}:  |\mathbf{v}| > N} |\langle \eta_{\mathbf{v}}, \eta_{\mathbf{u} \mathbf{v}} \rangle| \cdot \|p_{\mathbf{u} \mathbf{v}} x p_{\mathbf{v}}\| \\
&\leq \sup_{\mathbf{v} \in W_{\Gamma}: |\mathbf{v}| > N} \|p_{\mathbf{v}} x^* p_{\mathbf{u} \mathbf{v}} x p_{\mathbf{v}}\|^{1/2} \\
&\leq \|\mathbb{E}(x^* x) P_N^{\perp}\|^{1/2}.
\end{align*}
Thus, each $T_{\mathbf{u}}$ is compact, and therefore $m(x) \in \mathcal{K}(\mathcal{H}_{\Gamma})$. This completes the proof of the claim, and hence of the proposition.
\end{proof}

The following lemma asserts an analogue of the minimality condition in the crossed product setting, see \cite[Theorem 3.19]{Klisse23-1}.

\begin{lemma} \label{MinimalityAnalogue} Let $\Gamma$ be a finite, undirected, simplicial graph with $\#V\Gamma\geq2$ whose complement $\Gamma^{c}$ is connected, and let $\mathbf{A}:=(A_{v})_{v\in V\Gamma}$ be a collection of unital C$^{\ast}$-algebras, equipped with GNS-faithful states $(\omega_{v})_{v\in V\Gamma}$. Then any ideal $I\triangleleft\mathfrak{A}(\mathbf{A},\Gamma)/\mathfrak{I}(\mathbf{A},\Gamma)$ intersects the image of $\mathcal{D}(\mathbf{A},\Gamma)$ inside $\mathfrak{A}(\mathbf{A},\Gamma)/\mathfrak{I}(\mathbf{A},\Gamma)$ trivially. \end{lemma}

\begin{proof} Our assumptions in particular imply that the right-angled Coxeter group $W_{\Gamma}$ is infinite. Denote the quotient map $\mathfrak{A}(\mathbf{A},\Gamma)\twoheadrightarrow\mathfrak{A}(\mathbf{A},\Gamma)/\mathfrak{I}(\mathbf{A},\Gamma)$ by $\pi$, and let $I\triangleleft\pi(\mathfrak{A}(\mathbf{A},\Gamma))$ be an ideal. Assume that the intersection $I\cap\pi(\mathcal{D}(\mathbf{A},\Gamma))$ is non-trivial, let $\phi$ be an arbitrary state on $\pi(\mathfrak{A}(\mathbf{A},\Gamma))$ that vanishes on $I$, set $\psi:=\phi\circ\pi$, and pick a positive element $x\in\mathcal{D}(\mathbf{A},\Gamma)$ with $\Vert x\Vert=1$ and $\pi(x)\in I\cap\pi(\mathcal{D}(\mathbf{A},\Gamma))$.

For every $0<\varepsilon<1$, we find by Corollary \ref{DiagonalCharacterization} an element $y$ that can be written as a linear combination of summands of the form $(a_{1}^{\dagger}\cdots a_{k}^{\dagger})d(b_{1}^{\dagger}\cdots b_{k}^{\dagger})^{\ast}$ with $k\in\mathbb{N}$, $(u_{1},\ldots,u_{k})\in\mathcal{W}_{\text{red}}$, $a_{i},b_{i}\in A_{u_{i}}^{\circ}$, $d\in\mathcal{D}_{0}(\mathbf{A},\Gamma)$ such that $\Vert x-y\Vert<\frac{\varepsilon^{2}}{25}$. Observe that $y$ acts on a finite number of tensor legs, meaning that there exists $M\in\mathbb{N}$ such that for every $\mathbf{w}\in W_{\Gamma}$ with $|\mathbf{w}|>M$, there exists a decomposition $\mathbf{w}=\mathbf{w}_{1}\mathbf{w}_{2}$ with $\mathbf{w}_{1},\mathbf{w}_{2}\in W_{\Gamma}$, $|\mathbf{w}|=|\mathbf{w}_{1}|+|\mathbf{w}_{2}|$, $|\mathbf{w}_{1}|=M$ such that $y(\xi_{1}\otimes\xi_{2})=(y\xi_{1})\otimes\xi_{2}$ for all $\xi_{1}\in\mathcal{H}_{\mathbf{w}_{1}}$, $\xi_{2}\in\mathcal{H}_{\mathbf{w}_{2}}$. Here we view $\xi_{1}\otimes\xi_{2}$ as an element in $\mathcal{H}_{\mathbf{w}}^{\circ}\cong\mathcal{H}_{\mathbf{w}_{1}}^{\circ}\otimes\mathcal{H}_{\mathbf{w}_{2}}^{\circ}$.

We proceed by proving a claim.\\

\emph{Claim 1}. There exists a vertex $v_{0}\in V\Gamma$ with $\psi(Q_{v_{0}})\ne1$.

\emph{Proof of the claim:} Assume that $\psi(Q_{v})=1$ for every $v\in V\Gamma$. It follows that $Q_{v}$ is contained in the multiplicative domain of the state $\psi$ (see, e.g.,\ \cite[Proposition 1.5.7]{BrownOzawa08}) for every $v\in V\Gamma$. In particular, for $v,v^{\prime}\in V\Gamma$ with $(v,v^{\prime})\in E\Gamma^{c}$ we obtain 
\[
1=\psi(Q_{v})\psi(Q_{v^{\prime}})=\psi(Q_{v}Q_{v^{\prime}})=\psi(0)=0,
\]
which leads to a contradiction. This proves the claim.\\

Fix $v_{0}\in V\Gamma$ as in Claim 1 and choose a closed walk $(v_{1},...,v_{n})\in V\Gamma\times\dots\times V\Gamma$ in the complement $\Gamma^{c}$ that covers the whole graph and satisfies $v_{1}=v_{0}$. Set furthermore $\mathbf{g}:=v_{1}\cdots v_{n}\in W_{\Gamma}$. Now let $N>M+n$ be an arbitrary integer. Since $y$ acts diagonally, we find a group element $\mathbf{w}\in W_{\Gamma}$ of length $l:=|\mathbf{w}|>N$ and a unit vector $\xi\in\mathcal{H}_{\mathbf{w}}^{\circ}$ with $\Vert yP_{N}^{\perp}\Vert<\Vert y\xi\Vert+\frac{\varepsilon^{2}}{25}$.

By a density argument, we can assume that $\xi$ is given by a linear combination of the form $\xi=\sum_{i=1}^{n}c_{i,1}^{\dagger}\cdots c_{i,l}^{\dagger}\Omega$ with suitable $n\in\mathbb{N}$, $(w_{1},...,w_{l})\in\mathcal{W}_{\text{red}}$, $c_{i,j}\in A_{w_{j}}^{\circ}$, where $\mathbf{w}=w_{1}\cdots w_{l}$. The discussion above implies that, by possibly altering the tail of $\mathbf{w}_{2}$, we can furthermore assume that $\mathbf{g}\leq\mathbf{w}^{-1}$.\\

\emph{Claim 2}. The following identity holds: 
\[
\sum_{i,j}(c_{i,1}^{\dagger}\cdots c_{i,l}^{\dagger})^{\ast}(y^{\ast}y)(c_{j,1}^{\dagger}\cdots c_{j,l}^{\dagger})=\Vert y\xi\Vert^{2}Q_{v_{0}}^{\perp}.
\]

\emph{Proof of the claim:} For $\mathbf{u},\mathbf{v}\in W_{\Gamma}$ and $\eta\in\mathcal{H}_{\mathbf{u}}^{\circ}$, $\zeta\in\mathcal{H}_{\mathbf{v}}^{\circ}$ with $v_{0}\nleq\mathbf{u}$, $v_{0}\nleq\mathbf{v}$, we have by the previous discussion that 
\begin{eqnarray*}
 &  & \left\langle \sum_{i,j}(c_{i,1}^{\dagger}\cdots c_{i,l}^{\dagger})^{\ast}(y^{\ast}y)(c_{j,1}^{\dagger}\cdots c_{j,l}^{\dagger})\eta,\zeta\right\rangle \\
 & = & \sum_{i,j}\left\langle (yc_{i,1}^{\dagger}\cdots c_{i,l}^{\dagger}\Omega)\otimes\eta,(yc_{j,1}^{\dagger}\cdots c_{j,l}^{\dagger}\Omega)\otimes\zeta\right\rangle \\
 & = & \sum_{i,j}\left\langle yc_{i,1}^{\dagger}\cdots c_{i,l}^{\dagger}\Omega,yc_{j,1}^{\dagger}\cdots c_{j,l}^{\dagger}\Omega\right\rangle \langle\eta,\zeta\rangle\\
 & = & \Vert y\xi\Vert^{2}\langle Q_{v_{0}}^{\perp}\eta,\zeta\rangle,
\end{eqnarray*}
while 
\[
\left\langle \sum_{i,j}(c_{i,1}^{\dagger}\cdots c_{i,l}^{\dagger})^{\ast}(y^{\ast}y)(c_{j,1}^{\dagger}\cdots c_{j,l}^{\dagger})\eta,\zeta\right\rangle =0
\]
for all $\mathbf{u},\mathbf{v}\in W_{\Gamma}$, $\eta\in\mathcal{H}_{\mathbf{u}}^{\circ}$, $\zeta\in\mathcal{H}_{\mathbf{v}}^{\circ}$ with $v_{0}\leq\mathbf{u}$ or $v_{0}\leq\mathbf{v}$. This implies the claim.\\

By the choice of $v_{0}$, we can consider the positive linear functional $\overline{\psi}$ on $\mathfrak{A}(\mathbf{A},\Gamma)$ given by 
\[
z\mapsto\psi(Q_{v_{0}}^{\perp})^{-1}\sum_{i,j}\psi\left((c_{i,1}^{\dagger}\cdots c_{i,l}^{\dagger})^{\ast}z(c_{j,1}^{\dagger}\cdots c_{j,l}^{\dagger})\right).
\]

In the same way as in the proof of Claim 2 above, we deduce 
\[
\Vert\overline{\psi}\Vert=\overline{\psi}(1)=\psi(Q_{v_{0}}^{\perp})^{-1}\psi(Q_{v_{0}}^{\perp})=1,
\]
so that $\overline{\psi}$ is a state. Furthermore, we have $\overline{\psi}(x^{\ast}x)=0$ while by Claim 2, 
\[
\Vert y\xi\Vert^{2}=\overline{\psi}(y^{\ast}y)\leq\Vert y^{\ast}y-x^{\ast}x\Vert+\overline{\psi}(x^{\ast}x)\leq(\Vert y\Vert+\Vert x\Vert)\Vert x-y\Vert<\left(\frac{\varepsilon^{2}}{25}+2\right)\frac{\varepsilon^{2}}{25}<\frac{9\varepsilon^{2}}{25},
\]
implying that for all $N>M+n$, 
\[
\Vert\mathbb{E}(x)P_{N}^{\perp}\Vert=\Vert xP_{N}^{\perp}\Vert<\frac{\varepsilon^{2}}{25}+\Vert yP_{N}^{\perp}\Vert<\frac{2\varepsilon^{2}}{25}+\Vert y\xi\Vert<\frac{2\varepsilon^{2}}{25}+\frac{3}{5}\varepsilon<\varepsilon.
\]
Since $0<\varepsilon<1$ was arbitrary, we obtain $\lim_{N\rightarrow\infty}\Vert\mathbb{E}(x)P_{N}^{\perp}\Vert=0$. But by definition, this means that $x^{\frac{1}{2}}\in\mathfrak{I}(\mathbf{A},\Gamma)$, i.e., $\pi(x)=0$. This finishes the proof. \end{proof}

From the previous lemma, we deduce the main theorem of the present subsection.

\begin{theorem} \label{MaximalityTheorem} Let $\Gamma$ be a finite, undirected, simplicial graph with $\#V\Gamma\geq3$ whose complement $\Gamma^{c}$ is connected, and let $\mathbf{A}:=(A_{v})_{v\in V\Gamma}$ be a collection of unital C$^{\ast}$-algebras equipped with GNS-faithful states $(\omega_{v})_{v\in V\Gamma}$. Then $\mathfrak{I}(\mathbf{A},\Gamma)\triangleleft\mathfrak{A}(\mathbf{A},\Gamma)$ is a maximal ideal, i.e., $\mathfrak{A}(\mathbf{A},\Gamma)/\mathfrak{I}(\mathbf{A},\Gamma)$ is a simple C$^{\ast}$-algebra. \end{theorem}

\begin{proof} Denote the quotient map $\mathfrak{A}(\mathbf{A},\Gamma)\twoheadrightarrow\mathfrak{A}(\mathbf{A},\Gamma)/\mathfrak{I}(\mathbf{A},\Gamma)$ by $\pi$, and assume that there exists a non-trivial ideal $I\triangleleft\pi(\mathfrak{A}(\mathbf{A},\Gamma))$. We may then choose a positive element $x\in\mathfrak{A}(\mathbf{A},\Gamma)$ with $\Vert x\Vert=1$ and $0\neq\pi(x)\in I$. For every $0<\varepsilon<1$, we find by Proposition \ref{DensityStatement} an element $y$ in the $\ast$-subalgebra $\mathcal{A}$ of $\mathfrak{A}(\mathbf{A},\Gamma)$ generated by $\mathbf{A}_{\Gamma}$ and all projections $(Q_{v})_{v\in V\Gamma}$ with $\Vert x-y\Vert<\frac{\varepsilon^{2}}{25}$.

As in the proof of Lemma \ref{MinimalityAnalogue}, there exists $M\in\mathbb{N}$ such that every element $\mathbf{w}\in W_{\Gamma}$ with $|\mathbf{w}|>M$ admits a decomposition $\mathbf{w}=\mathbf{w}_{1}\mathbf{w}_{2}$ with $\mathbf{w}_{1},\mathbf{w}_{2}\in W_{\Gamma}$, $|\mathbf{w}|=|\mathbf{w}_{1}|+|\mathbf{w}_{2}|$, $|\mathbf{w}_{1}|=M$, such that $\mathbb{E}(y)(\xi_{1}\otimes\xi_{2})=(\mathbb{E}(y)\xi_{1})\otimes\xi_{2}$ for all $\xi_{1}\in\mathcal{H}_{\mathbf{w}_{1}}$, $\xi_{2}\in\mathcal{H}_{\mathbf{w}_{2}}$, where $\xi_{1}\otimes\xi_{2}$ is viewed as an element in $\mathcal{H}_{\mathbf{w}}^{\circ}\cong\mathcal{H}_{\mathbf{w}_{1}}^{\circ}\otimes\mathcal{H}_{\mathbf{w}_{2}}^{\circ}$.

By Proposition \ref{DensityStatement}, the element $y^{\ast}y$ can be written as a linear combination of summands of the form $y_{i}:=(a_{i,1}^{\dagger}\cdots a_{i,k_{i}}^{\dagger})d_{i}(b_{i,1}^{\dagger}\cdots b_{i,l_{i}}^{\dagger})^{\ast}$ with $1\leq i\leq n$, $k_{i},l_{i}\in\mathbb{N}$, $(u_{i,1},\ldots,u_{i,k_{i}}),(v_{i,1},\ldots,v_{i,k_{i}})\in\mathcal{W}_{\text{red}}$, $a_{i,j}\in A_{u_{i,j}}^{\circ}$, $b_{i,j}\in A_{v_{i,j}}^{\circ}$, and $d_{i}\in\mathcal{D}_{0}(\mathbf{A},\Gamma)$. Denote the (finite) set of signatures of these summands by $\mathcal{S}:=\{\Sigma(y_{i})\mid1\leq i\leq n\}$ and let $N>M$ be an arbitrary integer. We then find a group element $\mathbf{w}\in W_{\Gamma}$ of length $l:=|\mathbf{w}|>N$ and a unit vector $\xi\in\mathcal{H}_{\mathbf{w}}^{\circ}$ with $\Vert\mathbb{E}(y)P_{N}^{\perp}\Vert<\Vert\mathbb{E}(y)\xi\Vert_{2}+\frac{\varepsilon^{2}}{25}$. One can assume that $\xi$ is given by a linear combination of the form $\xi=\sum_{j=1}^{m}c_{j,1}^{\dagger}\cdots c_{j,l}^{\dagger}\Omega$ with suitable elements $m\in\mathbb{N}$, $(w_{1},\ldots,w_{l})\in\mathcal{W}_{\text{red}}$, $c_{i,j}\in A_{w_{j}}^{\circ}$, where $\mathbf{w}=w_{1}\cdots w_{l}$.

Choose an element $\mathbf{v}\in W_{\Gamma}$ and a closed path $(t_{1},\ldots,t_{n})\in V\Gamma\times\cdots\times V\Gamma$ as in Lemma \ref{TopologicalFrenessImplication}. Denote $z:=\lim_{L}(\mathbf{w}\mathbf{v}(t_{1}\cdots t_{n})^{L})\in\partial(W_{\Gamma},S_{\Gamma})$, and write for notational ease $\mathbf{v}(t_{1}\cdots t_{n})^{L}=s_{1}\cdots s_{k_{L}}$ with $(s_{1},\ldots,s_{k_{L}})\in\mathcal{W}_{\text{red}}$ for $L\in\mathbb{N}$. For every $1\leq j\leq k_{L}$, we choose an element $d_{j}\in A_{s_{j}}^{\circ}$ with $\omega_{s_{j}}(d_{j}^{\ast}d_{j})=1$ and set 
\[
\eta_{L}:=\sum_{j=1}^{m}c_{j,1}^{\dagger}\cdots c_{j,l}^{\dagger}d_{1}^{\dagger}\cdots d_{k_{L}}^{\dagger}\Omega\in\mathcal{H}_{\mathbf{w}\mathbf{v}(t_{1}\cdots t_{n})^{L}}^{\circ}.
\]

By the discussion above and the choice of $N$, we have that $\eta_{L}$ is a unit vector with $\Vert\mathbb{E}(y)P_{N}^{\perp}\Vert<\Vert\mathbb{E}(y)\eta_{L}\Vert_{2}+\frac{\varepsilon^{2}}{25}$. Denote the restriction of the corresponding vector state to $\mathfrak{A}(\mathbf{A},\Gamma)$ by $\omega_{L}$. We can then go over to a subnet of $(\omega_{L})_{L\in\mathbb{N}}$ with weak$^{\ast}$-limit $\omega$. Note that for every positive element $T\in\mathfrak{A}(\mathbf{A},\Gamma)$, 
\[
\omega_{L}(T) \leq \Vert \mathbb{E}(T) \eta_{L}\Vert \leq\Vert \mathbb{E}(T^{\frac{1}{2}}T^{\frac{1}{2}}) P_{L+l}^\perp \Vert,
\]
so that $\omega$ in particular vanishes on the ideal $\mathfrak{I}(\mathbf{A},\Gamma)$. We proceed by proving the following two claims.\\

\emph{Claim 1.} There exists a state $\phi$ on $\mathfrak{A}(\mathbf{A},\Gamma)$ with $\phi(T+a)=\omega(T)$ for $T\in\mathcal{D}(\mathbf{A},\Gamma)$, $a\in\pi^{-1}(I)$.

\emph{Proof of the claim:} First note that the map $\phi_{0}:\mathcal{D}(\mathbf{A},\Gamma)+\pi^{-1}(I)\rightarrow\mathbb{C}$, $T+a\mapsto\omega(T)$ is well-defined. Indeed, if $T_{1},T_{2}\in\mathcal{D}(\mathbf{A},\Gamma)$, $a_{1},a_{2}\in\pi^{-1}(I)$ with $T_{1}+a_{1}=T_{2}+a_{2}$, then $T_{1}-T_{2}=a_{1}-a_{2}\in\mathcal{D}(\mathbf{A},\Gamma)\cap\pi^{-1}(I)$, implying $\pi(T_{1}-T_{2})\in\pi(\mathcal{D}(\mathbf{A},\Gamma))\cap I$. But by Lemma \ref{MinimalityAnalogue}, the intersection $\pi(\mathcal{D}(\mathbf{A},\Gamma))\cap I$ is trivial, so $T_{1}-T_{2}\in\mathfrak{I}(\mathbf{A},\Gamma)$. Thus, $\omega(T_{1})=\omega(T_{2})$, and $\phi_{0}$ is well-defined. It is clearly positive and satisfies $\phi_{0}(1)=\omega(1)=1$, so it is a state. Extending $\phi_{0}$ to $\mathfrak{A}(\mathbf{A},\Gamma)$ yields the claim.\\

\emph{Claim 2.} Let $\phi$ be a state as in the previous claim. Then every summand $y_{i}$ with $1\leq i\leq n$ and $\Sigma(y_{i})\in\mathcal{S}\setminus\{e\}$ satisfies $\phi(y_{i})=0$.

\emph{Proof of the claim:} Consider a summand $y_{i}=(a_{i,1}^{\dagger}\cdots a_{i,k}^{\dagger})d_{i}(b_{i,1}^{\dagger}\cdots b_{i,k}^{\dagger})^{\ast}$ with $1\leq i\leq n$ and $\Sigma(y_{i})\in\mathcal{S}\setminus\{e\}$. Note that for $L_{0}<L$ and $T:=\sum_{j=1}^{m}(c_{j,1}^{\dagger}\cdots c_{j,l}^{\dagger}d_{1}^{\dagger}\cdots d_{k_{L_{0}}}^{\dagger})$ the element $TT^{\ast}$ is contained in $\mathcal{D}(\mathbf{A},\Gamma)$ with 
\begin{eqnarray*}
T^{\ast}\eta_{L} & = & \sum_{j_{1},j_{2}=1}^{m}\left((c_{j_{1},1}^{\dagger}\cdots c_{j_{1},l}^{\dagger}d_{1}^{\dagger}\cdots d_{k_{L_{0}}}^{\dagger})^{\ast}(c_{j_{2},1}^{\dagger}\cdots c_{j_{2},l}^{\dagger}d_{1}^{\dagger}\cdots d_{k_{L_{0}}}^{\dagger})\right)(d_{k_{L_{0}}+1}^{\dagger}\cdots d_{L}^{\dagger}\Omega)\\
 & = & \sum_{j_{1},j_{2}=1}^{m}\left(\omega_{w_{1}}(c_{j_{1},1}^{\ast}c_{j_{2},1})\cdots\omega_{w_{r}}(c_{j_{1},l}^{\ast}c_{j_{2},l})\times\omega_{s_{1}}(d_{1}^{\ast}d_{1})\cdots\omega_{s_{i_{0}}}(d_{k_{L_{0}}}^{\ast}d_{k_{L_{0}}})\right)(d_{k_{L_{0}}+1}^{\dagger}\cdots d_{L}^{\dagger}\Omega)\\
 & = & \Vert\eta_{L_{0}}\Vert^{2}(d_{k_{L_{0}}+1}^{\dagger}\cdots d_{L}^{\dagger}\Omega)\\
 & = & (d_{k_{L_{0}}+1}^{\dagger}\cdots d_{L}^{\dagger})\Omega
\end{eqnarray*}
and 
\[
TT^{\ast}\eta_{L}=\sum_{j=1}^{m}(c_{j,1}^{\dagger}\cdots c_{j,r}^{\dagger}d_{1}^{\dagger}\cdots d_{L}^{\dagger})\Omega=\eta_{L}.
\]
It follows that 
\[
\langle TT^{\ast}\eta_{L},\eta_{L}\rangle=\Vert T^{\ast}\eta_{L}\Vert^{2}=1\quad\text{ and }\quad\langle(TT^{\ast})^{2}\eta_{L},\eta_{L}\rangle=\Vert(TT^{\ast})\eta_{L}\Vert^{2}=1
\]
so that 
\[
\omega((TT^{\ast})^{2})=1=|\omega(TT^{\ast})|^{2}.
\]
But then $TT^{\ast}$ is contained in the multiplicative domain of $\phi$. Since we have chosen $\mathbf{v}$ and $(t_{1},\ldots,t_{n})$ according to Proposition \ref{TopologicalFrenessImplication}, we have that $\mathbf{w}\mathbf{v}(t_{1}\cdots t_{n})^{L_{0}}\nleq\Sigma(x_{i})\mathbf{w}\mathbf{v}(t_{1}\cdots t_{n})^{L_{0}}$ which implies $(TT^{\ast})x_{i}(TT^{\ast})=0$. Therefore, 
\[
\phi(y_{i})=\phi(TT^{\ast})\phi(y_{i})\phi(TT^{\ast})=\phi((TT^{\ast})y_{i}(TT^{\ast}))=0.
\]
This implies the claim.\\

Choose a state $\phi$ as in the claims above and note that $\phi(y^{\ast}y)=\phi\circ\mathbb{E}(y^{\ast}y)$, where $\mathbb{E}$ is the conditional expectation in Theorem \ref{GaugeActionTheorem}. Since 
\[
\Vert\mathbb{E}(x)P_{N}^{\perp}\Vert\leq\frac{\varepsilon^{2}}{25}+\Vert\mathbb{E}(y)P_{N}^{\perp}\Vert<\frac{2\varepsilon^{2}}{25}+\Vert\mathbb{E}(y)\eta_{L}\Vert\leq\frac{2\varepsilon^{2}}{25}+(\omega_{L}\circ\mathbb{E}(y^{\ast}y))^{\frac{1}{2}},
\]
we get 
\[
\Vert\mathbb{E}(x)P_{N}^{\perp}\Vert<\frac{2\varepsilon^{2}}{25}+(\omega\circ\mathbb{E}(y^{\ast}y))^{\frac{1}{2}}.
\]
In combination with 
\[
\omega\circ\mathbb{E}(y^{\ast}y)\leq\Vert y^{\ast}y-x^{\ast}x\Vert+\phi(x^{\ast}x)\leq(\Vert y\Vert+\Vert x\Vert)\Vert x-y\Vert\leq\left(\frac{\varepsilon^{2}}{25}+2\right)\frac{\varepsilon^{2}}{25}<\frac{9\varepsilon^{2}}{25},
\]
this gives 
\[
\Vert\mathbb{E}(x)P_{N}^{\perp}\Vert<\frac{2\varepsilon^{2}}{25}+\frac{3\varepsilon}{5}<\varepsilon.
\]
Since $0<\varepsilon<1$ was arbitrary, we obtain $\lim_{N\rightarrow\infty}\Vert\mathbb{E}(x)P_{N}^{\perp}\Vert=0$. But this means that $x^{\frac{1}{2}}\in\mathfrak{I}(\mathbf{A},\Gamma)$, i.e., $\pi(x)=0$, contradicting our assumption and completing the proof. \end{proof}

\begin{remark} An application of Proposition \ref{TensorDecomposition} implies that the statement in Theorem \ref{MaximalityTheorem} extends to general graphs whose complements decompose as disjoint unions of induced subgraphs containing at least three vertices. \end{remark}

\vspace{3mm}


\section{Applications to Graph Product C$^{\ast}$-Algebras}

\vspace{3mm}

In this section, we illustrate how structural properties of the C$^{\ast}$-algebras constructed in Section~\ref{sec:Main-construction} can be leveraged to gain insights into the corresponding graph product C$^{\ast}$-algebras.

\vspace{3mm}


\subsection{Nuclearity and Exactness}

As a direct consequence of Theorem~\ref{NuclearityExactness}, we obtain an alternative proof of the preservation of exactness under graph products of C$^{\ast}$-algebras, originally established in \cite[Corollary~2.17]{CaspersFima17}. Notably, our approach circumvents the use of techniques developed in \cite{Dykema04} (see also \cite{DykemaShlyakhtenko01}).

\begin{corollary}[{\cite[Corollary~2.17]{CaspersFima17}}] \label{GraphProductExactness}
Let $\Gamma$ be a finite, undirected, simplicial graph, and let $\mathbf{A} := (A_v)_{v \in V\Gamma}$ be a collection of unital C$^{\ast}$-algebras, each equipped with a GNS-faithful state  $\omega_{v}$. Then the graph product C$^\ast$-algebra $\mathbf{A}_{\Gamma}$ is exact if and only if each vertex algebra $A_v$ is exact.
\end{corollary}

\begin{proof}
For the ``if'' direction assume that each vertex algebra $A_v$ is exact. By Theorem~\ref{NuclearityExactness}, the ambient C$^{\ast}$-algebra $\mathfrak{A}(\mathbf{A}, \Gamma)$ is exact. Since exactness is preserved under the passage to C$^{\ast}$-subalgebras, it follows that $\mathbf{A}_{\Gamma}$ is exact.
The proof of the ``only if'' direction is trivial.
\end{proof}

Furthermore, our methods yield the following refinement of \cite[Theorem~H]{BorstCaspersChen24}, characterizing nuclearity of graph products under suitable assumptions.

\begin{corollary} \label{GraphProductNuclearity}Let $\Gamma$ be a finite, undirected, simplicial graph and let $\mathbf{A}:=(A_{v})_{v\in V\Gamma}$ be a collection of unital C$^{\ast}$-algebras, each equipped with a GNS-faithful state $\omega_{v}$. Suppose that for each $v \in V\Gamma$, the algebra $A_v \subseteq \mathcal{B}(\mathcal{H}_v)$ contains the compact operators. Then the graph product C$^{\ast}$-algebra $\mathbf{A}_\Gamma$ is nuclear if and only if each vertex algebra $A_v$ is nuclear.
\end{corollary}

\begin{proof}
For each $v \in V\Gamma$, let $p_v$ denote the orthogonal projection onto $\mathbb{C} \xi_v \subseteq \mathcal{H}_v$. Then for every $x \in A_v$, we have $p_v x p_v = \omega_v(x) p_v$. Viewed within $\mathcal{B}(\mathcal{H}_\Gamma)$, it follows that $p_v^\perp = Q_v$, implying that $\mathbf{A}_\Gamma = \mathfrak{A}(\mathbf{A}, \Gamma)$. The conclusion then follows directly from Theorem~\ref{NuclearityExactness}.
\end{proof}

\vspace{3mm}


\subsection{Simplicity of Graph Product C$^{\ast}$-Algebras} \label{GraphSimplicity}

The results and techniques developed in this subsection are inspired by \cite{Klisse23-2}, which characterizes the simplicity of right-angled Hecke C$^{\ast}$-algebras in terms of the growth series of the underlying Coxeter group. Earlier results on the simplicity of free product C$^{\ast}$-algebras can be found in \cite{McClanahan94, Avitzour82, Dykema99}.\\

Let $\Gamma$ be a finite, undirected, simplicial graph, and let $\mathbf{A} := (A_{v})_{v \in V\Gamma}$ be a collection of unital C$^{\ast}$-algebras, equipped with GNS-faithful states $(\omega_v)_{v \in V\Gamma}$. When the complement graph $\Gamma^{c}$ is disconnected, $\Gamma$ decomposes as a \emph{graph join} $\Gamma = \Gamma_1 + \Gamma_2$, where $\Gamma_1$ and $\Gamma_2$ are disjoint induced subgraphs of $\Gamma$ and $\Gamma_1 + \Gamma_2$ denotes the join obtained from the union of $\Gamma_{1}$ and $\Gamma_{2}$ by adding edges between every vertex of $\Gamma_1$ and every vertex of $\Gamma_2$.

In this case, the associated graph product C$^{\ast}$-algebra admits a tensor product decomposition:
\[
\star_{v, \Gamma}(A_{v}, \omega_{v}) \cong \left(\star_{v, \Gamma_1}(A_{v}, \omega_{v})\right) \otimes \left(\star_{v, \Gamma_2}(A_{v}, \omega_{v})\right).
\]
Since the tensor product of two C$^{\ast}$-algebras is simple if and only if both factors are simple (see, e.g., \cite[II.9.5.3]{Blackadar06}), in the context of this subsection we may henceforth assume that $\Gamma^{c}$ is connected.\\

Recall that for each tuple $z := (z_v)_{v \in V\Gamma} \in \mathbb{C}^{V\Gamma}$ and every group element $\mathbf{v} \in W_{\Gamma}$ corresponding to a reduced expression $(v_1, \dots, v_n) \in \mathcal{W}_{\text{red}}$ we defined $z_{\mathbf{v}} := z_{v_1} \cdots z_{v_n}$. This expression is independent of the choice of reduced word representing $\mathbf{v}=v_1 \cdots v_n \in W_\Gamma$. Let $\Gamma(z) := \sum_{\mathbf{w} \in W_{\Gamma}} z_{\mathbf{w}}$ denote the multivariate \emph{growth series}, and let $\mathcal{R}(\Gamma)$ denote its region of convergence in $\mathbb{C}^{V\Gamma}$, with $\overline{\mathcal{R}(\Gamma)}$ denoting the closure. We also denote by $\pi: \mathfrak{A}(\mathbf{A}, \Gamma) \twoheadrightarrow \mathfrak{A}(\mathbf{A}, \Gamma)/\mathfrak{I}(\mathbf{A}, \Gamma)$ the canonical quotient map, and write $\widetilde{Q}_v := \pi(Q_v)$ for $v \in V\Gamma$, where $\mathfrak{I}(\mathbf{A}, \Gamma)$ is the ideal constructed in Subsection~\ref{subsec:Ideal}.\\

Combining \cite[Proposition 2.10]{Klisse23-2} with Lemma~\ref{IdentificationLemma}, we obtain the following key auxiliary result.

\begin{proposition} \label{MainProposition}
Let $\Gamma$ be a finite, undirected, simplicial graph, and let $\mathbf{A} := (A_v)_{v \in V\Gamma}$ be a collection of unital, exact C$^{\ast}$-algebras, equipped with GNS-faithful states $(\omega_v)_{v \in V\Gamma}$. Suppose that the complement $\Gamma^{c}$ is connected, and let $q \in \mathbb{R}_{>0}^{V\Gamma} \setminus \overline{\mathcal{R}(\Gamma)}$. Let $(v_1, \dots, v_n) \in V\Gamma \times \cdots \times V\Gamma$ be a walk in $\Gamma^{c}$ that covers the whole graph. Then for any state $\phi$ on $\mathcal{B}(\mathcal{H}_{\Gamma})$, there exists a sequence $(\mathbf{w}_i)_{i \in \mathbb{N}} \subseteq W_{\Gamma}$ of group elements with increasing word length such that $v_1 \cdots v_n \leq \mathbf{w}_i^{-1}$ for all $i \in \mathbb{N}$ and $ q_{\mathbf{w}_i}^{-1} \phi(Q_{\mathbf{w}_i}) \to 0$.
\end{proposition}

\begin{proof}
Let $\Psi$ denote the $\ast$-isomorphism provided by Lemma~\ref{IdentificationLemma}, and let $\phi$ be a state on $\mathcal{B}(\mathcal{H}_{\Gamma})$. Composing the restriction of $\phi$ to the C$^{\ast}$-subalgebra generated by the projections $(Q_{\mathbf{w}})_{\mathbf{w} \in W_{\Gamma}}$ with $\Psi$, we obtain a state $\psi$ on $\mathcal{D}(W_{\Gamma}, S_{\Gamma})$. Applying \cite[Proposition~2.10]{Klisse23-2} to $\psi$ yields a sequence $(\mathbf{w}_i)_{i \in \mathbb{N}} \subseteq W_{\Gamma}$ of group elements of increasing word length such that $v_1 \cdots v_n \leq \mathbf{w}_i^{-1}$ for all $i$ and $q_{\mathbf{w}_i}^{-1} \phi(Q_{\mathbf{w}_i}) = q_{\mathbf{w}_i}^{-1} \psi (P_{\mathbf{w}_i}) \to 0$, as desired.
\end{proof}

In addition to Proposition~\ref{MainProposition}, we will require further information on the combinatorial structure of the projections $(Q_{\mathbf{w}})_{\mathbf{w} \in W_{\Gamma}}$.

\begin{lemma} \label{ConjugationLemma}
Let $\Gamma$ be a finite, undirected, simplicial graph, and let $\mathbf{A} := (A_{v})_{v \in V\Gamma}$ be a collection of unital C$^{\ast}$-algebras, equipped with GNS-faithful states $(\omega_{v})_{v \in V\Gamma}$. Then, for every $v \in V\Gamma$, the following identities hold:
\begin{enumerate}
\item $a^{\ast}Q_{v}^{\perp}a \leq \omega_{v}(aa^{\ast}) Q_{v}$ for all $a \in A_{v}^{\circ}$;
\item $a^{\ast}Q_{\mathbf{w}}a \leq \omega_{v}(aa^{\ast}) Q_{v\mathbf{w}}$ for all $a \in A_{v}^{\circ}$ and $\mathbf{w} \notin C_{W_{\Gamma}}(v)$ with $v \nleq \mathbf{w}$.
\end{enumerate}
Here, $C_{W_{\Gamma}}(v) := \{\mathbf{w} \in W_{\Gamma} \mid v\mathbf{w} = \mathbf{w}v\}$ denotes the \emph{centralizer} of $v$ in $W_{\Gamma}$.
\end{lemma}

\begin{proof}
\emph{About (1):} Every $a \in A_{v}^{\circ}$ admits a decomposition of the form $a = a^{\dagger} + ((a^{\ast})^{\dagger})^{\ast} + \mathfrak{d}(a)$. Then, using Proposition \ref{MainIdentities}, we compute
\[
a^{\ast}Q_{v}^{\perp}a = (a^{\ast})^{\dagger}((a^{\ast})^{\dagger})^{\ast} = Q_{v} \left( (a^{\ast})^{\dagger}((a^{\ast})^{\dagger})^{\ast} \right) Q_{v} \leq \| (a^{\ast})^{\dagger} \|^{2} Q_{v} = \omega_{v}(aa^{\ast}) Q_{v}.
\]

\emph{About (2):} By $Q_{v} Q_{\mathbf{w}} = Q_{\mathbf{w}} Q_{v} = 0$ and applying Lemma \ref{ActionLemma} and Proposition \ref{MainIdentities}, we find
\[
\begin{aligned}
a^{\ast}Q_{\mathbf{w}}a &= (a^{\ast})^{\dagger} Q_{\mathbf{w}} ((a^{\ast})^{\dagger})^{\ast} \\
&= (v.Q_{\mathbf{w}}) \left( (a^{\ast})^{\dagger}((a^{\ast})^{\dagger})^{\ast} \right) (v.Q_{\mathbf{w}}) \\
&= Q_{v\mathbf{w}} \left( (a^{\ast})^{\dagger}((a^{\ast})^{\dagger})^{\ast} \right) Q_{v\mathbf{w}} \\
&\leq  \| (a^{\ast})^{\dagger} \|^{2} Q_{v\mathbf{w}} \\
&= \omega_{v}(aa^{\ast}) Q_{v\mathbf{w}}.
\end{aligned}
\]
This concludes the proof.
\end{proof}

\begin{proposition} \label{StateExistence}
Let $\Gamma$ be a finite, undirected, simplicial graph whose complement $\Gamma^{c}$ is connected, and let $\mathbf{A} := (A_{v})_{v \in V\Gamma}$ be a collection of unital C$^{\ast}$-algebras, each equipped with a GNS-faithful state $\omega_{v}$. Suppose that for every $v \in V\Gamma$, there exist $a_{v} \in \text{\emph{ker}}(\omega_v)$ and $q_{v} > 0$ such that $a_{v} a_{v}^{\ast} \geq q_{v} \omega_{v}(a_{v}^{\ast} a_{v}) 1 > 0$, and that $(q_{v})_{v \in V\Gamma} \in \mathbb{R}_{>0}^{V\Gamma} \setminus \overline{\mathcal{R}(\Gamma)}$. Let $\pi(\mathbf{A}_{\Gamma}) \subseteq \mathcal{A} \subseteq \mathfrak{A}(\mathbf{A},\Gamma)/\mathfrak{I}(\mathbf{A},\Gamma)$ be an intermediate C$^{\ast}$-algebra. Then, for every closed two-sided ideal $I \subseteq \mathcal{A}$ and every vertex $v \in V\Gamma$, there exists a state $\phi$ on $\mathfrak{A}(\mathbf{A},\Gamma)/\mathfrak{I}(\mathbf{A},\Gamma)$ such that $\phi(I) = 0$, and $\phi(\widetilde{Q}_{v}) = 1$. \end{proposition}

\begin{proof}
Let $\vartheta$ be a state on $\mathfrak{A}(\mathbf{A},\Gamma)/\mathfrak{I}(\mathbf{A},\Gamma)$ vanishing on the ideal $I$. Composing with $\pi$ and extending to a state $\psi$ on $\mathcal{B}(\mathcal{H}_{\Gamma})$, we proceed as follows.

Let $(v_{1}, \ldots, v_{n})$ be a closed walk in $\Gamma^{c}$ that covers the whole graph, with $v_{1} = v$. Let $q := (q_{w})_{w \in V\Gamma}$. By Proposition \ref{MainProposition}, there exists a sequence $(\mathbf{w}_{i})_{i \in \mathbb{N}} \subseteq W_{\Gamma}$ of increasing word length such that $v_{1} \cdots v_{n} \leq \mathbf{w}_{i}^{-1}$ and $q_{\mathbf{w}_{i}}^{-1} \psi(Q_{\mathbf{w}_{i}}) \to 0$. Define $a := (a_{w})_{w \in V\Gamma}$ and note that the assumptions imply $\vartheta\big(\pi(a_{\mathbf{w}_{i}} a_{\mathbf{w}_{i}}^{\ast})\big) \geq q_{\mathbf{w}_{i}} \omega_{\mathbf{w}_{i}} > 0$, where $\omega := (\omega_{w}(a_{w}^{\ast} a_{w}))_{w \in V\Gamma}$.

Consider the sequence of states
\begin{equation}
\left( \vartheta(\pi(a_{\mathbf{w}_{i}} a_{\mathbf{w}_{i}}^{\ast}))^{-1} \, \vartheta\big(\pi(a_{\mathbf{w}_{i}}) (\cdot) \pi(a_{\mathbf{w}_{i}}^{\ast})\big) \right)_{i \in \mathbb{N}}. \label{eq:StateSequence}
\end{equation}
Using Lemma \ref{ConjugationLemma}, we obtain for all $i \in \mathbb{N}$:
\[
\left| \frac{\vartheta(\pi(a_{\mathbf{w}_{i}}) \widetilde{Q}_{v} \pi(a_{\mathbf{w}_{i}}^{\ast}))}{\vartheta(\pi(a_{\mathbf{w}_{i}} a_{\mathbf{w}_{i}}^{\ast}))} - 1 \right| 
= \left| \frac{\psi(a_{\mathbf{w}_{i}} Q_{v}^{\perp} a_{\mathbf{w}_{i}}^{\ast})}{\psi(a_{\mathbf{w}_{i}} a_{\mathbf{w}_{i}}^{\ast})} \right| 
\leq q_{\mathbf{w}_{i}}^{-1} |\psi(Q_{\mathbf{w}_{i}})| \to 0.
\]

By weak$^{\ast}$-compactness, the sequence in \eqref{eq:StateSequence} admits a subnet converging to a state $\phi$ with $\phi(\widetilde{Q}_{v}) = 1$. This completes the proof.
\end{proof}

\begin{lemma} \label{ElementExpression} 
Let $\Gamma$ be a finite, undirected, simplicial graph whose complement $\Gamma^{c}$ is connected, and let $\mathbf{A} := (A_{v})_{v \in V\Gamma}$ be a collection of unital C$^{\ast}$-algebras, each equipped with a GNS-faithful state $\omega_v$. Let $v \in V\Gamma$ be a vertex, and let $(u_{1}, \ldots, u_{m}) \in V\Gamma \times \cdots \times V\Gamma$ be a closed walk in $\Gamma^{c}$ with $v = u_1$ and covering the whole graph. Suppose that $a_1 \in A_{u_1}^{\circ}, \ldots, a_m \in A_{u_m}^{\circ}$ are non-zero elements, and define the group element $\mathbf{g} := u_1 \cdots u_m \in W_{\Gamma}$ corresponding to the walk. Set $a := (a_i)_{1 \leq i \leq m} \in A_{u_1}^{\circ} \times \cdots \times A_{u_m}^{\circ}$.

Then for any choice of vertices $v_1, \ldots, v_n \in V\Gamma$ and elements $x_1, \ldots, x_n \in \mathbf{A}_{\Gamma}^{\mathrm{alg}}$, there exists an integer $j_0 \in \mathbb{N}$ such that for every $j \geq j_0$, the element
\begin{equation}
\left(Q_{v_1}x_1Q_{v_2}x_2 \cdots Q_{v_n}x_n\right)a_{\mathbf{g}^j} Q_v \in \mathfrak{A}(\mathbf{A}, \Gamma) \label{BuildingBlocks}
\end{equation}
can be written in the form $y Q_v$ for some $y \in \mathbf{A}_{\Gamma}^{\mathrm{alg}}$.
\end{lemma}

\begin{proof}
We proceed by induction on $n$.

For the base case $n=1$, the element in~\eqref{BuildingBlocks} has the form $Q_{v_1}x_1 a_{\mathbf{g}^j} Q_v$, where $x_1 \in \mathbf{A}_{\Gamma}^{\mathrm{alg}}$. Without loss of generality, we may assume that $x_1$ is a reduced word, say $x_1 = b_1 \cdots b_k$ for some $b_i \in A_{w_i}^{\circ}$ with $(w_1, \ldots, w_k) \in \mathcal{W}_{\mathrm{red}}$. Since $(u_1, \ldots, u_m)$ is a closed walk in $\Gamma^c$, it follows that for large enough $j_0$, the product $x_1 a_{\mathbf{g}^{j_0}}$ can be expressed as a finite sum
\[
x_1 a_{\mathbf{g}^{j_0}} = \sum_{\mathbf{v} \in W_\Gamma : \mathbf{g} \leq_L \mathbf{v}} y_{\mathbf{v}}
\]
with each $y_{\mathbf{v}} \in \mathbf{A}_{\Gamma}^{\mathrm{alg}}$ a reduced operator of type $\mathbf{v}$. Thus, for all $j \geq j_0$ and $\eta \in \mathcal{H}_{\mathbf{w}}^{\circ}$ with $v \leq \mathbf{w}$, we have
\begin{eqnarray}
\nonumber
Q_{v_1} x_1 a_{\mathbf{g}^j} Q_v \eta &=& Q_{v_1} (x_1 a_{\mathbf{g}^{j_0}}) a_{\mathbf{g}^{j - j_0}} \eta  = \sum_{\mathbf{v} : \mathbf{g} \leq_L \mathbf{v}} Q_{v_1} y_{\mathbf{v}} a_{\mathbf{g}^{j - j_0}} \eta \\
\nonumber
&=& \left( \sum_{\mathbf{v} : v_1 \leq \mathbf{v}, \, \mathbf{g} \leq_L \mathbf{v}} y_{\mathbf{v}} a_{\mathbf{g}^{j - j_0}} \right) Q_v \eta =  y Q_v \eta,
\end{eqnarray}
where $y := \sum_{\mathbf{v} : v_1 \leq \mathbf{v}, \, \mathbf{g} \leq_L \mathbf{v}} y_{\mathbf{v}} a_{\mathbf{g}^{j - j_0}}  \in \mathbf{A}_{\Gamma}^{\mathrm{alg}}$, as desired.

For the induction step assume the statement holds for all natural numbers strictly less than some $n \geq 2$. Consider the element $(Q_{v_1}x_1Q_{v_2}x_2 \cdots Q_{v_n}x_n)a_{\mathbf{g}^j} Q_v$. By the induction hypothesis, there exists $j_0 \in \mathbb{N}$ such that $Q_{v_2}x_2 \cdots Q_{v_n}x_n a_{\mathbf{g}^{j_0}} Q_v = y Q_v$ for some $y \in \mathbf{A}_{\Gamma}^{\mathrm{alg}}$. Since $a_{\mathbf{g}^{j - j_0}} Q_v = Q_v a_{\mathbf{g}^{j - j_0}} Q_v$, we compute:
\[
(Q_{v_1}x_1Q_{v_2}x_2 \cdots Q_{v_n}x_n)a_{\mathbf{g}^j} Q_v 
= Q_{v_1}x_1 (y Q_v) a_{\mathbf{g}^{j - j_0}} Q_v 
= (Q_{v_1}x_1 y) a_{\mathbf{g}^{j - j_0}} Q_v.
\]
Since $x_1 y \in \mathbf{A}_{\Gamma}^{\mathrm{alg}}$, the base case applies to the right-hand side, which shows that this element is of the desired form for all $j \geq j_0$. This concludes the proof.
\end{proof}

Note that finite sums of elements of the form appearing in the bracketed expression of~\eqref{BuildingBlocks} are norm-dense in the C$^{\ast}$-algebra $\mathfrak{A}(\mathbf{A}, \Gamma)$. In combination with the preceding lemma, this observation will allow us to derive a simplicity criterion for graph product C$^{\ast}$-algebras.

\begin{theorem} \label{SimplicityCriterion1}
Let $\Gamma$ be a finite, undirected, simplicial graph with $\# \Gamma \geq 3$, and let $\mathbf{A} := (A_{v})_{v \in V\Gamma}$ be a collection of unital C$^*$-algebras, each equipped with a GNS-faithful state $\omega_{v}$. Suppose that the complement $\Gamma^{c}$ is connected, and that for every vertex $v \in V\Gamma$ there exist elements $a_{v} \in \ker(\omega_{v})$, $q_{v} > 0$ such that $a_{v} a_{v}^{*} \geq q_{v} \omega_{v}(a_{v}^{*} a_{v}) 1 > 0$. Assume further that $(q_{v})_{v \in V\Gamma} \in \mathbb{R}_{>0}^{V\Gamma} \setminus \overline{\mathcal{R}(\Gamma)}$. Then the graph product C$^*$-algebra $\mathbf{A}_{\Gamma}$ is simple if and only if $\mathbf{A}_{\Gamma} \cap \mathfrak{I}(\mathbf{A}, \Gamma) = 0$.

Moreover, if $\mathbf{A}_{\Gamma}$ is simple, the canonical inclusion $\mathbf{A}_{\Gamma} \hookrightarrow \mathfrak{A}(\mathbf{A}, \Gamma)/\mathfrak{I}(\mathbf{A}, \Gamma)$ is \emph{C$^*$-irreducible} in the sense that every intermediate C$^*$-algebra is simple as well; see \cite[Definition 3.1]{Rordam23}.
\end{theorem}

\begin{proof}
The ``only if'' direction is immediate, since $\mathfrak{I}(\mathbf{A}, \Gamma)$ is a closed two-sided ideal in $\mathfrak{A}(\mathbf{A}, \Gamma)$.

For the converse, suppose $\mathbf{A}_{\Gamma} \cap \mathfrak{I}(\mathbf{A}, \Gamma) = 0$. Let $\pi: \mathbf{A}_{\Gamma} \to \mathfrak{A}(\mathbf{A}, \Gamma)/\mathfrak{I}(\mathbf{A}, \Gamma)$ be the quotient map, and consider an intermediate C$^*$-algebra $\mathcal{A}$ such that $\mathbf{A}_\Gamma \cong \pi(\mathbf{A}_{\Gamma}) \subseteq \mathcal{A} \subseteq \mathfrak{A}(\mathbf{A}, \Gamma)/\mathfrak{I}(\mathbf{A}, \Gamma)$, with $\mathcal{A}$ admitting a closed two-sided ideal $I$.

Let $(u_{1}, \ldots, u_{n}) \in V\Gamma \times \cdots \times V\Gamma$ be a closed walk in the complement $\Gamma^{c}$ covering the whole graph. By Proposition~\ref{StateExistence}, there exists a state $\phi$ on $\mathfrak{A}(\mathbf{A}, \Gamma)/\mathfrak{I}(\mathbf{A}, \Gamma)$ such that $\phi$ vanishes on $I$ and satisfies $\phi(\widetilde{Q}_{u_{1}}) = 1$.

Let $\mathbf{g} := u_{1} \cdots u_{n} \in W_{\Gamma}$ be the group element corresponding to the walk, and let $a := (a_{v})_{v \in V\Gamma}$. By assumption, for each $i \in \mathbb{N}$, $\phi(\pi(a_{\mathbf{g}^{i}} a_{\mathbf{g}^{i}}^{*})) \geq q_{\mathbf{g}^{i}} \omega_{\mathbf{g}^{i}} > 0$, where $\omega := (\omega_{v}(a_{v}^{*} a_{v}))_{v \in V\Gamma}$.

Consider the sequence of states:
\begin{equation}
\left( \phi(\pi(a_{\mathbf{g}^{i}} a_{\mathbf{g}^{i}}^{*}))^{-1} \, \phi\left( \pi(a_{\mathbf{g}^{i}}) (\,\cdot\,) \pi(a_{\mathbf{g}^{i}}^{*}) \right) \right)_{i \in \mathbb{N}}. \label{eq:StateSequence-1}
\end{equation}
By the Banach–Alaoglu theorem, this sequence admits a subnet converging to some state $\psi$.

We claim that $\psi$ vanishes on the closed two-sided ideal $J$ generated by $I$ in $\mathfrak{A}(\mathbf{A}, \Gamma)/\mathfrak{I}(\mathbf{A}, \Gamma)$. Indeed, let $x \in I $, and take any $x_{1}, \ldots, x_{k}, x_{1}', \ldots, x_{l}' \in \mathbf{A}_{\Gamma}^{\mathrm{alg}}$, and vertices $v_{1}, \ldots, v_{k}, v_{1}', \ldots, v_{l}' \in V\Gamma$. Since $(u_{n}, \ldots, u_{1})$ is also a closed walk in $\Gamma^{c}$ covering the whole graph, Lemma~\ref{ElementExpression} yields $j_{0} \in \mathbb{N}$ such that for all $j \geq j_{0}$, there exist $y_{j}, y_{j}' \in \mathbf{A}_{\Gamma}^{\mathrm{alg}}$ with
\[
Q_{u_{1}} a_{\mathbf{g}^{j}} (Q_{v_{1}} x_{1} Q_{v_{2}} x_{2} \cdots Q_{v_{k}} x_{k})^{*} = Q_{u_{1}} y_{j}^{*},
\quad
(Q_{v_{1}'} x_{1}' Q_{v_{2}'} x_{2}' \cdots Q_{v_{l}'} x_{l}') a_{\mathbf{g}^{i}}^{*} Q_{u_{1}} = y_{j}' Q_{u_{1}}.
\]
Since $\phi(\widetilde{Q}_{u_{1}}) = 1$, the projection $\widetilde{Q}_{u_{1}}$ lies in the multiplicative domain of $\phi$. Thus,
\begin{eqnarray}
\nonumber
& & \phi(\pi(a_{\mathbf{g}^{i}}a_{\mathbf{g}^{i}}^{\ast}))^{-1}\phi\left(\pi(a_{\mathbf{g}^{i}})\pi(Q_{v_{1}}x_{1}Q_{v_{2}}x_{2}\cdots Q_{v_{k}}x_{k})x\pi(Q_{v_{1}^{\prime}}x_{1}^{\prime}Q_{v_{2}^{\prime}}x_{2}^{\prime}\cdots Q_{v_{l}^{\prime}}x_{l}^{\prime})\pi(a_{\mathbf{g}^{i}}^{\ast})\right)\\
\nonumber
&=& \phi(\pi(a_{\mathbf{g}^{i}}a_{\mathbf{g}^{i}}^{\ast}))^{-1}\phi\left(\pi\left(Q_{u_{1}}a_{\mathbf{g}^{i}}(Q_{v_{1}}x_{1}Q_{v_{2}}x_{2}\cdots Q_{v_{k}}x_{k})^{\ast}\right)x\pi\left((Q_{v_{1}^{\prime}}x_{1}^{\prime}Q_{v_{2}^{\prime}}x_{2}^{\prime}\cdots Q_{v_{l}^{\prime}}x_{l}^{\prime})a_{\mathbf{g}^{i}}^{\ast}Q_{u_{1}}\right)\right) \\
\nonumber
&=& \phi(\pi(a_{\mathbf{g}^{i}}a_{\mathbf{g}^{i}}^{\ast}))^{-1}\phi\left(\pi(Q_{u_{1}}y_{j}^{\ast})x\pi(y_{j}^{\prime}Q_{u_{1}})\right) \\
\nonumber
&=& \phi(a_{\mathbf{g}^{i}}a_{\mathbf{g}^{i}}^{\ast})^{-1}\phi\left(\pi(y_{j}^{\ast})x\pi(y_{j}^{\prime})\right) \\
\nonumber
&=& 0
\end{eqnarray}
for all $j \geq j_{0}$. It follows that
\[
\psi\left( \pi(Q_{v_{1}} x_{1} \cdots Q_{v_{k}} x_{k})^{*} x \pi(Q_{v_{1}'} x_{1}' \cdots Q_{v_{l}'} x_{l}')^{*} \right) = 0.
\]
Since sums of elements of the form $(Q_{v_{1}}x_{1}Q_{v_{2}}x_{2}\cdots Q_{v_{k}}x_{k})^{\ast}$ and $Q_{v_{1}^{\prime}}x_{1}^{\prime}Q_{v_{2}^{\prime}}x_{2}^{\prime}\cdots Q_{v_{l}^{\prime}}x_{l}^{\prime}$ as above are dense in $\mathfrak{A}(\mathbf{A},\Gamma)$, we conclude that $\psi$ vanishes on $J$.

As $\psi \neq 0$, we have $J \neq \mathfrak{A}(\mathbf{A}, \Gamma)/\mathfrak{I}(\mathbf{A}, \Gamma)$. Applying Theorem~\ref{MaximalityTheorem}, we deduce that $J = 0$, and thus $I = 0$. Hence $\mathcal{A}$ is simple, as desired.
\end{proof}

\begin{remark}
Let $\Gamma$ be a finite, undirected, simplicial graph, and let $\mathbf{A} := (A_v)_{v \in V\Gamma}$ be a collection of unital C$^{*}$-algebras, each equipped with a GNS-faithful state $\omega_v$. When $\#\Gamma = 2$, the graph product $\mathbf{A}_\Gamma$ reduces to either a reduced free product of the vertex algebras or a minimal tensor product. The tensor product case has been discussed earlier, while results on the free product setting -- complementary to Theorem~\ref{SimplicityCriterion1} -- can be found in \cite{McClanahan94, Avitzour82, Dykema99}.
\end{remark}

The proof of the following corollary is inspired by \cite[Theorem 4.3]{Klisse23-1}.

\begin{corollary} \label{SimplicityCriterion2} 
Let $\Gamma$ be a finite, undirected simplicial graph with $\# \Gamma \geq 3$, and let $\mathbf{A} := (A_v)_{v \in V\Gamma}$ be a collection of finite-dimensional C$^{\ast}$-algebras equipped with faithful states $(\omega_v)_{v \in V\Gamma}$.  Assume that the complement $\Gamma^c$ is connected, and that for each vertex $v \in V\Gamma$, there exist elements $a_v \in \ker(\omega_v)$, $q_v > 0$ such that $a_v a_v^* \geq q_v \omega_v(a_v^* a_v)1 > 0$. Suppose further that $(q_v)_{v \in V\Gamma} \in \mathbb{R}_{>0}^{V\Gamma} \setminus \overline{\mathcal{R}(\Gamma)}$.  Then the graph product C$^{\ast}$-algebra $\mathbf{A}_\Gamma$ is simple, and the inclusion $\mathbf{A}_\Gamma \hookrightarrow \mathfrak{A}(\mathbf{A}, \Gamma) / \mathfrak{I}(\mathbf{A}, \Gamma)$ is C$^{\ast}$-irreducible.
\end{corollary}

\begin{proof}
By Proposition \ref{FiniteDimensionalIdeal}, the ideal $\mathfrak{I}(\mathbf{A}, \Gamma)$ coincides with the ideal of compact operators $\mathcal{K}(\mathcal{H}_\Gamma)$ on $\mathcal{H}_\Gamma$. By Theorem~\ref{SimplicityCriterion1}, it remains to show that $\mathbf{A}_\Gamma \cap \mathcal{K}(\mathcal{H}_\Gamma) = 0$.

Each $A_v$ embeds into $\mathcal{B}(\mathcal{H}_v)$ and is closed under the strong operator topology; thus, it is a von Neumann algebra. Following the construction in~\cite[Subsection 2.3]{CaspersFima17}, let $\mathcal{M}$ denote the graph product von Neumann algebra generated by $\bigcup_{v \in V\Gamma} \lambda_v(A_v)$ inside $\mathcal{B}(\mathcal{H}_\Gamma)$. By the same reasoning as in~\cite[Subsection 2.3]{CaspersFima17}, the vacuum vector $\Omega$ is cyclic and separating for $\mathcal{M}$.

Let $J$ denote the modular conjugation associated to $\omega_\Gamma$. Then we have
\[
J \mathbf{A}_\Gamma J \subseteq J \mathcal{M} J = \mathcal{M}'.
\]
Suppose, toward a contradiction, that $\mathbf{A}_\Gamma \cap \mathcal{K}(\mathcal{H}_\Gamma) \neq 0$. Then $\mathcal{M}'$ contains a non-zero compact operator and its spectral projections, and hence also a non-zero finite-rank projection $P$ commuting with all elements of $\mathbf{A}_\Gamma$.

Let $(e_i)_{1\leq i \leq n} \subseteq P \mathcal{H}_\Gamma$ be an orthonormal basis of $P \mathcal{H}_\Gamma$. Define $a := (a_v)_{v \in V\Gamma}$, and observe that by assumption, $a_{\mathbf{w}} a_{\mathbf{w}}^* \geq q_{\mathbf{w}} \omega_{\mathbf{w}} > 0$, where $\omega:=\left(\omega_{v}(a_{v}^{\ast}a_{v})\right)_{v\in V\Gamma}$. Hence,
\[
\Vert P\Omega\Vert^{2}\leq(q_{\mathbf{w}}\omega_{\mathbf{w}})^{-1}\Vert Pa_{\mathbf{w}}^{\ast}\Omega\Vert^{2}=\sum_{i=1}^{n}(q_{\mathbf{w}}\omega_{\mathbf{w}})^{-1}\left|\langle e_{i},Pa_{\mathbf{w}}^{\ast}\Omega\rangle\right|^{2}=\sum_{i=1}^{n}(q_{\mathbf{w}}\omega_{\mathbf{w}})^{-1}\left|\langle e_{i},a_{\mathbf{w}}^{\ast}\Omega\rangle\right|^{2}.
\]

We now distinguish two cases:

\begin{itemize}
\item \textit{Case 1:} Suppose there exists a constant $C > 0$ such that for every $\mathbf{w} \in W_\Gamma$, there exists $1 \leq i \leq n$ with 
\[
(q_{\mathbf{w}} \omega_{\mathbf{w}})^{-1} \left| \langle e_i, a_{\mathbf{w}}^* \Omega \rangle \right|^2 > C.
\]
Since the vectors $a_{\mathbf{w}}^* \Omega$ are pairwise orthogonal, we compute
\begin{eqnarray}
\nonumber
\sum_{i=1}^{n}\Vert e_{i}\Vert^{2} &\geq& \sum_{i=1}^{n}\sum_{\mathbf{w}\in W_{\Gamma}}\Vert a_{\mathbf{w}}\Omega\Vert^{-1}\left|\left\langle e_{i},a_{\mathbf{w}}^{\ast}\Omega\right\rangle \right|^{2} \\
\nonumber
&=& \sum_{i=1}^{n}\sum_{\mathbf{w}\in W_{\Gamma}}\omega_{\mathbf{w}}^{-1}\left|\left\langle e_{i},a_{\mathbf{w}}^{\ast}\Omega\right\rangle \right|^{2} \\
\nonumber
&>& C\sum_{\mathbf{w}\in W_{\Gamma}}q_{\mathbf{w}}.
\end{eqnarray}
Since $(q_v)_{v \in V\Gamma} \notin \overline{\mathcal{R}(\Gamma)}$, the sum on the right diverges, while the left-hand side is finite -- yielding a contradiction.

\item \textit{Case 2:} Suppose now that there exists a sequence $(\mathbf{w}_j)_{j \in \mathbb{N}} \subseteq W_\Gamma$ such that for each $1 \leq i \leq n$,
\[
(q_{\mathbf{w}_j} \omega_{\mathbf{w}_j})^{-1} | \langle e_i, a_{\mathbf{w}_j}^* \Omega \rangle |^2 \to 0 \quad \text{as } j \to \infty.
\]
Then,
\[
\|P\Omega\|^2 \leq \sum_{i=1}^n (q_{\mathbf{w}_j} \omega_{\mathbf{w}_j})^{-1} | \langle e_i, a_{\mathbf{w}_j}^* \Omega \rangle |^2 \to 0,
\]
implying $P\Omega = 0$. But since $\Omega$ is separating and cyclic for $\mathcal{M}$, it is also separating for $\mathcal{M}'$, so this forces $P = 0$, a contradiction.
\end{itemize}

In either case, we reach a contradiction. Therefore, $\mathbf{A}_\Gamma \cap \mathcal{K}(\mathcal{H}_\Gamma) = 0$, completing the proof.
\end{proof}

\begin{corollary} \label{SimplicityCriterion3}
Let $\Gamma$ be a finite, undirected, simplicial graph with $\#V\Gamma \geq 3$, and let $\mathbf{A} := (A_{v})_{v \in V\Gamma}$ be a collection of unital C$^{\ast}$-algebras, equipped with GNS-faithful states $(\omega_{v})_{v \in V\Gamma}$. Suppose that the complement $\Gamma^{c}$ is connected, and that for every vertex $v \in V\Gamma$, there exists a unitary $u_{v} \in \ker(\omega_{v})$ with $\omega_v(u_v x)=\omega_v(x u_v)$ for all $x \in A_v$. Then the graph product C$^{\ast}$-algebra $\mathbf{A}_{\Gamma}$ is simple, and the canonical inclusion $\mathbf{A}_{\Gamma} \hookrightarrow \mathfrak{A}(\mathbf{A}, \Gamma)/\mathfrak{I}(\mathbf{A}, \Gamma)$ is C$^{\ast}$-irreducible.
\end{corollary}

\begin{proof}
By Theorem \ref{SimplicityCriterion1}, it suffices to show that $\mathbf{A}_{\Gamma} \cap \mathcal{K}(\mathcal{H}_{\Gamma}) = 0$. To this end, let $x \in \mathbf{A}_{\Gamma} \cap \mathcal{K}(\mathcal{H}_{\Gamma})$. Since each $u_{v}$ lies in the centralizer $A_{v}^{\omega_{v}}$, it is also contained in the centralizer of $\omega_{\Gamma}$, viewed as a state on $\mathbf{A}_{\Gamma}$. Consequently, for every reduced word $\mathbf{w} \in W_{\Gamma}$ of length $|\mathbf{w}| > N$, we have
\[
\|x\Omega\|^{2} = \omega_{\Gamma}(x^{\ast}x) = \omega_{\Gamma}(u_{\mathbf{w}}^{\ast} x^{\ast}x u_{\mathbf{w}}) = \|x(u_{\mathbf{w}} \Omega)\|^{2} \leq \|x P_{N}^{\perp}\|^{2},
\]
where $ u:=(u_{w})_{w\in V\Gamma}$.  It follows that $\Vert x\Omega\Vert=0$ and hence $x=0$. Therefore, $\mathbf{A}_{\Gamma} \cap \mathcal{K}(\mathcal{H}_{\Gamma}) = 0$, which completes the proof.
\end{proof}

\vspace{3mm}


\subsection{Trace-Uniqueness}

In parallel with our analysis of the simplicity of graph product C$^{\ast}$-algebras, we now investigate the uniqueness of tracial states under the assumption that all vertex states are tracial. Related results for free product C$^{\ast}$-algebras can be found in \cite{Dykema99}. As in Subsection \ref{GraphSimplicity}, it suffices to restrict attention to graphs whose complements are connected.

The proof of the following proposition follows the same general strategy as that of Proposition~\ref{StateExistence}.

\begin{proposition} \label{StateExistence2}
Let $\Gamma$ be a finite, undirected, simplicial graph with $\#V\Gamma \geq 3$ such that its complement $\Gamma^{c}$ is connected, and let $\mathbf{A} := (A_{v})_{v \in V\Gamma}$ be a collection of unital C$^{\ast}$-algebras, equipped with GNS-faithful states $(\omega_{v})_{v \in V\Gamma}$. Suppose that every vertex $v \in V\Gamma$ admits a unitary $u_{v} \in \ker(\omega_{v})$. Then for every tracial state $\tau$ on $\mathbf{A}_{\Gamma}$ and every $v \in V\Gamma$, there exists a state $\phi$ on $\mathcal{B}(\mathcal{H}_{\Gamma})$ satisfying $\phi(Q_{v}) = 1$ and whose restriction to $\mathbf{A}_{\Gamma}$ coincides with $\tau$.
\end{proposition}

\begin{proof}
Extend $\tau$ to a state $\psi$ on $\mathcal{B}(\mathcal{H}_{\Gamma})$, and let $(v_{1}, \dots, v_{n}) \in V\Gamma \times \cdots \times V\Gamma$ be a closed walk in $\Gamma^{c}$ covering the whole graph, with $v_{1} = v$. Since $\#V\Gamma \geq 3$ and $\Gamma^{c}$ is connected, it follows that $W_{\Gamma}$ is non-amenable, and hence $1 \in \mathbb{R}_{>0}^{V\Gamma} \setminus \overline{\mathcal{R}(\Gamma)}$ by \cite[Proposition~17.2.1]{Davis08}.

Applying Proposition~\ref{MainProposition}, we obtain a sequence $(\mathbf{w}_{i})_{i \in \mathbb{N}} \subseteq W_{\Gamma}$ with increasing word length such that $v_{1} \cdots v_{n} \leq \mathbf{w}_{i}^{-1}$ for all $i$, and $\psi(Q_{\mathbf{w}_{i}}) \to 0$.

Let $u := (u_{w})_{w \in V\Gamma}$ and consider the sequence of states $( \psi ( u_{\mathbf{w}_{i}} (\cdot) u_{\mathbf{w}_{i}}^{\ast} ) )_{i \in \mathbb{N}}$. By Lemma~\ref{ConjugationLemma}, for each $i \in \mathbb{N}$ we have
\[
| \psi(u_{\mathbf{w}_{i}} Q_{v} u_{\mathbf{w}_{i}}^{\ast}) - 1 | = | \psi(u_{\mathbf{w}_{i}} Q_{v}^{\perp} u_{\mathbf{w}_{i}}^{\ast}) | \leq \psi(Q_{\mathbf{w}_{i}}) \to 0.
\]
Then, by the weak$^{\ast}$-compactness of the state space of $\mathcal{B}(\mathcal{H}_{\Gamma})$, there exists a subnet converging to a state $\phi$ such that $\phi(Q_{v}) = 1$. By construction, we have $\phi|_{\mathbf{A}_{\Gamma}} = \tau$.
\end{proof}

\begin{theorem} \label{TraceUniqueness}
Let $\Gamma$ be a finite, undirected, simplicial graph, and let $\mathbf{A} := (A_{v})_{v \in V\Gamma}$ be a collection of unital C$^{\ast}$-algebras, each equipped with a GNS-faithful state $\omega_{v}$. Suppose that for every $v \in V\Gamma$ there exists a unitary $u_{v} \in \ker(\omega_{v})$. Then:
\begin{enumerate}
\item If the states $(\omega_{v})_{v\in V\Gamma}$ are tracial, then $\omega_{\Gamma}$ is the unique tracial state on $\mathbf{A}_{\Gamma}$.
\item If at least one $\omega_{v}$ is non-tracial, then $\mathbf{A}_{\Gamma}$ admits no tracial state.
\end{enumerate}
\end{theorem}

\begin{proof}
Let $\tau$ be a tracial state on $\mathbf{A}_{\Gamma}$, and consider an element $x = a_{1} \cdots a_{m} \in \mathbf{A}_{\Gamma}$, where $m \geq 1$, $a_{i} \in A_{u_{i}}^{\circ}$, and $(u_{1}, \dots, u_{m}) \in \mathcal{W}_{\mathrm{red}}$ is a reduced word with associated product $\mathbf{u} := u_{1} \cdots u_{m}$.

Choose a closed walk $(v_{1}, \dots, v_{n})$ in $\Gamma^{c}$ that covers the whole graph and satisfies $v_{n} = u_{m}$, and set $\mathbf{g} := v_{1} \cdots v_{n}$. By Proposition~\ref{StateExistence2}, there exists a state $\phi$ on $\mathcal{B}(\mathcal{H}_{\Gamma})$ with $\phi(Q_{v_{1}}) = 1$ and $\phi|_{\mathbf{A}_{\Gamma}} = \tau$. In particular, $Q_{v_{1}}$ lies in the multiplicative domain of $\phi$.

For $u := (u_{v})_{v \in V\Gamma}$ and $i \in \mathbb{N}$, we compute
\[
\tau(x) = \tau(u_{\mathbf{g}^{i}}^{\ast} x u_{\mathbf{g}^{i}}) = \phi(u_{\mathbf{g}^{i}}^{\ast} x u_{\mathbf{g}^{i}}) = \phi(Q_{v_{1}} u_{\mathbf{g}^{i}}^{\ast} x u_{\mathbf{g}^{i}} Q_{v_{1}}) = \phi(u_{\mathbf{g}^{i}}^{\ast} (Q_{\mathbf{g}^{i}} x Q_{\mathbf{g}^{i}}) u_{\mathbf{g}^{i}}).
\]
For each $\mathbf{w} \in W_{\Gamma}$, by the choice of $\mathbf{g}$ we have $Q_{\mathbf{g}^{i}}xQ_{\mathbf{g}^{i}}\mathcal{H}_{\mathbf{w}}^{\circ}\neq0$ if and only if $\mathbf{g}^{i}\leq\mathbf{w}$ and $\mathbf{g}^{i}\leq\mathbf{u}\mathbf{w}$, where $|\mathbf{u}\mathbf{w}|=|\mathbf{u}|+|\mathbf{w}|$. This implies that $Q_{\mathbf{g}^{i}} x Q_{\mathbf{g}^{i}} = 0$ for sufficiently large $i$, unless $\mathbf{u} \in C_{W_{\Gamma}}(\mathbf{g})$, the centralizer of $\mathbf{g}$. However, the construction of $\mathbf{g}$ ensures that $\mathbf{u} \notin C_{W_{\Gamma}}(\mathbf{g})$, and so $\tau(x) = 0$.

Since $x$ was arbitrary, we conclude that $\tau = \omega_{\Gamma}$. The fact that $\omega_{\Gamma}$ is tracial if and only if all vertex states $\omega_{v}$ are tracial follows easily from the construction of the graph product state.
\end{proof}

\vspace{3mm}



\end{document}